\documentclass{article}
\usepackage[utf8]{inputenc}

\usepackage{amsmath,amssymb,amsfonts}
\usepackage{dsfont}
\usepackage{mathrsfs}
\usepackage{enumitem}
\usepackage{relsize}
\usepackage{comment}

\usepackage{amsthm}

	\newtheorem{theorem}{Theorem}

	\newtheorem{proposition}[theorem]{Proposition}

	\newtheorem{remark}[theorem]{Remark}
	
	\newtheorem{lemma}[theorem]{Lemma}
	
	\newtheorem{corollary}[theorem]{Corollary}
    	
	\newtheorem{conjecture}[theorem]{Conjecture}

    \newtheorem{problem}[theorem]{Problem}

	\numberwithin{theorem}{section}

	\def\R{{\mathbb{R}}}
	\def\Z{{\mathbb{Z}}}
	\def\N{{\mathbb{N}}}

	\def\E{{\mathbb{E}}}
	\def\P{{\mathbb{P}}}
	\def\1{\mathds{1}}
	\def\Var{\textup{Var}}
	\def\omg{\omega}
	
	\DeclareMathOperator{\Cov}{Cov}
	\DeclareMathOperator{\Corr}{\mathrm{Corr}}
	\DeclareMathOperator{\I}{I}

	\def\e{\mathbf{e}}

	\def\u{\mathbf{u}}

	\def\argmax{\mathrm{argmax}}

	\def\eps{\varepsilon}

	\def\p{\partial}

\usepackage[backend = biber, style = alphabetic, sorting = nyt, maxbibnames=99]{biblatex}
\usepackage{graphicx}
\usepackage{xcolor}

\title{Noise sensitivity in last-passage percolation}
\author{
Daniel Ahlberg\thanks{Department of Mathematics, Stockholm University, 
\texttt{daniel.ahlberg@math.su.se}}
\and
Malo Hillairet\thanks{Institut Fourier, Universit\'e Grenoble Alpes,
\texttt{malo.hillairet@univ-grenoble-alpes.fr}}
\and
Ekaterina Toropova\thanks{Department of Mathematics, Stockholm University, 
\texttt{ekaterina.toropova@math.su.se}}
}
\date{January 8, 2026}

\begin{document}

\maketitle

\begin{abstract}
The study of noise sensitivity of Boolean functions was initiated in a seminal paper of Benjamini, Kalai and Schramm, published in 1999. While this study has revealed fascinating phenomena in the context of Bernoulli percolation, few results have been obtained regarding other random spatial processes. In this paper we prove the first instance of noise sensitivity for a spatial growth process associated to the Kardar-Parisi-Zhang class of universality. More specifically, we show that travel times in geometric last-passage percolation are noise sensitive with respect to a perturbation acting on a Bernoulli encoding of the geometric weights. Our method of proof includes a generalisation of the celebrated Benjamini-Kalai-Schramm noise sensitivity/influence theorem, and precise bounds on the probability of a given vertex being on a geodesic, which we believe to be of independent interest.\\


\noindent
{\em Funding:} Research in part supported by the Swedish Research Council through grant 2021-03964.
\end{abstract}

\section{Introduction}

In the study of random spatial processes and structures there is an abundance of models with a simple probabilistic formulation at a microscopic level that give rise to highly complex behaviour at macroscopic scales. It is believed that such models to a large extent can be divided into classes, referred to as universality classes, where all models in one class exhibit a similar large-scale behaviour, although their local description differs. This large scale behaviour is characterised by deterministic quantities such as power-law exponents describing the scaling of different features at the critical point. Two prominent examples of universality classes are those of planar Bernoulli percolation on the triangular lattice, described by work of Schramm~\cite{schramm00} and Smirnov~\cite{smi01}, and the KPZ universality class for $1+1$ dimensional growing interfaces, introduced in work of Kardar, Parisi and Zhang~\cite{karparzha86}, and made rigorous in work of Baik, Deift and Johansson~\cite{baideijoh99,joh00}.

%Percolation theory aims to describe how fluid propagates through intricate structures or environments. In mathematics, percolation models often involve a simple probabilistic description at the microscopic level, through iid random variables, for example, and highly non-trivial behavior arises at macroscopic scales. It is believed percolation models are divided into very large classes, called universality classes, where in each class all of the models have similar large-scale behavior although their probabilistic descriptions may differ. This large-scale behavior is characterized by deterministic quantities such as exponents in a power-law. Major examples of universality classes, associated to many predictions from the physics literature, involve the class of planar Bernoulli percolation on the triangular lattice, characterized by the Cardy-Smirnov formula~\cite{car92, smi01}, and the KPZ universality class,  introduced by Kardar, Parisi and Zhang~\cite{karparzha86}.

The most classic model associated to the KPZ universality class was introduced in the 1960s in work of Eden \cite{eden61} and Hammersley and Welsh~\cite{hamwel65}, and is referred to as first-passage percolation. In its most simple setting, the edges of the $\Z^2$ nearest-neighbour lattice are equipped with i.i.d.\ non-negative random weights, inducing a random metric on $\Z^2$.
%It is a random metric space where distances correspond to geodesic distances on a graph with randomized edge lengths.
A central objective is to understand the asymptotic behaviour of distances, balls and geodesics in this random metric space. 
%The observables in this model are distances, or travel-times.
Despite its simple formulation, there is no rigorous mathematical analysis proving that first-passage percolation satisfies the predictions of KPZ universality. 
However, there are closely related models, whose formulation is amenable to favourable calculations, that tie them to the KPZ class. 
One such model is last-passage percolation, in which non-negative random weights are assigned to the vertices instead of edges, and where weight is maximised over directed paths as opposed to minimised over all paths.
%resulting in a model with metric-like properties.
In the case of exponentially or geometrically distributed weights, this models has serendipitous connections to random matrix models via combinatorial identities, allowing for a precise analysis of fluctuations of distances and geodesic, in agreement with the predictions of KPZ universality~\cite{joh00,baideimclmilzho01}.

%For instance, the distance $T(u, v)$ from $u$ to $v$ is expected to have variance of order $|u - v|^{2/3}$. The corresponding optimal path, called geodesic, is expected to deviate from the straight line by $|u - v|^{2/3}$. The first rigorous analysis of an exactly solvable model leading to such exponents is due to Baik, Deift and Johansson \cite{baideijoh99}, and simpler proofs relying on more probabilistic arguments were found later \cite{balcatsep06} \textcolor{red}{$+$ Cator and Groenenboom?}.

In parallel, the exploration of noise sensitivity, introduced in seminal work of Benjamini, Kalai and Schramm~\cite{benkalsch99},
%and H\"aggstr\"om, Peres and Steif~\cite{hagperste97}
has highlighted a remarkable phenomenon in Bernoulli percolation.
In this context, noise sensitivity refers to the surprising property that a small random perturbation of a critical percolation configuration is sufficient in order to completely decorrelate its connectivity properties. More precisely, critical Bernoulli bond percolation on $\Z^2$ is noise sensitive in the sense that if $\omega$ is a configuration of 0s and 1s encoding the states of the edges (i.e.\ open or closed), the perturbation $\omega^\eps$ is obtained from $\omega$ by independently resampling each entry of $\omega$ with probability $\eps$, and $f_n$ is the function encoding the existence of a horizontal crossing of open edges of an $n\times n$ square, then 
\begin{equation}\label{def:bksns}
\E[f_n(\omega)f_n(\omega^\eps)]-\E[f_n(\omega)]^2\to0\quad\text{as }n\to\infty.
\end{equation}

Noise sensitivity is straightforwardly extended to the framework of Boolean functions, that is, functions of the form $f:\{0, 1\}^n\to\{0, 1\}$. Again we sample an element $\omega\in\{0, 1\}^n$ according to (a possibly biased) product measure, and obtain a perturbation $\omega^\eps$ of $\omega$ by independently resampling each coordinate with probability $\eps$. A sequence $(f_n)_{n\ge1}$ of Boolean functions is said to be {\bf noise sensitive} if it satisfies~\eqref{def:bksns}. The mathematical techniques for the study of noise sensitivity have been extensively developed in the context of Boolean functions~\cite{garste15, odonnell}. These methods include the study of influences, Fourier decomposition and randomized algorithms. These methods have led to a precise description, in planar Bernoulli percolation, of the transition from a stable regime, where the noise is too weak to alter percolation crossings, to a noise sensitive regime \cite{schste10,garpetsch10,tasvan23}. For Bernoulli site percolation on the triangular lattice, where critical exponents are known explicitly, Garban, Pete and Schramm~\cite{garpetsch10} showed that the transition occurs for the noise strength $\eps$ scaling with $n$ as $\eps\asymp n^{-3/4}$.

In related work, Schramm and Smirnov~\cite{schsmi11} proved that the scaling limit of Bernoulli percolation is a black noise -- a notion introduced by Tsirelson~\cite{tsi04} -- referring to the property that observables cannot be expressed through linear functions or averages, which are the most stable to noise. Thus, being a black noise established a form of noise sensitive behaviour for the scaling limit, although it is not known if noise sensitivity of a percolation model implies that its scaling limit is a black noise, or vice versa.
Recently, Himwich and Parekh~\cite{himpar25} showed that also the directed landscape is a black noise. The directed landscape was introduced by Dauvergne, Ortmann and Virag~\cite{dauortvir22} as the supposed scaling limit of models in the KPZ class of universality. This indicates that models in the KPZ class can be expected to be noise sensitive.

To this point, noise sensitivity has not been established for a single model associated to KPZ universality. In fact, the only prior evidence for noise sensitivity in the context of spatial growth appears to be work of Ahlberg and de la Riva~\cite{ahlriv26}, elaborated upon by Dembin and Elboim~\cite{demelb25}. Their work shows that `being above the median' is a noise sensitive property in a simplified model of first-passage percolation exhibiting Gaussian fluctuations. One difficulty in studying noise sensitivity in the context of spatial growth is that central observables are not Boolean functions, which limits the existing techniques. Another difficulty is that large regions need to be explored in order to identify the weight maximising path between two points, rendering exploration-based algorithm methods inadequate.

%However, there are for now very little results in that direction. Major difficulties are that travel-times are not Boolean functions, and there is little hope of algorithm methods being conclusive. Pioneering the study of noise sensitivity in first passage percolation, Ahlberg and de la Riva \cite{ahlriv23}, followed by Dembin and Elboim \cite{demelb25} proved noise sensitivity of the indicator function of being above the median, for a version of travel-times with sufficient symmetries. It is not known however if the travel-time from $0$ to $n \e_1$ is noise sensitive in first passage percolation. In this article, we propose to extend the definition of noise sensitivity to functions from $\{0, 1\}^n$ to the set of real numbers $\R$, and we prove noise sensitivity in that sense of last-passage percolation travel-times with geometric weights.

In this paper we establish noise sensitivity for a model known to belong to the KPZ class. More precisely, we prove that planar last-passage percolation with geometric weights exhibits noise sensitivity, via the method of influences. In order to accommodate the method to the present context, we extend the BKS influence theorem from Boolean to real-valued functions, and present precise bounds on the probability that a vertex is visited by the weight-maximising path between two given points.

\subsection{Main result}

We consider planar last-passage percolation with geometrically distributed weights, which we describe next. To the integer points of $\Z^2$ we assign independent geometrically distributed weights $\omega=(\omega_x)_{x\in\Z^2}$ of parameter $p \in (0, 1)$.\footnote{We work with the convention that $\N=\{0,1,2,\ldots\}$ and that $X$ is geometrically distributed with parameter $p\in(0,1)$ if $\P(X = k) = p(1 - p)^k$ for $k\in\N$.}
We introduce a partial order on $\Z^2$ and write $u\le v$ for $u=(u_1,u_2)$ and $v=(v_1,v_2)$ if $u_1\le v_1$ and $u_2\le v_2$. Let $\e_1:=(1,0)$ and $\e_2:=(0,1)$ denote the coordinate vectors. Given $u\le v$, a {\bf directed path} $\gamma$ from $u$ to $v$ is an upwards-rightwards sequence of vertices $x_0=u,x_1,\ldots,x_n=v$ where $x_i-x_{i-1}\in\{\e_1,\e_2\}$. For $u\le v$, we define the {\bf last-passage time}, also referred to as the {\bf travel time}, from $u$ to $v$ as
\begin{equation}\label{def:T}
T(u,v):=\max\big\{T(\gamma):\gamma\text{ is a directed path from $u$ to $v$}\big\},\quad T(\gamma):=\sum_{x\in\gamma}\omega_x.
\end{equation}

We refer to a path attaining the maximum in the definition of $T(u, v)$ as a \textbf{geodesic} from $u$ to $v$. Since the travel time is integer-valued, there may be multiple geodesics between two points, and we let $\pi(u, v)$ denote the set of points which belong to at least one of the geodesic from $u$ to $v$, that is,
\[
   \pi(u, v) := \big\{x \in \Z^2 : \exists \text{ a geodesic } \gamma \text{ from } u \text{ to } v \text{ such that } x \in \gamma\big\}.
\]
For ease of notation we occasionally write $T_n:=T(0,n\e_+)$ and $\pi_n := \pi(0, n \e_+)$, where $\e_+:=\e_1+\e_2=(1,1)$. In addition, we shall write $T_n(\omega)$ when we want to make explicit (as we often do) that $T_n$ is a function of the weight configuration $\omega$.

We next introduce a random perturbation of the weight configuration, and to do so we begin with an encoding of the geometric weights via Bernoulli random variables.
For each $v\in\Z^2$ and $i\in\N=\{0, 1,2,\ldots\}$, let $X_{v,i}$ be independent Bernoulli distributed random variables with parameter $p\in(0,1)$, and set
\begin{equation}\label{eq:encoding}
    \omega_v:=\min\{i\ge0:X_{v,i}=1\}.
\end{equation}
Hence, $(\omega_v)_{v\in\Z^2}$ are independent and geometrically distributed with parameter $p$.

We now introduce the perturbation. For $v\in\Z^2$ and $i\in\N$, let $X'_{v,i}$ be Bernoulli distributed with parameter $p$ and let $U_{v,i}$ be exponentially distributed with parameter $1$ (i.e.\ with density $t \mapsto e^{-t}$ with respect to the Lebesgue measure on $[0, \infty)$). We assume $X_{v, i}, X'_{v, i}, U_{v, i}$ to be independent of each other, and independent between different $v$ and $i$. For $t \geq 0$, let
\begin{equation*}\label{eqn: noise definition}
X_{v,i}^t:=\left\{
\begin{aligned}
    X_{v,i} && \text{if }t<U_{v,i};\\
    X'_{v,i} && \text{if }t\ge U_{v,i}.
\end{aligned}
\right.
\end{equation*}
Finally, we define $\omega^t$, for $v\in\Z^2$ and $t\ge0$, as
\begin{equation}\label{eq:bitresampling}
%    \omega_v&:=\min\{i\ge0:X_{v,i}=1\};\\
    \omega_v^t:=\min\{i\ge0:X^t_{v,i}=1\}.
\end{equation}
Note that for every $v\in\Z^2$ we have $\omega_v^0=\omega_v$, and that for $t\ge0$ we have that $\omega_v^t$ is geometrically distributed with parameter $p$, and that the variables $(X^t_{v,i})$ are obtained from $(X_{v,i})$ by resampling each bit with probability $\eps=1-e^{-t}$. For small $t>0$, one may hence think of $\omega^t$ as a small perturbation of $\omega$.

Our main theorem states that the sequence $(T_n)_{n\ge1}$ exhibits noise sensitivity, and provides the first instance of noise sensitivity in a model associated to the KPZ class of universality.

\begin{theorem} \label{th:nsenoftraveltimes}
For $n\ge1$, let $T_n = T_n(\omg)$ denote the last-passage time from $(0, 0)$ to $(n, n)$ with respect to the weight configuration $\omega$. 
Planar last-passage percolation with geometric weights is noise sensitive in the sense that, for every $p\in(0,1)$ and $t>0$,
%The sequence $(T_n)_{n \in \N}$ is noise sensitive
%with respect to the bit-resampling noise,
%in the sense that, for all $t > 0$,
\[
\lim_{n \rightarrow \infty} \Corr \big(T_n(\omg), T_n(\omg^t) \big) = 0 .
\]
\end{theorem}

We shall proceed below with an outline of the method of proof, and the properties of last-passage percolation needed for the method to work. We shall also comment upon the fact that when extending the notion of noise from a Boolean to a more general context, there are various ways in which this can be done; we show that two natural ways this can be done are roughly equivalent. We end the introduction with a brief discussion of related work on the concept of chaos, and why this work suggests that the transition from stability to noise sensitivity is expected to occur at $t\asymp n^{-1/3}$.
Our proof of Theorem~\ref{th:nsenoftraveltimes} shows that the conclusion of the theorem holds as long as $t\gg1/\log n$.

\subsection{Noise sensitivity and influences}

Noise sensitivity was introduced by Benjamini, Kalai and Schramm~\cite{benkalsch99} in the context of Boolean functions $f:\{0,1\}^n\to\{0,1\}$, but extends straightforwardly to square-integrable functions $f:\{0,1\}^\N\to\R$.
To define things properly, let $X=(X_i)_{i \in \N}$ be a family of independent Bernoulli random variables of parameter $p$, and let $X' = (X'_i)_{i \in \N}$ be an independent copy of $X$. Moreover, let $U_i$ be exponentially distributed with parameter $1$, and define, for $t \geq 0$, a perturbation $X^t$ of $X$ as
\begin{equation}\label{def:noise}
X^t_i := \left\{ \begin{aligned}
X_i &&& \text{if }t< U_i;\\
%\text{with probability } e^{-t}, \\
X'_i &&& \text{if }t\ge U_i.
%\text{with probability } 1 - e^{-t}.
\end{aligned}
\right.
\end{equation}
We say that a sequence $(f_n)_{n\ge1}$ of square-integrable functions $f_n:\{0,1\}^\N\to\R$ is {\bf noise sensitive} if for every $t>0$ we have
\begin{equation}\label{def:sensitivity}
\lim_{n \rightarrow \infty} \Corr \big(f_n(X), f_n(X^t) \big) = 0 .
\end{equation}
Note that the perturbed configuration $X^t$ is obtained from $X$ by resampling each bit independently with probability $\eps=1-e^{-t}$. We occasionally refer to the parameter $t$ as the noise strength. The reason for this parametrization is that $t$ will later be interpreted as time, as $(X^t)_{t \geq 0}$ is recast as a time-homogeneous Markov process under that parametrization. Note also that noise sensitivity as expressed in the context of last-passage percolation in Theorem~\ref{th:nsenoftraveltimes} is equivalent to noise sensitivity as defined in~\eqref{def:noise}, via the Bernoulli encoding of geometric variables.
%We will refer to $(X^t)_{t \geq 0}$ as the \textbf{bit-resampling noise}.

%{\bf Perhaps no reason to name this process. The process $(\omega^t)_{t\ge0}$ deserved a name, as it will be compared to other notions of noise.}

%Noise sensitivity as expressed in the context of last-passage percolation in Theorem~\ref{th:nsenoftraveltimes}, can be thought of as a generalisation of the notion of noise sensitivity, as introduced by Benjamini, Kalai and Schramm~\cite{benkalsch99} in~\eqref{def:bksns}, to real-valued functions.

It appears, through the identity
$$
\Var(f) - \Cov\big(f(X), f(X^t)\big) = \frac{1}{2} \E\big[ \big(f(X) - f(X^t)\big)^2\big],
$$
that noise sensitivity of $(f_n)_{n\ge1}$ means that,
%$f_n(X) - f_n(X^t)$ fluctuates as much as $f_n(X)$. In other words, 
in the limit as $n \rightarrow \infty$, $f_n(X^t)$ is not even slightly more concentrated around $f_n(X)$ than an independent copy, contrary to what one could expect for small values of $t$, for which $X$ and $X^t$ are very similar. Note that in the case where the $f_n$ are Boolean functions of variance bounded away from zero (such as planar percolation crossings at criticality), noise sensitivity as defined above coincides with noise sensitivity in the sense of~\eqref{def:bksns}.

An important quantity in the study of noise sensitivity is the influence of bits. The relevance of influences to tackle noise sensitivity problems can be conceived via the heuristic that, in order to understand the effect of randomly changing the values of a tiny proportion of bits, it would be useful to investigate the effect of changing exactly one bit. In the context of Boolean functions, the influence of bit $i$ is defined as the probability that changing the input coordinate $i$ changes the output of the function. In extending this notion to functions taking values in $\R$, it seems reasonable to also take the amount of change into account. 
For $i \in \N$, we define the {\bf influence} of $i$ on $f : \{0, 1\}^{\N} \rightarrow \R$ as
\begin{equation}\label{def:influence}
\I_i(f) := \E \big[ \big|\E \big[ f(X) \big| (X_j)_{j \neq i} \big] - f(X) \big| \big] .
\end{equation}
%For $i \in \N$, the influence of bit $i$ on a function $f : \{0, 1\}^{\N} \rightarrow \{0, 1\}$ is classically defined in analysis of Boolean functions as the probability $\P(f(\sigma_i(X)) \neq f(X))$ that, when sampling $X$ according to the product Bernoulli measure, flipping bit $i$ would change the value of $f(X)$. This was the definition used by Benjamini, Kalai and Schramm. For $f : \{0, 1\}^{\N} \rightarrow \R$, however, the amount by which the outcome of $f$ is updated when changing bit $i$ should matter. Given $f \in L^2$, $x \in \{0, 1\}^{\N}$ and $\xi \in \{0, 1\}$, the difference in the outcome of $f$ induced by changing the value of bit $i$ to $\xi$ in $x$ is $f \circ \sigma_i^{\xi}(x) - f(x)$.
For Boolean functions, the definition of influences given in~\eqref{def:influence} coincides with the classical definition up to a factor $2p(1-p)$, as explained in Section~\ref{se:nsen}. For real-valued functions, the definition is analogous to that used in~\cite{ivazha24, cardon25}, but differs from the $L^2$-notion used e.g.\ in~\cite{ahldeisfr1,ahldeisfr24}.

The first result justifying the relevance of the heuristic that ``in order to understand the effect of randomly changing the values of a tiny proportion of bits, it would be useful to investigate the effect of changing each bit individually'' is due to Benjamini, Kalai and Schramm~\cite{benkalsch99}. The authors famously proved that a sufficient condition for a sequence $(f_n)_{n\ge1}$ of Boolean functions, each depending on finitely many coordinates, to be noise sensitive is that 
\begin{equation}\label{eq:infsquared}
\lim_{n \rightarrow \infty} \sum_{i = 1}^{\infty} \I_i(f_n)^2 = 0.
\end{equation}
This result has come to be referred to as the BKS theorem. The authors further noted that for monotone functions, the condition in~\eqref{eq:infsquared} is also necessary. In many settings in which noise sensitivity has been studied -- including Bernoulli percolation, but also first- and last-passage percolation -- key observables correspond to monotone functions. As a consequence, the condition in~\eqref{eq:infsquared} is of fundamental importance, since it is not possible to establish noise sensitivity without, at least indirectly, also establishing~\eqref{eq:infsquared}.

In order to prove Theorem~\ref{th:nsenoftraveltimes}, we extend the BKS theorem, from~\cite{benkalsch99}, to real-valued function on the Boolean cube. Our result is quantitative, and applies to pairs of functions from $\{0, 1\}^{\N}$ to $\R$.

\begin{theorem} \label{th:genbks}
For all $p \in (0, 1)$ and $t > 0$ there exists $\theta = \theta(p, t) > 0$ such that for all functions $f, g$ from $\{0, 1\}^{\N}$ to $\R$ having finite variance, we have
\[
\Cov \big( f(X), g(X^t) \big) \leq \sqrt{\Var(f) \Var(g)}^{1 - \theta} \bigg(\sum_{i \in \N} \I_i(f) \I_i(g) \bigg)^{\theta} .
\]
The statement holds for $\theta(p,t)=\tanh(\rho t/2)$, where $\rho=1$ for $p=1/2$ and $\rho=2\frac{2p-1}{\log(p) - \log(1 - p)}$ otherwise.
\end{theorem}

There exist a few versions of the BKS theorem in the literature in the case of a single function $f$. The original theorem was proved for Boolean functions and uniform measure ($p=1/2$) in~\cite{benkalsch99}. The theorem was extended to biased measures ($p\neq1/2$) independently by Keller and Kindler~\cite{kelkin13} and Ahlberg, Broman, Griffiths and Morris~\cite{ahlbrogrimor14}; the former of which providing a quantitative version. Very recently, and in parallel to the following paper being written, Caravenna and Donadini~\cite{cardon25} proved a version of Theorem~\ref{th:genbks} for noises including the one in~\eqref{def:noise}, with methods different from ours.
%with a similar conclusion and with assumptions including the case of the bit-resampling noise. 
%The methods of proof in their and our work differ.
In~\cite[Theorem~2.17]{cardon25} the proof uses hypercontractivity estimates in a Fourier framework, inspired by~\cite{kelkin13}. Our proof of Theorem~\ref{th:genbks} combines hypercontractivity and a Markov semigroup framework, closer to related results of van Handel, as outlined in~\cite{roshan20}, and Ivanishvili and Zhang~\cite{ivazha24}. We deviate, however, from these proofs by replacing a martingale argument by the use of a semigroup property. The extension to two functions $f$ and $g$, instead of just one, has not been considered before, and will come to use in forthcoming work.

%Finally, we try to give a more central role to this theorem as a natural generalization of the BKS theorem and as a criterion for proving noise sensitivity for non-Boolean functions. (\textcolor{red}{Not sure what the last sentence is trying to say.})

%Markov semigroup methods to prove Talagrand inequality or other relevant inequalities is done in~\cite{benros08, corled12, eldgro22},

\subsection{Bounding the influences}

The first evidence of KPZ-scaling was established for the longest increasing subsequence of a random permutation, in work of Baik, Deift and Johansson~\cite{baideijoh99,johansson00b}. Analogous results were soon after obtained for planar last-passage percolation with geometric weights by Johansson~\cite{joh00} and Baik, Deift, McLaughlin, Miller and Zhou~\cite{baideimclmilzho01}.
The approach developed in these works is analytic, exploiting combinatorial identities to enable techniques from the study of random matrices. A few years later, a probabilistic approach was developed by Cator and Groeneboom~\cite{catgro06} and Bal\'azs, Cator and Seppäläinen~\cite{balcatsep06}. The probabilistic approach combines coupling techniques and queueing theory for the study of a stationary version of the last-passage process, revealing remarkable independence and large deviation properties.

KPZ-scaling refers to the scaling of so-called shape fluctuations and transversal fluctuations. For last-passage percolation this entails that the travel time $T_n$ fluctuates of order $n^{1/3}$ around its mean, and the geodesic $\pi_n$ fluctuates transversally of order $n^{2/3}$ away from the line segment connecting its endpoints; see~\cite{joh00,baideimclmilzho01}. In particular, as first proved in~\cite[Corollary~1.3]{baideimclmilzho01}, there exist constants $0<c, C<\infty$ such that for all large $n$ we have
\begin{equation} \label{eqn:variance of Tn}
c n^{2/3} \leq \Var(T_n) \leq C n^{2/3} .
\end{equation}
Due to Theorem~\ref{th:genbks}, completing the proof of Theorem~\ref{th:nsenoftraveltimes} reduces to bounding $\sum_{v\in\Z^2}\sum_{i \in \N} I_{v,i}(T_n)^2$, where $I_{v, i}(T_n)$ is the influence on the travel time $T_n$ of bit $i$ used to encode the geometric variable a vertex $v$, and proving that this quantity is $o(n^{2/3})$ when $n$ tends to infinity.

%satisfies the scaling relations predicted by the KPZ equation refers to the fact that $T_n$ fluctuates order $n^{\chi}$ around its mean, and that $\pi_n$ fluctuates order $n^{\xi}$ around the line segment connecting its endpoints, and that these exponents take the values $\chi=1/3$ and $\xi=2/3$. 

Note that flipping the variables encoding the weight at a vertex $v$ from zero to one has an effect on the travel time $T_n$ only if that vertex is on some geodesic between the origin and $n\e_+$. On a heuristic level,
%one would expect that the is roughly equivalent to that vertex being on some geodesic between the origin and $n\e_+$.
this suggests that the sum of influences of the bits associated to a vertex $v$ are roughly proportional to the probability of this vertex being on the geodesic, and hence that
%We can actually reduce this computation, with a small error in the exponent, to
\[
\sum_{v\in\Z^2}\sum_{i\in\N}I_{v,i}(T_n)^2\approx\sum_{v \in \Z^2} \P(v \in \pi_n)^2 ,
\]
where we recall that $\pi_n$ is the set of points that are on some geodesic from the origin to $n\e_+$. Indeed, we shall later see that this is the case, with a small error in the exponent. This reduces the computation of influences to bounds on the probability of a given vertex being on the geodesic.

As mentioned, the geodesic between the origin and $n\e_+$ is known to fluctuate within a range of order $n^{2/3}$ from the straight line connecting the two points. Vertices at distance further than $n^{2/3}$ from the diagonal should therefore contribute very little to the total influence. In addition, one could expect the contribution from vertices within distance $n^{2/3}$ to be roughly equal, and thus roughly of order $n^{-2/3}$.

Given $v = (v_1, v_2) \in \Z^2$, we let $|v|_1 = |v_1| + |v_2|$ denote the $\ell_1$-norm of $v$. We shall present two bounds on the probability of a given vertex being on the geodesic. While both bounds apply to all vertices in the square of side length $n$, the first will be relevant for vertices `far' from the main diagonal, whereas the second will be relevant for vertices `close' to the main diagonal. `Close' here refers to a vertex $v=(v_1,v_2)$ such that $|v_2-v_1|$ being at most order $s|v|_1^{2/3}$ and `far' refers to $|v_2-v_1|$ being at least $s|v|_1^{2/3}$, for some $s$ growing slowly with $|v|_1$. (We later set $s=C(\log|v|_1)^{1/3}$.)

The theorem we prove is the following.

%In exactly solvable models such as geometric LPP, it is known that, at distance $\approx n$ from the corners $(0, 0)$ and $(n, n)$, the geodesic set $\pi_n$ is very likely to be contained within distance $n^{2/3}$ of the diagonal, and there are large deviation results on the probability it deviates from that prediction. Thus, in the above sum, we will distinguish two types of vertices. On the one hand, vertices in the interior of the `Bulk' of the square of side $n$, that is, roughly, the points at distance less than $n^{2/3}$ from the diagonal from $(0, 0)$ to $(n, n)$. To give ourselves a bit more space, for $\eta > 0$ we let $\mathrm{Bulk}_n(\eta)$ be defined by
%$$
%\mathrm{Bulk}_n(\eta) := \Big\{ v \in [0, n]^2 : |v_2 - v_1| \leq \min \big( |v|, |n \e_+ - v| \big)^{2/3 + \eta} \Big\}.
%$$
%For vertices outside the Bulk, the probability to be visited by a geodesic is stretched exponentially small, so the contribution of these vertices is negligible. For vertices inside the $\eta$-bulk, we can prove the probability to be visited by a geodesic is not much more than $\frac{1}{n^{2/3 - \eta}}$, which means the location of a geodesic is fairly distributed in the fluctuation range around the diagonal. We use the notation $|v|_1$ for the $\ell^1$ norm of vertex $v$, that is, the sum of the absolute values of its coordinates.

\begin{theorem} \label{th:probageod}
There exists $0<c,C<\infty$ such that for all $n\ge1$ and $v \in [0, n]^2 \cap \Z^2$ such that $|v|_1 \leq n$ then
\begin{enumerate}[label=(\roman*)]
\item
$\displaystyle{
\P\big(v \in \pi(0,n\e_+)\big) \leq C \exp \left( - c \frac{|v_2 - v_1|^3}{|v|_1^2} \right)};
$
%If moreover $v \in \mathrm{Bulk}_n(\eta)$, then
\item
$
\displaystyle{\P\big(v \in \pi(0,n\e_+)\big) \leq C \left(\frac{ \log |v|_1}{|v|_1} \right)^{2/3}}.
$
\end{enumerate}
\end{theorem}

%There are two approaches that capture the $1/3$ and $2/3$ scaling exponents of travel times and geodesics. The approach of~\cite{baideijoh99,johansson00b} is analytic, exploiting combinatorial identities to enable techniques from the study of random matrices. A few years later, a probabilistic approach was developed by Cator and Groeneboom~\cite{catgro06} and Bal\'azs, Cator and Seppäläinen~\cite{balcatsep06}. The probabilistic approach combines coupling techniques and queueing theory for the study of a stationary version of the last-passage process, revealing remarkable independence and large deviation properties.
%We shall exploit these properties in our proof of Theorem~\ref{th:probageod}.

While the correct order of fluctuations have been known since the work of Baik, Deift and Johansson~\cite{baideijoh99,johansson00b}, refinements of these methods able to produce upper bounds on the probability of visiting a given vertex are more recent.
We will derive the first part of Theorem~\ref{th:probageod} from known large deviation bounds on so-called exit points in stationary last-passage percolation. Such bounds were obtained for exponential last-passage percolation by Basu, Ganguly and Zhang~\cite{basganzha21} via the analytic approach, and by Elnur, Janjigian and Seppäläinen~\cite{elnjansep23} via coupling techniques of the probabilistic approach. The latter was extended to geometric last-passage percolation by Groathouse, Janjigian and Rassoul-Agha~\cite{grojanras25}. We deduce part~\emph{(i)} of Theorem~\ref{th:probageod} from Theorem~\ref{backgeodlargedev}, which is a slight reformulation of a theorem from~\cite{grojanras25}.

The second part of Theorem~\ref{th:probageod} requires a careful analysis of the probability of a given vertex close to the line segment connecting the origin to $n\e_+$ being on the geodesic. 
Basu, Hoffman and Sly~\cite{bashofsly22} obtained bounds of this type from coalescence estimates obtained via the analytic method.
%in work of Basu, Sarkar and Sly~\cite{bassarsly19}.
%We shall in our proof opt for the probabilistic approach based on couplings of stationary last-passage percolation processes. As discovered in work of 
Shortly after, Bal\'azs, Busani and Seppäläinen~\cite{balbussep20} obtained similar bounds using the coupling technique of stationary last-passage percolation processes. This technique consists in reducing the calculation of the probability of a vertex close to the main diagonal being on a geodesic, to a calculation of a random walk staying above zero for a given number of steps. We shall here opt for this more probabilistic approach.

In~\cite{balbussep20}, the authors considered exponential last-passage percolation. Their approach was later extended to geometric last-passage percolation by Groathouse, Janjigian and Rassoul-Agha~\cite{grojanras25}. Although the precise statement of the above theorem is new, these works obtained versions of this theorem. Our Theorem~\ref{th:probageod} refines the estimates in these papers, and extends them to cover the full $n\times n$ square, as opposed to merely apply for vertices in the centre of the square, as in all of the above mentioned works. Indeed, it will be crucial to obtain bounds for vertices also close to the corners in order to bound the sum of influences, and hence prove Theorem~\ref{th:nsenoftraveltimes}.

\begin{remark}
We have, in Theorem~\ref{th:probageod}, considered diagonal crossings of squares only. However, the same proof extends to diagonal crossings of rectangles. More precisely, it is possible to show that for every $\delta>0$ there exists constants $0<c,C<\infty$ such that for all $0\le v\le u$ such that $u_2/u_1\in(\delta,1/\delta)$ and $|v|_1 \leq \frac{1}{2}|u|_1$ we have
\begin{enumerate}[label=(\roman*)]
\item
$\displaystyle{
\P\big(v \in \pi(0,u)\big) \leq C \exp \left( - c \frac{|v_2 - v_1u_2/u_1|^3}{|v|_1^2} \right)};
$
%If moreover $v \in \mathrm{Bulk}_n(\eta)$, then
\item
$
\displaystyle{\P\big(v \in \pi(0,u)\big) \leq C \left(\frac{ \log |v|_1}{|v|_1} \right)^{2/3}}.
$
\end{enumerate}
\end{remark}

\begin{remark}
As an immediate corollary of Theorem~\ref{th:probageod} we obtain the well-known bound on the transversal fluctuation exponent $\xi=2/3$; see Corollary~\ref{cor:xi} for the precise statement.
This result is, of course, well-known from~\cite{joh00,baideimclmilzho01}. The novelty here is the observation that the lower bound $\xi\ge2/3$ can be obtained via the coupling method of stationary last-passage times, avoiding the use of integral probability. (In~\cite[Remark~5.4]{sep18} this was stated as unknown.) It is possible that this observation has been made before, but it seems plausible that the observation has not yet appeared in print.
\end{remark}

\subsection{Discussion of the notion of noise}

When perturbing a uniformly chosen element in $\{0,1\}^n$, which acts independently and homogeneously on the coordinates, there is only one choice for a given noise strength. In the presence of an underlying geometry, other notions of perturbation may be natural, such as the exclusion dynamics studied in~\cite{brogarste13}, but these would not preserve the independence between coordinates.

When extending the notion of a perturbation beyond the Boolean setting, there are various ways to do this, maintaining the independence between coordinates. We have in Theorem~\ref{th:nsenoftraveltimes} opted for a Bernoulli encoding of the geometric variables, and perturbed the configuration by resampling a small proportion of the bits in this encoding, much like in the Boolean setting. In the discussion that follows, we will refer to this perturbation, defined in~\eqref{eq:bitresampling}, as the {\bf bit-resampling noise}.

Since geometric random variables are usually introduced in the context of Bernoulli trials, the bit-resampling noise becomes a natural notion of noise in this setting. For other specific distributions, other notions of noise may seem natural. For instance, in the case of the Gaussian distribution, a notion of noise considered in the literature is by updating the Gaussian variables independently according to Ornstein-Uhlenbeck processes; see e.g.~\cite{cha14}. The bit-resampling noise and the Ornstein-Uhlenbeck dynamics have in common that the corresponding Markov semi-group is hypercontractive, enabling an analysis of noise sensitivity via influences, as in Theorem~\ref{th:genbks}; see Theorem~\ref{th:gaussbks}.

For other distributions, such as the exponential, under which last-passage percolation remains exactly solvable, these notions of noise may also be applied. Indeed, any real-valued random variable can be encoded using a countable family of Bernoulli random variables. Moreover, an exponential can be expressed as the sum of two independent Gaussian variables squared. Hence, in an implicit way, Theorem~\ref{th:genbks} does not only loosen the condition on the codomain of the functions under consideration, but also on the domain.
%Indeed, any real-valued random variable can be encoded using a countable family of Bernoulli random variables, and hence be investigated through this encoding.
Arguably, this generalisation of the input space is only in appearance, as we are limited to consider noise sensitivity with respect to the bit-resampling noise applied via an abstract encoding, which for most distributions would feel artificial.
%However, we believe that slight improvements to our method could produce criteria for noise sensitivity for a wider range of dynamics.

%For example, we could work with Exponential weights instead of Geometric ones. Encoding exponential random variables as sums of two Gaussian variables squared, and letting the Gaussians evolve according to Ornstein-Uhlenbeck as a noise, the analogues of Theorems \ref{th:probageod} and \ref{th:nsenoftraveltimes} in Exponential last-passage percolation hold with essentially the same proofs. However, in that example we are still using an encoding relying on the Ornstein-Uhlenbeck dynamics, which is hypercontractive.

In a setting like last-passage percolation, where weights are assigned to the sites of the square lattice, another candidate for noise is to resample each weight entirely with a small probability. This notion has the advantage of remaining equally natural for distributions other than the geometric one. 
%A rightful question to ask is whether we can say anything for more general types of noise. 
%Indeed, the property to be noise sensitive should not depend on such specific properties of the noise. 
Let $(\omg_v)_{v \in \Z^2}$ and $(\omg'_v)_{v \in \Z^2}$ be two independent collections of independent geometrically distributed random variables, with parameter $p\in(0,1)$. Let also $(U_v)_{v\in\Z^2}$ be exponentially distributed with mean 1, and independent from the remaining variables. We obtain a perturbation $\tilde\omega^t=(\tilde{\omg}^t_v)_{v\in\Z^2}$ of $\omega=(\omg_v)_{v \in \Z^2}$ via
\begin{equation}\label{eq:siteresamping}
\tilde{\omg}_t :=
\left\{
\begin{aligned}
	\omg_v &&& \text{if } t<U_v; \\
	\omg'_v &&& \text{if } t\ge U_v.
\end{aligned}
\right.
\end{equation}
We refer to the perturbation in~\eqref{eq:siteresamping} as the \textbf{site-resampling noise}.

One would expect that the property of being noise sensitive should not depend on the specific notion of noise under consideration.
However, the site-resampling noise is (in general) not hypercontractive, so trying to extend our Theorem~\ref{th:genbks} to site-resampling of the geometric distribution fails; see Section~\ref{se:geometric resampling not hyperc}. Consequently, it remains an open problem to establish noise sensitivity of geometric last-passage percolation with respect to the site-resampling noise.

We next state a result which provides a comparison between the bit-resampling noise defined in~\eqref{eq:bitresampling} and the site-resampling noise defined in~\eqref{eq:siteresamping}. This shows that our approach just barely fails to establish noise sensitivity also with respect to the site-resampling noise.

%We express a relation between different types of noise in the form of a covariance inequality. It notably implies that functions noise sensitive for the bit-resampling noise are almost noise sensitive for the geometric resampling noise.

\begin{proposition} \label{prop:couplingargument}
For every $p\in(0,1)$ there exists a constant $C > 0$ such that for every sequence $(t_n)_{n \in \N}$ such that $0<t_n\le1/(C\log n)$, we have that if
\[
\lim_{n \rightarrow \infty} \Corr \big(T_n(\omg), T_n(\omg^{t_n}) \big) = 0 ,
\]
then
\[
\lim_{n \rightarrow \infty} \Corr \big( T_n(\omg), T_n(\tilde{\omg}^{C \,t_n\log n}) \big) = 0 .
\]
\end{proposition}

In particular, if the travel times $T_n$ of geometric last-passage percolation are noise sensitive for the bit-resampling noise along a sequence $t_n = n^{- \alpha + o(1)}$, then they are also noise sensitive for the site-resampling noise along the sequence $t'_n = C \,t_n\log n= n^{-\alpha + o(1)}$. Inspecting our proof of Theorem~\ref{th:nsenoftraveltimes}, which is based on Theorem~\ref{th:genbks}, shows that the conclusion of the theorem holds for the bit-resampling noise for sequences $t_n\gg1/\log n$. Hence, our approach just barely fails to establish noise sensitivity for site-resampling.
%Note that this is just not enough for the BKS theorem in the case of the bit-resampling noise to obtain noise sensitivity for the geometric-resampling noise.

\subsection{From stability to chaos in last-passage percolation}

Closely related to noise sensitivity is the notion of chaos, first explored in the context of Gaussian disordered systems by Chatterjee~\cite{cha14}. In the language of disordered systems, chaos differs from noise sensitivity in that it asks for the sensitivity of the so-called `ground state' of a disordered system, as opposed to the sensitivity of the `energy' of the ground state. Instead of making these terms precise, let us simply explain what they correspond to in the context of last-passage percolation. Thinking of the weight configuration $\omega$ as a random potential, and the weight accumulated along a path as the energy of the path in the random potential, then geodesics play the role of ground states, i.e.\ energy-optimising paths, and their travel times corresponds to the ground state energies. Chaos would hence correspond to sensitivity of the geodesic $\pi_n$ as the weight configuration is exposed to a slight perturbation.

As discussed above, there are various possible ways to extend the notion of a perturbation when moving from a Boolean to a more general setting. In order to align the discussion below with prior work on the topic, we shall here consider the site-resampling noise. Hence,  we shall say that geometric last-passage percolation exhibits {\bf chaos} if for every $t>0$, as $n\to\infty$, we have
\begin{equation}\label{eq:chaos}
    \E\big[\big|\pi_n(\omega)\cap\pi_n(\tilde\omega^t)\big|\big]=o(n).
\end{equation}

Chatterjee~\cite{cha14} studied certain Gaussian disordered systems, including random polymers, random matrices and spin-glasses, and showed that statements analogous to~\eqref{eq:chaos} hold under Ornstein-Uhlenbeck dynamics. Later, Bordenave, Lugosi and Zhivotovskiy~\cite{borlugzhi20} established a sharp transition at $t\asymp n^{-1/3}$, from a stable regime to a chaotic regime, for the largest eigenvector of an $n\times n$ Wigner matrix. In the context of Brownian last-passage percolation, Ganguly and Hammond~\cite{ganham24} established a sharp transition from stability to chaos, again at $t\asymp n^{-1/3}$. A weaker form of this transition was later established for a large family of first- and last-passage percolation models by Ahlberg, Deijfen and Sfragara~\cite{ahldeisfr1,ahldeisfr24}, occurring at $t\asymp \frac1n\Var(T_n)$, which under KPZ-scaling is equivalent to $t\asymp n^{-1/3}$. In particular, it was shown in~\cite{ahldeisfr24} that for geometric last-passage percolation under site-resampling, as $n\to\infty$, we have
\begin{description}
    \item[\quad {\rm\it Stability:}] for $t\ll n^{-1/3}$ we have $\Corr \big(T_n(\omg), T_n(\tilde\omg^t) \big)=1-o(1)$;
    \item[\quad {\rm\it Chaos:}] for $t\gg n^{-1/3}$ we have $\E\big[\big|\pi_n(\omega)\cap\pi_n(\tilde\omega^t)\big|\big]=o(n)$.
\end{description}

Note that the above result considers different observables in the two regimes, contrary to~\cite{ganham24} which considers the expected overlap of geodesics in both regimes. Altogether, these results give evidence that the transition from stability to chaos, for models adhering to KPZ-scaling, should occur at $t\asymp n^{-1/3}$. We expect that the transition from stability to noise sensitivity scales similarly, leading to the following conjecture, stated also in~\cite{ahldeisfr24}.

%We refer to lecture notes of Seppäläinen \cite{sep09} for an overview of the model and analysis leading to some of the exactly solvable properties. The effect of noise on $\pi_n$ was studied by Ahlberg, Deijfen and Sfragara \cite{ahldeisfr24} and is related to the alternative notions of chaos and multiple valleys, introduced by Chatterjee in the context of polymer and spin glass models \cite{cha14}. Ganguly and Hammond \cite{ganham24} proved precise results analogous to the ones of \cite{ahldeisfr24} concerning the transition from stability to chaos in the alternative model of Brownian last-passage percolation. The work in~\cite{ganham24} suggests than the transition from noise stability to noise sensitivity should occur at $t\sim n^{-1/3}$.

\begin{conjecture}\label{conj:threshold}
    Consider planar last-passage percolation with geometric weights. The transition from stability to noise sensitivity occurs at $t\asymp n^{-1/3}$ in the sense that, as $n\to\infty$,
    \begin{enumerate}[label=(\roman*)]
    \item for $t\ll n^{-1/3}$ we have $\Corr \big(T_n(\omg), T_n(\tilde\omg^t) \big)=1-o(1)$;
    \item for $t\gg n^{-1/3}$ we have $\Corr \big(T_n(\omg), T_n(\tilde\omg^t) \big)=o(1)$.
    \end{enumerate}
\end{conjecture}

Further support for the above conjecture comes from a recent approach to noise sensitivity developed by Tassion and Vanneuville~\cite{tasvan23}. They showed, in the context of Bernoulli site percolation on the triangular lattice, that stability of the pivotal set for $t\ll n^{-3/4}$ implies noise sensitive behaviour for $t\gg n^{-3/4}$, thus reproving the occurrence of the noise sensitivity threshold at $t\asymp n^{-3/4}$, as established already in~\cite{garpetsch10}. This suggests, indeed, that the thresholds for noise sensitivity and chaos coincide.

We emphasise that the first part of Conjecture~\ref{conj:threshold} was proved in~\cite{ahldeisfr24}. Theorem~\ref{th:nsenoftraveltimes} takes a first step towards a proof of the second part of the conjecture. We provide a list of further open problems at the end of this paper.

%\textcolor{red}{(sketchy: conjecture of transition from stability to noise sensitivity at $n^{-1/3}$) We can establish
%\[
%\frac{d}{dt}\Cov(f(\omg), f(\omg_t)) \leq \sum_i \Cov( \nabla_i f(\omg), \nabla_i f(\omg_t))
%\]
%where $\nabla_i f$ is a quantity we introduce later but can be understood as an "influence event": for Bernoulli percolation it is the four-arm event and for LPP it is the event of being on the geodesic. Knowledge of arm exponents in critical Bernoulli percolation on the triangular lattice (derived by Smirnov and Werner~\cite{sw01} and Lawler, Schramm and Werner~\cite{lsw02}) imply stability to noise up to $t_n = n^{-3/4}$ with this "simple" bound. Noise sensitivity along sequences $t_n \gg n^{-3/4}$ is much more difficult to prove and it was a result by Garban, Pete and Schramm \cite{garpetsch10}.
%A similar argument, exposed in \cite{ahldeisfr24}, establishes stability to noise in LPP for $t_n \ll n^{-1/3}$. It is conjectured that noise sensitivity of travel-times along sequences $t_n \gg n^{-1/3}$. It would never be possible to obtain such a result using a BKS type of argument, this can only work up to $t_n \gg 1/\log(n)$.}

\subsection{Organization of the article}

In Section~\ref{se:nsen}, we study noise sensitivity of non-Boolean functions, and prove the generalisation of the BKS theorem stated in Theorem~\ref{th:genbks}. In Section~\ref{se:stationarylpp}, we revise the stationary approach to last-passage percolation via models with boundary, and present the preliminary results that we shall need for Section~\ref{se:probageod}. In Section~\ref{se:probageod}, we provide two bounds on the probability of a given vertex being on the geodesics, which together amount to Theorem~\ref{th:probageod}. Sections~\ref{se:stationarylpp} and~\ref{se:probageod} can be read independently of Section~\ref{se:nsen}. In Section~\ref{se:proof of nsen lpp} we use the results of Sections~\ref{se:nsen} and~\ref{se:probageod} to prove our main result Theorem~\ref{th:nsenoftraveltimes}. In Section~\ref{se:coupling argument} we prove Proposition~\ref{prop:couplingargument} on the comparison between the bit-resampling noise and the site-resampling noise. We end the paper with a discussion of further open problems in Section~\ref{se:open}.

\section{A generalization of the BKS theorem} \label{se:nsen}

In this section we prove a version of the BKS theorem, which is stated as Theorem~\ref{th:genbks} of the introduction, for pairs of functions on the hypercube that take values in $\R$ instead of $\{0, 1\}$. The approach we take involves interpreting the bit resampling dynamics as a Markov semigroup, and understanding its behaviour with respect to certain differential operators.

A key ingredient in the different proofs of the BKS theorem, as well as related results such as the Kahn-Kalai-Linial theorem and Talagrand's inequality, is an hypercontractive inequality. This is for example discussed in~\cite{garste15,odonnell}. Benjamini, Kalai and Schramm~\cite{benkalsch99} used hypercontractivity in a spectral setting, and their proofs rely on the fact that Boolean functions are bounded and can take only two values. 
%It relies on a definition of influence taking into account the amount by which changing a bit changes the outputs of the functions. 
It has been known for some time that interpolation methods relying on calculus with Markov semigroups and differential operations can produce noise sensitivity-relevant results. For instance, the hypercontractive properties of the semigroup associated to resampling bits of the discrete hypercube at constant rate was used to prove Talagrand's inequality by Cordero-Erausquin and Ledoux~\cite{corled12}.
%This motivates the introduction of $L^1$ influences as was done in.
A similar approach was used by Chatterjee~\cite{cha14} to establish an equivalence between superconcentration and chaos for certain Gaussian disordered systems, such as spin glasses and directed polymers, based on the Ornstein-Uhlenbeck semigroup.
Related results have been obtained also by Eldan and Gross~\cite{eldgro22}.
%\textcolor{red}{Benaim-Rossignol?}

Our approach, on the Boolean hypercube, will follow the exposition of related results presented in lectures held by Ramon van Handel, and transcribed by Rosenthal~\cite{roshan20}.
Notably, we interpret Lemma~11 of~\cite{roshan20} as a generalization of the BKS theorem. Our proof differs in that we extend the argument to biased coin flips and to pairs of functions $f$ and $g$ defined in terms of a countable sequence of bits, and we replace a martingale argument using properties of the Markov semigroup.
We mention once again that results of a similar flavour have been obtained in parallel by different methods in work of Caravenna and Donadini~\cite{cardon25}.

\subsection{Noise sensitivity in a semigroup framework}

Let $p \in (0, 1)$ be fixed, and let $X = (X_i)_{i \in \N}$ be a sequence of independent Bernoulli random variables of parameter $p$. Let $L^2$ denote the set of functions $f : \{0, 1\}^{\N} \rightarrow \R$ such that $f(X)^2$ has finite expectation.\footnote{The space which we call $L^2$ thus depends on $p$.} For $f \in L^2$, we write $\E[f]$ and $\Var(f)$ for the expectation and variance of $f(X)$.

We introduce the \textbf{bit-resampling process}
%or \textbf{bit-resampling noise}
$\mathbf{X} := (X^t)_{t \geq 0} = (X^t_i)_{t \geq 0, i \in \N}$ as the continuous-time stochastic process on $\{0, 1\}^{\N}$ started from $X$ and where the dynamics is described by independently resampling the bits according to independent Poisson clocks of rate 1. Hence, $X^0_i = X_i$ for all $i \in \N$, and at every time $t$ the Poisson clock assigned to $i$ rings, $X_i^t$ is resampled according to a Bernoulli distribution of parameter $p$ independently of everything else. Note that the bit-resampling process defined here differs from the bit-resampling noise as defined in~\eqref{def:noise} in that
\[
\mathbf{X} \text{ is Markov, reversible and invariant for the law of X}.
\]
Note, however, that the joint distribution of the pair $(X,X^t)$ is equal in the two definitions. Since the result we aim to prove is a result about the joint distribution of this pair, we may henceforth work with the bit-resampling process as defined here. The advantage for doing so is that we will be able to exploit the framework of Markov semigroups.

Given $x \in \{0, 1\}^{\N}$, we will denote by $\P_x$ and $\E_x$ the conditional probability and expectation corresponding to $\mathbf{X}$ with initial configuration $X=x$.
We introduce the semigroup family $(P_t)_{t \geq 0}$ associated to the Markov process $\mathbf{X}$, where for $t \geq 0$ we let $P_t : L^2 \rightarrow L^2$ be defined, for $f \in L^2$ and $x \in \{0, 1\}^{\N}$, as
\[
P_t f(x) := \E_x[f(X^t)] = \E[f(X^t) |X = x] .
\]
The family $(P_t)_{t\ge0}$ satisfies the \textbf{semigroup property}, which is a way of expressing the Markov property:
\begin{equation} \label{semi}
P_{t+s} = P_t \circ P_s\quad\text{for all }s,t\ge0 .
\end{equation}
By stationarity and reversibility of $\mathbf{X}$ we have for every $t \geq 0$ that the operator $P_t$ is symmetric:
\begin{equation}\label{eq:Pt_symmetry}
\E[f P_t g] = \E[P_t f g] \quad\text{for all }f,g\in L^2.
\end{equation}
Using the semigroup framework, and properties~\eqref{semi} and~\eqref{eq:Pt_symmetry}, we may rewrite central quantities of interest as
$$
\E[f(X)g(X^{2t})]=\E[f(P_{2t}g)]=\E[fP_t(P_tg)]=\E[(P_tf)(P_tg)].
$$
Subtracting $\E[f]\E[g]$ from both sides gives
\begin{equation}\label{eq:cov_id}
\Cov\big(f(X),g(X^{2t})\big)=\Cov(P_tf,P_tg).
\end{equation}

We may also recast the notion of influence in terms of difference operators $(\nabla_i)_{i\in\N}$ acting on $f$. Given $f \in L^2$, $x \in \{0, 1\}^{\N}$ and $\xi \in \{0, 1\}$, the difference in the outcome of $f$ induced by changing the value of bit $i$ to $\xi$ in $x$ is $f \circ \sigma_i^{\xi}(x) - f(x)$, where
%Before giving a definition of influences adapted to real-valued functions, we introduce the functions $\sigma_i, \sigma_i^1, \sigma_i^0$ from $\{0, 1\}^{\N}$ to itself having the effect of, respectively, flipping the value of bit $i$, forcing its value to one and forcing it to zero, as expressed below:
\begin{equation*}
%\sigma_i(x) &:= (x_1, \dots, x_{i-1}, 1-x_i, x_{i+1}, \dots)  ,\\
\sigma_i^\xi(x) := (x_1, \dots, x_{i-1}, \xi, x_{i+1}, \dots).
%\sigma_i^0(x) &:= (x_1, \dots, x_{i-1}, 0, x_{i+1}, \dots)  .
\end{equation*}
%For $i \in \N$, the influence of bit $i$ on a function $f : \{0, 1\}^{\N} \rightarrow \{0, 1\}$ is classically defined in analysis of Boolean functions as the probability $\P(f(\sigma_i(X)) \neq f(X))$ that, when sampling $X$ according to the product Bernoulli measure, flipping bit $i$ would change the value of $f(X)$. This was the definition used by Benjamini, Kalai and Schramm. For $f : \{0, 1\}^{\N} \rightarrow \R$, however, the amount by which the outcome of $f$ is updated when changing bit $i$ should matter.
Whenever a bit changes value in the process $(X^t)_{t \geq 0}$, its new value is decided by a Bernoulli random variable of parameter $p$. We introduce the difference operators $\nabla_i$, for $i \in \N$, acting on $L^2$ by sampling $\xi \sim \mathrm{Ber}(p)$ and letting
\begin{equation}\label{def:nabla}
\nabla_i f(x) := \E^{\xi}\big[f \circ \sigma_i^{\xi}(x) \big] - f(x) ,
\end{equation}
where $\E^{\xi}$ means the expectation is taken with respect to the random variable $\xi$ only. We now note that the notion of influence, introduced in~\eqref{def:influence}, can be expressed in therms of the difference operators as 
\begin{equation}\label{eq:inf_equiv}
    \I_i(f) = \E \big[ |\nabla_i f | \big]  .
\end{equation}
    
We can make the expression of $\nabla_i f$ more explicit by computing the expectation with respect to $\xi$ and using that $x_i \in \{0, 1\}$ to write $f(x) = x_i f(\sigma_i^1(x)) +  (1-x_i) f(\sigma_i^0(x))$, yielding
\begin{equation} \label{nabiexpr}
\nabla_i f(x) = (p - x_i) \big(f \circ \sigma_i^1(x) - f \circ \sigma_i^0(x) \big) .
\end{equation}
Since $|f \circ \sigma_i^1(X) - f \circ \sigma_i^0(X)|$ only depends on $(X_j)_{j \neq i}$, the influence of $i$ on $f$ can then be expressed as
\begin{equation} \label{exprofinfl}
\I_i(f) = 2p(1-p)\E\big[ |f \circ \sigma_i^1(X)) - f \circ \sigma_i^0(X) | \big]  .
\end{equation}
In the case where $f$ is Boolean, we recover the classical definition of influences up to the factor $2p(1-p)$.

%The choice of the absolute value instead of a square in the definition of influence is meaningful here.
%Indeed, the generalization we prove of the BKS theorem uses the $L^1$ definition.
%Ahlberg, Sfragara and Deijfen \cite{ahldeisfr24} defined an $L^2$ notion of influence, which they referred to as the coinfluence. Also here bounding the analogue of the coinfluence will be central, and we shall do so by reducing it to a statement of $L^1$ influences, as defined above.
%but doing so requires to go through the $L^1$ influence. All of these definitions may be compared to definitions adapted to Boolean functions, appearing for instance in \cite{benkalsch99, garste15, odonnell}.
Ahlberg, Sfragara and Deijfen~\cite{ahldeisfr24} worked in a similar setting with an $L^2$ notion of influence, referred to as the `coinfluence': $\E[\nabla_i f(X) \nabla_i f(X^t)]$. Bounding the coinfluences will be central also here, and we shall do so by reducing it to a statement of $L^1$ influences, as defined in~\eqref{def:influence} and~\eqref{eq:inf_equiv}.
%Here, we choose a $L^1$ definition of influence because it generalizes nicely the BKS theorem.
Indeed it is common that the $L^1$ norm of a gradient appears in $L^1-L^2$ estimates such as Talagrand's inequality. A reason for that, which will become clear in the proof, is that the hypercontractive inequality gives appropriate control of the coinfluence $\E[\nabla_i f(X) \nabla_i f(X^t)]$ in terms of $\E[|\nabla_i f|]$.

\subsection{Operator properties}

In this subsection we explore some key properties of the semigroup and difference operators introduced above, in preparation for the proof of Theorem~\ref{th:genbks}. 
At occasions it will be useful to work with local functions, that is functions depending on finitely many coordinates:
We say that $f \in L^2$ is \textbf{local} and write $f \in L^2_{loc}$ if and only if there exists $m \in \N$ such that $f(x) = f(y)$ whenever $x_j = y_j$ for all $1 \leq j \leq m$.

Note that, for any $f \in L^2$, we can define the sequence of functions $(f_k)_{k \in \N}$ by
\[
f_k(x) = \E[f(X) | X_1 = x_1, \dots, X_k = x_k] \, ,
\]
The sequence $(f_k)_{k \in \N}$ is then a sequence of local functions which converges to $f$ in $L^2$. Any time we invoke `density of $L^2_{loc}$ in $L^2$', it can be thought of as referring to an argument where we approximate functions of $L^2$ by functions of $L^2_{loc}$ in this way and let $k$ go to $+ \infty$.

Our first lemma expresses the commutativity of the operators $P_t$ and $\nabla_i$. It is a consequence of independence of the bits with respect to each other.

\begin{lemma}[Commutativity]\label{com}
For every $t \geq0, i \in \N$ and $f \in L^2$,
\[
\nabla_i P_t f = P_t \nabla_i f  .
\]
\end{lemma}

\begin{proof}
By density of $L^2_{loc}$ in $L^2$ it is sufficient to prove $\nabla_i P_t f = P_t \nabla_i f$ for every $f \in L^2_{loc}$. Let $f \in L^2_{loc}$ and $m$ be such that $f$ only depends on the first $m$ coordinates. Let $i \in \{0, \dots, m\}$, $x \in \{0, 1\}^{\N}$ and $t \geq 0$ be fixed. Let $\xi$ be a $\mathrm{Ber}(p)$-distributed random variable. When $\mathbf{X}$ is started from $\sigma_i^{\xi}(x)$, $X^t_i$ is equal to $\xi$ with probability $e^{-t}$ and it is an independent $\mathrm{Ber}(p)$ random variable with probability $1 - e^{-t}$,
so that under $\E^{\xi} \E_{\sigma_i^{\xi}(x)}$, $X^t_i$ is a $\mathrm{Ber}(p)$ random variable independent of $(X^t_j)_{j \neq i}$. For $j \neq i$, $X^t_j$ has the same distribution under $\P_{\sigma_i^{\xi}(x)}$ as under $\P_x$. Therefore, the law of $X^t$ under $\E^{\xi} \E_{\sigma_i^{\xi}(x)}$ is the same as the law of $\sigma_i^{\xi}(X^t)$ under $\E^{\xi} \E_x$. As a consequence,
\[
\E^{\xi} \Big[ \E_{\sigma_i^{\xi}(x)} \big[f(X^t)\big] \Big] = \E^{\xi}  \Big[\E_x \big[f \circ \sigma_i^{\xi}(X^t)\big] \Big]  .
\]
Subtracting $P_t f(x)$ on both sides and using the definition of $P_t$, it follows that
\[
\E^{\xi} \big[ \big( P_t f \big) \circ \sigma_i^{\xi}  (x) \big] - P_t f(x) = \E_x \Big[ \E^{\xi}\big[f \circ \sigma_i^{\xi}(X^t)\big] - f(X^t) \Big]  .
\]
Finally, we obtain from the definition of $\nabla_i$ and $P_t$ that the left-hand side is $\nabla_i P_t f(x)$ and the right-hand side is $P_t \nabla_i f(x)$.
\end{proof}

The next lemma is an expression for the infinitesimal generator of $P_t$. We refer to this as a heat equation due to the infinitesimal generator $\sum_i \nabla_i = \sum_i \nabla_i \circ \nabla_i$ being a discrete version of the Laplacian operator. It is helpful to keep in mind this analogy with a partial differential equation for interpreting other operator properties. It could also be referred to as a `quenched' dynamical Margulis-Russo formula, whereas an `annealed' version would be, for instance, the case of monotone functions in \cite[Lemma~3.2]{tasvan23}. It would be tedious to express in which sense this formula holds for general $L^2$ functions, and we circumvent this difficulty by going through local functions.

\begin{lemma}[Heat equation] \label{dynmr}
Let $f \in L^2_{loc}$. For every $x \in \{0, 1\}^{\N}$, $t \mapsto P_t f(x)$ is differentiable and its time-derivative satisfies
\[
\frac{d}{dt} P_t f(x) = \sum_{i = 0}^{\infty} \nabla_i P_t f .
\]
\end{lemma}

\begin{proof}
Using the assumption that $f$ is local, let $m$ be an integer such that $f$ depends only on the first $m$ coordinates. This allows us to consider $f$ as being a function from $\{0, 1\}^m$ to $\R$. Note that $f$ is bounded since $\{0, 1\}^m$ is finite. Fix $x \in \{0, 1\}^m$ from now on. By linearity of the expectation, for all $0\le s<t$ we have
\[
 P_t f(x) - P_s f(x) = \E\big[f(X^t) - f(X^s) | X = x \big] .
\]
Let $E_{s, t}$ be the random set
\[
E_{s, t} := \{i \in \{0, \dots, m\} : (X_i) \text{ is resampled between $s$ and $t$}\} ,
\]
and for $0 \leq i \leq m$ let $R_i(s,t)$ be the event
\[
R_i(s, t) := \{E_{s, t} = \{i\} \}.
\]
By definition of `resampling at rate 1', for every $i$ the number of times at which $(X_i)$ is resampled between $s$ and $t$ is a Poisson random variable of parameter $t-s$ and these are independent for distinct $i$. The function $f$ is bounded, the powerset of $\{0, 1\}^m$ is finite and $f(X^t) - f(X^s)$ is 0 on $\{E_{s, t} = \emptyset\}$. Summing over the possible issues of $E_{s, t}$ yields, as $|t-s|\to0$,
\[
\E \big[f(X^t) - f(X^s) \big| X = x \big] = \sum_{i = 0}^m \E \big[ \big(f(X^t) - f(X^s) \big) \1_{R_i(s, t)} \big|X = x \big] + o(|t-s|) .
\]
The event $R_i(s, t)$ is independent of $X^s$ and has probability $(t-s)(1 + o(1))$. Moreover, the distribution of $X^t$ given $X^s$ and $R_i(s, t)$ is that of $\sigma_i^{\xi}(X^s)$ where $\xi$ is a $\mathrm{Ber}(p)$ random variable independent of everything else. As a consequence,
\begin{align*}
\E \big[\big(f(X^t) - f(X^s) \big) \1_{R_i(s, t)} \big| X^s \big] &= \big(\E \big[f \circ \sigma_i^{\xi}(X^s) |X^s\big] - f(X^s) \big) \P \big(R_i(s, t) \big) \\ &= \nabla_i f(X^s) (t-s)(1 + o(1)) .
\end{align*}
Thus,
\[
P_t f(x) - P_s f(x) = (t-s) \sum_{i = 0}^m \E[ \nabla_i f(X^s)|X = x] + o(|t-s|) .
\]
Dividing by $t-s$ and evaluating the limits $t \rightarrow s^+$ and $s \rightarrow t^-$ yields 
\[
\frac{d}{dt} P_t f(x) = \sum_{i = 0}^m P_t \nabla_i f(x) ,
\]
and we conclude by exchanging $\nabla_i$ and $P_t$ using Lemma \ref{com}.\footnote{One may prove Lemma~\ref{dynmr} in a similar way by evaluating the difference $\E[f(X_t | X_0 = x) - \E[f(X_t) | X_{-h} = x]$ and letting $h$ go to $0$. Interestingly, this alternative proof does not require Lemma~\ref{com}. 
%\textcolor{red}{could be a remark rather than a footnote. Point is the proof gives a similar heat equation in contexts where commutativity is not true.}
}
\end{proof}

In the heat equation analogy, the following lemma is an integration by parts identity. Once again, this relies on the independence of the bits with respect to each other, and also on the stationarity and reversibility of the process.

\begin{lemma}[Integration by parts]\label{ipp}
For every $i \in \N$, and $f, g \in L^2$,
\[
\E \big[ f(\nabla_i g) \big] = - \E \big[(\nabla_i f)(\nabla_i g) \big]  .
\]
\end{lemma}

\begin{proof}
Let $f, g \in L^2$ and $i \in \N$. Let $\xi, \xi'$ be independent $\mathrm{Ber}(p)$ random variables independent of $X$. Recall $\E^{\xi}$ is the conditional expectation with respect to everything but $\xi$. Since $X$ is a sequence of i.i.d.\ $\mathrm{Ber}(p)$ random variables, $(X, \sigma_i^{\xi}(X))$ is equal in distribution to $(\sigma_i^{\xi'}(X), \sigma_i^{\xi}(X))$, so that
\[
\E  \big[ f \big( \sigma^{\xi}_i(X) \big) g(X) \big] = \E \big[ f \big( \sigma^{\xi}_i(X) \big) g \big( \sigma^{\xi'}_i(X) \big) \big]  ,
\]
Using independence of $\xi, \xi'$ and $X$ and regrouping all the terms inside the expectation, we obtain
\[
\E  \Big[  \E^{\xi}[f \circ \sigma^{\xi}_i(X)] g(X) - \E^{\xi}\big[f \circ \sigma^{\xi}_i(X) \big] \E^{\xi'} \big[ g  \circ \sigma^{\xi'}_i(X)  \big] \Big] = 0  ,
\]
which rewrites as $- \E [ (f+ \nabla_if) \nabla_i g ] = 0$ and gives the desired identity.
\end{proof}

The following lemma expresses that the effect of changing a single bit at time 0 decays exponentially in time. We shall mainly use it for large $t$ in order to integrate $|P_t \nabla_i f|$ near infinity. However, this property of $P_t$ and $\nabla_i$ to commute up to a multiplicative correction in $e^{\kappa t}$ with $\kappa \in \R$ is also important on its own even at small times, $\kappa$ being interpreted as a curvature.

\begin{lemma}[Time-decorrelation] \label{timedecor}
For $i \in \N$, we denote by $\mathrm{pr}_i:\{0, 1\}^{\N}\to\R$ the map $x \mapsto x_i$. For all $t \geq 0$ and $f \in L^2$,
\begin{equation}\label{eq:timedecor}
P_t \nabla_i f = e^{-t} (p-\mathrm{pr}_i) P_t \Big( \frac{\nabla_i f}{p - \mathrm{pr}_i} \Big)  .
\end{equation}
It follows that $\I_i(P_t f) = \E \big[|\nabla_i P_t f| \big]$ decays exponentially in $t$, in that
\[
\I_i(P_t f) \leq 2 e^{-t} \max(p, 1-p) \I_i(f)  .
\]
\end{lemma}

\begin{proof}
Let $f \in  L^2, i \in \N, t > 0, x \in \{0, 1\}^{\N}$. Let $T_i$ be the first time at which $(X_i^t)_{t \geq 0}$ is resampled. Rewriting $\nabla_i f$ as $(p- \mathrm{pr}_i)(f \circ \sigma_i^1 - f \circ \sigma_i^0)$ as in (\ref{nabiexpr}), and using independence of bit $i$ with $f \circ \sigma_i^1 - f \circ \sigma_i^0$, we have
\[
P_t \nabla_i f = P_t(p-\mathrm{pr}_i) P_t (f \circ \sigma_i^1 - f \circ \sigma_i^0) = P_t (p - \mathrm{pr}_i) P_t \Big( \frac{\nabla_i f}{p - \mathrm{pr}_i} \Big)  ,
\]
and for all $x$ in $\{0, 1\}^{\N}$,
\[
P_t (p - \mathrm{pr}_i)(x) =  p - \E_x[X_i^t] = e^{-t} (p-x_i)  ,
\]
which gives~\eqref{eq:timedecor}.

To obtain the announced inequality on $\I_i(P_t f)$, using the above and the independence of bit $i$ with $P_t (f \circ \sigma_i^1 - f \circ \sigma_i^0)$ we write
\[
\I_i(P_t f) = e^{-t} \E \big[ |P_t (p-\mathrm{pr}_i) P_t (f \circ \sigma_i^1 - f \circ \sigma_i^0) | \big] \leq e^{-t} \E \big[ |p - \mathrm{pr}_i| \big] \E \Big[ P_t \Big| \frac{\nabla_i f}{p - \mathrm{pr}_i}  \Big| \Big] .
\]
Since $\mathrm{pr}_i$ takes values in $\{0,1\}$, $\frac{1}{p - \mathrm{pr}_i}$ is uniformly bounded by $\frac{1}{\min (p,1-p)}$, so
\[
P_t \Big|\frac{\nabla_i f}{p - \mathrm{pr}_i} \Big| \leq \frac{1}{\min(p, 1-p)} P_t| \nabla_i f|.
\]
We conclude by using $\E \big[|p-X_i| \big] = 2 p(1-p)$ and the stationarity of $X^t$.
\end{proof}

\subsection{Hypercontractivity}

By Jensen's inequality, the operator $P_t$ is a contraction from $L^2$ to itself. The following lemma, stated without a proof, expresses the so-called \textbf{hypercontractive} property of $P_t$ that it is actually a contraction from $L^2$ to $L^{q(t)}$, where $q(t) < 2$ when $t > 0$. The explicit form of the constant is stated in \cite[Page~5]{corled12}, and directly implied by Theorems 1.3.2 and 2.8.2 in \cite{abcfgmrs00}. See also \cite[Chapter~6]{garste15} for a proof of the case $p =1/2$.

\begin{lemma}[Hypercontractivity] \label{hclemma}
Let $\rho := 2 \frac{2p-1}{\log(p) - \log(1-p)}$ if $p \neq \frac{1}{2}$ and $\rho := 1$ if $p = \frac{1}{2}$. For all $t \geq 0$ and $f \in L^2$,
\[
\E \big[ (P_tf)^2 \big] \leq \E \big[ |f|^{1 + e^{- 2 \rho t}} \big]^{\frac{2}{1 + e^{- 2 \rho t}}} .
\]
\end{lemma}

Combining hypercontractivity with Hölder's inequality, we can obtain the following corollary, see also \cite[Corollary~10]{roshan20}.

\begin{corollary} \label{hccor}
Let $f \in L^2$ and let $\rho = \rho(p)$ be as in the statement of Lemma \ref{hclemma}. For all $t \geq 0$,
\[
\E_p \big[ (P_t f)^2 \big] \leq \E_p \big[ f^2 \big]^{1 - \tanh(\rho t)} \E_p \big[ |f| \big] ^{2 \tanh(\rho t)} .
\]
\end{corollary}

\begin{proof}
Let $t > 0$ and $f \in L^2$ and let $\rho$ be as required. By Lemma \ref{hclemma},
\[
\E \big[ (P_t f)^2 \big] \leq \E \big[ |f|^{1 + e^{-2 \rho t}} \big]^{\frac{2}{1 + e^{-2 \rho t}}} .
\]
Applying Hölder's inequality to the functions $|f|^{2 e^{- 2 \rho t}}$, $|f|^{1 - e^{- 2 \rho t}}$ and to the conjugate exponents $e^{2 \rho t}$, $(1 - e^{- 2 \rho t})^{-1}$, we obtain
\[
\E \big[ |f|^{1 + e^{-2 \rho t}} \big] \leq \E \big[ |f|^2\big]^{e^{-2 \rho t}} \E \big[ |f|\big]^{1 - e^{- 2 \rho t}} .
\]
Combining these last two inequalities yields
\[
\E \big[(P_t f)^2 \big] \leq \E \big[|f^2| \big]^{\frac{2e^{-2 \rho t}}{1 + e^{-2 \rho t}}} \E \big[ |f| \big]^{2 \frac{1 - e^{- 2 \rho t}}{1 + e^{- 2 \rho t}}},
\]
which is the desired inequality.
\end{proof}

\subsection{Proof of the generalized BKS theorem}

We hereby prove Theorem~\ref{th:genbks}. This proof is similar to the one given in~\cite{roshan20}, but differs in that we replace a martingale argument by splitting $P_s$ into $P_t P_{s-t}$ at the right place. We also invoke the Cauchy-Schwarz inequality before and after using Corollary~\ref{hccor} in order to deal with different functions $f$ and $g$.

\begin{proof}[Proof of Theorem \ref{th:genbks}] \label{proofofgenbks}

Let $f, g \in L^2_{loc}$. Any sum will be over $i \in \N$ in this proof and, $f, g$ being local, these sums will always be finite.
Recall from~\eqref{eq:cov_id} that $\Cov(f(X), g(X^{2t})) = \Cov(P_t f, P_t g)$.
Consider the map $t \mapsto \E[(P_t f)(P_t g)]$. By Lemma~\ref{dynmr}, differentiation of a product of two functions and the dominated convergence theorem, the map is twice differentiable and hence $C^1$. So, we can use integration by parts (Lemma~\ref{ipp}), then commutativity (Lemma~\ref{com}) to further compute the derivative as follows:
\[
\frac{d}{dt} \E\big[(P_t f)(P_t g)\big] = \sum_{i} \E\big[P_t f (\nabla_i P_t g) + (\nabla_i P_t f) P_t g\big] = - 2 \sum_i \E\big[(P_t \nabla_i f)(P_t \nabla_i g)\big] .
\]
Moreover, $P_t f$ converges to $\E[f]$ almost surely and thus in $L^2$ when $t \rightarrow \infty$, and similarly for $g$. When integrating from $t$ to $+ \infty$, we obtain
\begin{align} \label{intformula}
\Cov(P_t f, P_t g) &= - \big( \lim_{s \rightarrow + \infty}\E[(P_s f) P_s g)] - \E[P_t f P_t g] \big) \nonumber \\
&= 2 \int_t^{\infty} \sum_i \E \big[(P_s \nabla_i f)(P_s \nabla_i g) \big] ds .
\end{align}
By the Cauchy-Schwarz inequality,
\[
\E\big[(P_s \nabla_i f)(P_s \nabla_i g)\big] \leq \sqrt{\E[(P_s \nabla_i f)^2] \E[(P_s \nabla_i g)^2]} .
\]
By the semigroup property (\ref{semi}), for all $s \geq t$, $P_s \nabla_i f = P_t (P_{s-t} \nabla_i f)$. Applying Corollary \ref{hccor} and commutativity again yields
\[
\E \big[(P_s\nabla_i f)^2 \big] \leq \E \big[(P_{s-t} \nabla_i f)^2 \big]^{1 - \tanh(\rho t)} \I_i(P_{s-t} f)^{2 \tanh(\rho t)} .
\]
After applying this inequality to $P_s \nabla_i g$ as well and using it in (\ref{intformula}), we obtain
\[
\Cov(P_t f, P_t g) \leq 2 \int_t^{\infty} \sum_i \sqrt{\E[(P_{s-t} \nabla_i f)^2] \E[(P_{s-t} \nabla_i g)^2]}^{1 - \tanh(\rho t)} \big( \I_i(P_{s-t} f) \I_i(P_{s-t} g) \big)^{\tanh(\rho t)} ds .
\]

We then apply Hölder's inequality for the Lebesgue and counting measure $(\int \sum)$, with conjugate exponents $1/(1 - \tanh(\rho t))$ and $1/(\tanh(\rho t))$, resulting in
\[
\begin{split}
\Cov(P_t f, P_t g) \leq \Big(2 \int_t^{\infty} \sum_i \sqrt{\E[(P_{s-t} \nabla_i f)^2] \E[(P_{s-t} \nabla_i g)^2] \big]} ds \Big)^{1 - \tanh(\rho t)}\\
\times \Big(2 \int_t^{\infty} \sum_i \I_i(P_{s-t} f) \I_i(P_{s-t} g)  ds \Big)^{\tanh(\rho t)} .
\end{split}
\]

Consider the first integral in the above expression. After the change of variable $u = s-t$ and applying the Cauchy-Schwarz inequality to $\int \sum$ again, 
$$
2\int_t^{\infty} \sum_i \sqrt{\E\big[(P_{s- t} \nabla_i f)^2 \big] \E \big[(P_{s-t} \nabla_i g)^2 \big]} ds \leq \sqrt{2\int_0^{\infty} \sum_i \E[(P_u \nabla_i f)^2] du  } \sqrt{2\int_0^{\infty} \sum_i \E[(P_u \nabla_i g)^2] du},
$$
which by~\eqref{intformula} we identify as $\sqrt{\Var(f) \Var(g)}$.
We can upper-bound the second integral and relate it to the sum of $I_i(f) I_i(g)$ by the second inequality in Lemma \ref{timedecor}:
\[
2\int_t^{\infty} \sum_i \I_i(P_{s-t} f) \I_i(P_{s-t} g) ds \leq \big(2 \max(p, 1-p) \big)^2\int_{0}^{\infty} 2 e^{-2u} du \bigg(\sum_i \I_i(f) \I_i(g) \bigg) .
\]
The integral of $2 e^{-2u}du$ being $1$, and since $2 \max(p, 1 - p) \leq 2$, we deduce for all $f, g \in L^2_{loc}$, for all $t \geq 0$,
\[
\Cov \big( f(X), g(X^{2t}) \big) \leq \sqrt{\Var(f) \Var(g) }^{1 - \tanh(\rho t)} \left( 4 \sum_{i\in\N} \I_i(f) \I_i(g) \right)^{\tanh(\rho t)} .
\]

Finally, the proof extends to every $f, g \in L^2$ by density using approximation of $f \in L^2$ by $f_k:x \mapsto \E[f| X_0 = x_0, \dots, X_k = x_k]$ when $k \rightarrow +\infty$ (and respectively for $g$). Indeed, $\I_i(f_k) \leq \I_i(f)$ by Jensen's inequality, and $\Var(f_k) \rightarrow \Var(f)$ and $\Var(P_t f_k) \rightarrow \Var(P_t f)$ by the $L^2$ martingale convergence theorem. (Note that $P_t f_k = (P_t f)_k$ by stationarity of the Markov process.)
\end{proof}

We end this section by discussing the limitations in extending the above proof to other semigroups (say, of Markov reversible processes). A key property of the bit-resampling dynamics is that it acts independently on the bits, which is at the core of the proofs of the commutativity (Lemma \ref{com}), integration by parts (Lemma \ref{ipp}) and exponential decay of $\I_i(P_t f)$ (Lemma \ref{timedecor}). We expect these properties to hold with corrections for small values of $t$ for dynamics where the information propagates slowly, neighbor-wise for example.

The hypercontractive properties of the semigroup are also essential to the proof. An example of a resampling dynamics that is not hypercontractive is that of a geometric random variable resampled at exponential rate (see Section~\ref{se:geometric resampling not hyperc} below). This is why we later encode geometric random variables through $\{0, 1\}$-valued bits and run the bit-resampling dynamics, in order to obtain noise sensitivity of geometric last-passage percolation (for that specific noise). For a semigroup satisfying a version of hypercontractivity with corrections, one could possibly extract a noise sensitivity criterion similar to the BKS theorem. In the case of independent Ornstein-Uhlenbeck processes, hypercontractivity is known, so the following theorem follows with the same proof as Theorem~\ref{th:genbks}.

\begin{theorem}\label{th:gaussbks}
Let $(X_i^t)_{i \in \N, t \geq 0}$ be a family of i.i.d.\ Ornstein-Uhlenbeck processes started from their invariant measure. Then, for every sequence $(f_n)_{n\ge1}$ of square integrable local\footnote{We say that $f : \R^{\N} \rightarrow \R$ is local if there is an $m$ such that $f$ depends only on the first $m$ coordinates.} functions $f_n:\R^{\N} \rightarrow \R$,
\[
\lim_{n \rightarrow\infty} \frac{\sum_{i \in \N} \E[|\p_i f_n|]^2}{\Var(f_n)} = 0 \quad\Rightarrow\quad\lim_{n\to\infty}\Corr\big(f_n(X), f_n(X^t)\big)=0.
\]
%then $(f_n)$ is noise sensitive in the sense that for every $t > 0$, $\Corr(f(X), f(X^t))$ converges to $0$ when $n$ goes to $\infty$.
\end{theorem}

Satisfying a log-Sobolev inequality is equivalent to hypercontractivity by Gross's theorem; we refer to \cite[Chapter~5]{led01} or \cite[Appendix~B]{cha14} for a proof of a log-Sobolev inequality in the case of an Ornstein-Uhlenbeck dynamics (this also relies on properties such as commutativity). Diaconis and Saloff-Coste \cite{diasal96} established log-Sobolev inequalities for a number of Markov chains on finite sets, so one could obtain similar criteria of noise sensitivity in these contexts up to managing the long-time behavior of the chains and providing an analogue to the sum of influences squared.

\subsection{Non-hypercontractive semi-groups}
 \label{se:geometric resampling not hyperc}
 
%\begin{remark} \label{rk:geometric resampling not hyperc}
The site-resampling semigroup for the geometric distribution is not hypercontractive. We sketch an argument to prove this for functions of a single geometric variable. We will omit a few computations related to computing the moments of a geometric random variable, and we will exchange derivative and expectations without justification.

Let $G$ be a geometric random variable of parameter $p$ and let $G^t$ be obtained by jumping to an independent geometric random variable of the same parameter at every mark of a unit Poisson point process.
%(this is the geometric-resampling process, see also Section \ref{se:coupling argument}).
We let $P_t$ denote the associated Markov semigroup, which for $f : \N \rightarrow \R$ and $k \in \N$ is defined by
\[
P_t f(k) = \E[f(G^t) |G^0 = k]
\]
Assume, by contradiction, that there exists $\rho > 0$ such that for all $f : \N \rightarrow \R$ with $\E[f(G)^2] < \infty$,
\[
\| P_t f \|_2 \leq \|f\|_{1 + e^{- 2 \rho t}} ,
\]
Let $f :\N \rightarrow \R$ be a positive function such that $\E[f(G)^2]$ is finite. Let $\phi(t) := \log\|f\|_{1 + e^{- 2 \rho t}}^2$ and $\psi(t):= \log \|P_t f \|_2^2$, so that $\phi(t) \geq \psi(0)$ for all $t \geq 0$. Since $\phi(0) = \psi(0) = \log \E[f^2]$, we deduce $\phi'(0) \geq \psi'(0)$. Computing both derivatives yields
\begin{equation} \label{eqn:lsi}
\rho(\E[f^2 \log(f^2)] - \E[f^2] \log \E[f^2]) \leq \Var(f) .
\end{equation}
where the expectations and variance are with respect to the law of $G$. We then let $u > 0$ be a small real number, and we let $f_u(k) := \sqrt{\frac{1-u}{1-p}}^k$ for all $k \in \N$. We have, when $u \rightarrow 0$,
\[
\E[f_u^2 \log(f_u^2)] - \E[f_u^2] \log \E[f_u^2] = 2p \log \left(\sqrt{\frac{1 - u}{1 - p}}\right)(1-u)u^{-2} - \frac{p}{u} \log(\frac{p}{u}) \sim \frac{p \log(\frac{1}{1-p})}{u^2} ,
\]
\[
\Var(f_u) = \frac{p}{u} - \frac{p^2}{(1 - \sqrt{(1 - p)(1 - u)})^2} \sim \frac{p}{u}
\]
Thus, when $u \rightarrow 0$,
\[
\frac{\E[f_u^2 \log(f_u^2)] - \E[f_u^2] \log \E[f_u^2]}{\Var(f_u)} \rightarrow + \infty ,
\]
which contradicts (\ref{eqn:lsi}). Therefore, the site-resampling semigroup is not hypercontractive for the geometric distribution. %To conclude this remark, we point out that Equation  (\ref{eqn:lsi}) is of the form
%\[
%\mathrm{Ent}(f^2) \leq C \E[ \|\nabla_f\|^2]
%\]
%where $\mathrm{Ent}$ and $\E[\|\nabla \cdot \|^2]$ are respectively the \textbf{entropy} and \textbf{energy} associated to the geometric-resampling semigroup. In the larger context of Markov reversible semigroups endowed with an invariant measure, it is possible to define an entropy and an energy, and it is said the Markov process satisfies a \textbf{log-Sobolev inequality} when the above is satisfied for a constant $C$ independent of $f$. Gross's theorem \cite{gross} establishes equivalence between satisfying a log-Sobolev inequality and the semigroup being hypercontractive. 
%\end{remark}

\section{Stationary last-passage percolation} \label{se:stationarylpp}

In this section we shall review the probabilistic approach to the study of exactly solvable models. We will introduce a stationary version of last-passage percolation, and expand on some of the properties arising from the coupling approach, that will be instrumental in subsequent sections. This approach dates back to work of Cator and Groeneboom~\cite{catgro06}, and Bal\'azs, Cator and Sepp\"al\"ainen~\cite{balcatsep06}, and allows one to compute an explicit formula for the so-called shape function and establish the $\chi=1/3$ and $\xi=2/3$ fluctuation exponents. 

We refer the reader to~\cite{sep09,sep18} for a more thorough introduction to last-passage percolation with geometric and exponential weights.
The most recent developments in this regard, that will be used here, come from~\cite{balbussep20,grojanras25}. The constructions we give below are analogous to those in~\cite{balbussep20,grojanras25}, and we shall refer to the latter for proofs of key statements. We avoid as much as possible treating these key statements as black boxes, in an attempt for Sections \ref{se:stationarylpp} and \ref{se:probageod} to be understandable without extensive prerequisites.
We start out by listing some of the notation that will recur in the next few sections.

\subsection{Some notation}

For points $u = (u_1, u_2)$ and $v = (v_1, v_2)$ in $\Z^2$ we write $u\le v$ if $u_1\leq v_1$ and $u_2\leq v_2$; we write $u<v$ if $u_1<v_1$ and $u_2<v_2$. Given $u\leq v$ we let $R_{u, v}$ be the rectangle $[u_1, v_1] \times [u_2, v_2]$. We let $\e_1=(1,0)$ and $\e_2=(0,1)$ denote the coordinate directions, and let
$$
\e_+ := \e_1+\e_2.
$$
For $u,v\in\Z^2$ we occasionally write $[u, v]$ for the line segment from $u$ to $v$, i.e.\ $[u,v]=\{tu+(1-t)v:t\in[0,1]\}$. We shall refer to the line segment $[0,n\e_+]$ as the {\bf main diagonal} of the square $R_{0, n \e_+}$ and to the line segment $[n\e_1, n \e_2]$ as its {\bf transverse diagonal}.
%Given vertices $u \leq v$, we will sometimes refer to the rectangle $R_{u+ \e_+, v}$ deprived of its left and bottom boundaries. 
For every $v = (v_1, v_2) \in \Z^2$, let
\begin{eqnarray*}
H_v \text{ be the horizontal line } \R \times \{v_2\}\\ 
\text{ and } V_v \text{ the vertical line }\{v_1\} \times \R.
\end{eqnarray*}
We denote by $|v| := \sqrt{v_1^2 + v_2^2}$ the Euclidean norm of $v$ and $|v|_1 := |v_1| + |v_2|$ its $\ell^1$ norm.

%We call a \textbf{directed path} a neighbour-to neighbour path of vertices in $\Z^2$ which goes either up or to the right at every step. 

%We call a path attaining the maximum in the definition of $T(u, v)$ a \textbf{geodesic} from $u$ to $v$. Since the geometric weights are integer-valued, there may be multiple geodesics, and we let $\pi(u, v)$ be the set of points which belong to at least one of the geodesic from $u$ to $v$, that is,
%\[
%   \pi(u, v) := \{x \in \Z^2 : \exists \text{ a geodesic } \gamma \text{ from } u \text{ to } v \text{ such that } x \in \gamma\}.
%\]

In the following, for each $\lambda\in(0,1)$, we introduce a collection of last-passage times $\{T(\lambda;x,y):x,y\in\Z^2\}$ where $T(\lambda;x,y)$ has the weights along the bottom-left boundary of the rectangle $R_{x,y}$ replaced by another set of weights. The parameter $\lambda$ indexing the boundary last-passage times is meant to correspond to the vector $\u_{\lambda} := \lambda \e_1 + (1 - \lambda) \e_2$, which we later interpret as a `preferred' or `characteristic' direction of the geodesics associated to the boundary travel times at parameter $\lambda$.

The construction below will first be carried out to obtain a collection of boundary travel times $T(\lambda;x,y)$ on the quarter plane $u+\Z_{\ge0}^2$, for some arbitrary reference vertex $u\in\Z^2$. It will thereafter be extended to all of $\Z^2$ through a limiting argument. Let us remark already that the full-plane extension of the boundary travel times is not essential for this paper. It is however convenient and will allow us to give a geometric interpretation of the boundary weights in terms of so-called Busemann functions, that have become an essential tool in the study of geodesics in planar growth models.

Finally, recall that we adopt the convention that $\N=\{0,1,2,\ldots\}$, and that a random variable $G$ has a {\bf geometric distribution} with parameter $p \in (0, 1)$ if
\[
\P(G = k) = p (1-p)^{k} \quad\text{for }k\in\N .
\]

\subsection{Construction of stationary last-passage times}

Let $u\in\Z^2$ and $\lambda\in(0,1)$ be fixed. Let $(\omega_x)_{x\in\Z^2}$ be i.i.d.\ geometrically distributed with parameter $p$, and let $T(x,y)$ be defined for $x\le y$ as before. For $p\in(0,1)$ and $\lambda\in(0,1)$ we define $q:(0,1)\to(0,p)$ as
$$
q(\lambda):=\frac{p\lambda+p\sqrt{(1-p)\lambda(1-\lambda)}}{1-p+p\lambda+2\sqrt{(1-p)\lambda(1-\lambda)}}.
$$
We note that $q(0)=0$, $q(1)=p$ and that $q(\lambda)$ is continuously differentiable on $(0,1)$ with positive derivative
\begin{equation}
q'(\lambda)=\frac{p(1-p)}{2\sqrt{(1-p)\lambda(1-\lambda)}\big(\sqrt{\lambda}+\sqrt{(1-p)\lambda(1-\lambda)}\big)^2}. \label{eqn: derivative of q(lambda)}
\end{equation}
Hence, $q(\lambda)$ is a bijection from $(0,1)$ to $(0,p)$. Finally, we set
$$
p_H(\lambda):=q(\lambda)\quad\text{and}\quad p_V(\lambda):=1-\frac{1-p}{1-q(\lambda)}.
$$
We note that $p_H(\lambda)$ is increasing (from $0$ to $p$) and $p_V(\lambda)$ is decreasing (from $p$ to $0$) in $\lambda$.
%In the construction of boundary weights, horizontal and vertical boundary weights will be geometrically distributed with parameters $p_H(\lambda)$ and $p_V(\lambda)$, respectively.
%Later, we shall interpret the parameter $\lambda$ as a `preferred' or `characteristic' direction of the boundary travel times $T(\lambda;x,y)$.

Next, we construct the set of boundary weights $\{\omega_x^H(\lambda),\omega_x^V(\lambda):x\in u+\Z_{\ge0}^2\}$ through an auxiliary process $L_x^\lambda$ defined for $x\in u+\Z_{\ge0}\times\Z$. Let $\{\omega_{u+j\e_2}^V(\lambda):j\in\Z\}$ be i.i.d.\ $\textrm{Geom}(p_V(\lambda))$ and independent of the bulk weights $(\omega_x)_{x\in\Z^2}$ that define $T(x,y)$. Set, for $j\in\Z$,
$$
L_\lambda(u)=0\quad\text{and}\quad L_\lambda(u+j\e_2)-L_\lambda(u+(j-1)\e_2)=\omega^V_{u+j\e_2}(\lambda).
$$
(Note that $L_\lambda(u+j\e_2)<0$ for $j<0$.) For $x\in u+\Z_{>0}\times\Z$ we set
\begin{equation}\label{eq:L_lambda}
L_\lambda(x):=\sup_{j\le x_2-u_2}\big[L_\lambda(u+j\e_2)+T(u+\e_1+j\e_2,x)\big].
\end{equation}
(The supremum is attained for finite $j$ since the weights along the boundary stochastically dominate the weights in the bulk.) Finally, for $x\in u+\Z_{>0}\times\Z$ we define boundary weights by
\begin{equation}\label{eq:boundary}
\begin{aligned}
\omega^H_x(\lambda)&:=L_\lambda(x)-L_\lambda(x-\e_1),\\
\omega^V_x(\lambda)&:=L_\lambda(x)-L_\lambda(x-\e_2),
\end{aligned}
\end{equation}
%obtain $T(\lambda;x,y)$ for $u\le x\le y$ via~\eqref{eq:T_lambda}. 
%Given a collection of boundary weights $\{\omega_x^H(\lambda),\omega_x^V(\lambda):x\in u+\Z_{\ge0}^2\}$, 
and the {\bf boundary travel time} or {\bf last-passage time} for $u\le x\le y$ as
\begin{equation}\label{eq:T_lambda}
T(\lambda;x,y):=\max\bigg\{\sum_{i=1}^k\omega^H_{x+i\e_1}(\lambda)+T(x+k\e_1+\e_2,y)\bigg\}\bigvee\max\bigg\{\sum_{j=1}^\ell\omega^V_{x+j\e_2}(\lambda)+T(x+\e_1+\ell\e_2,y)\bigg\}.
\end{equation}
We remark that we could equivalently define the boundary travel time by first defining the travel time of a directed path $\gamma$ from $x$ to $y$ as
\begin{equation}\label{eq:alternative T_lambda}
T(\lambda; \gamma) :=  \sum_{z \in H_x \cap \gamma \backslash \{x\} }\omg_z^H(\lambda) + \sum_{z \in V_x \cap \gamma \backslash \{x\} } \omg_z^V(\lambda) + \sum_{z \in R_{x + \e_+, y} \cap \gamma} \omg_z,
\end{equation}
and then maximising over all directed paths from $x$ to $y$.

As mentioned before, boundary travel times are here defined with respect to a base vertex $u$, but can be extended in a consistent way to the full plane. For this reason, $u$ is not explicit in the notation.

Several useful properties of the boundary travel times are straightforward from the construction. First, note that any path attaining the supremum in~\eqref{eq:L_lambda} must visit either $x-\e_1$ or $x-\e_2$, and hence that
$$
L_\lambda(x)=\max\{L_\lambda(x-\e_1)+\omega_x,L_\lambda(x-\e_2)+\omega_x\},
$$
which can be rewritten as the identity
\begin{equation}\label{eq:domination}
\omega_x=\min\{\omega_x^H(\lambda),\omega^V_x(\lambda)\}.
\end{equation}
An analogous argument shows, for $u\le x\le y$, that
$$
L_\lambda(y)=\max_{1\le i\le y_1-x_1}\big\{L_\lambda(x+i\e_1)+T(x+i\e_1+\e_2,y)\big\}\bigvee\max_{1\le j\le y_2-x_2}\big\{L_\lambda(x+j\e_2)+T(x+\e_1+j\e_2,y)\big\}.
$$
Subtracting $L_\lambda(x)$ from both sides gives, using~\eqref{eq:boundary} and~\eqref{eq:T_lambda}, that
$$
T(\lambda;x,y)=L_\lambda(y)-L_\lambda(x),
$$
and hence that the boundary travel times are additive, meaning that for $x\le y\le z$ in $u+\Z_{\ge0}^2$ we have
\begin{equation}\label{eq:additivity}
T(\lambda;x,z)=T(\lambda;x,y)+T(\lambda;y,z).
\end{equation}

More remarkable properties of the above construction can be obtained via techniques from queuing theory. Indeed, interpreting the boundary weights $(\omega^V_{u+j\e_2}(\lambda))_{j\in\Z}$ as inter-arrival times and the bulk weights $(\omega_{u+\e_1+j\e_2})_{j\in\Z}$ as service times, the weights $(\omega^V_{u+\e_1+j\e_2}(\lambda))_{j\in\Z}$ can be interpreted as inter-departure times and $(\omega^H_{u+\e_1+j\e_2}(\lambda))_{j\in\Z}$ as the sojourn times of a single server queue at equilibrium. This connection is by now classical, and we summarise the consequences thereof in the following theorem.

\begin{theorem}\label{thm:boundary1}
For every $p\in(0,1)$, $\lambda\in(0,1)$ and $u\in\Z^2$, we have for every $x\in u+\Z_{\ge0}^2$ that
\begin{enumerate}[label=(\roman*)]
    \item (stationarity) $\{T(\lambda;x,x+y):y\in\Z_{\ge0}^2\}\stackrel{d}{=}\{T(\lambda;u,u+y):y\in\Z_{\ge0}^2\}$;
    \item (distribution) $\omega^H_x(\lambda)\sim\textrm{Geom}(p_H(\lambda))$ and $\omega^V_x(\lambda)\sim\textrm{Geom}(p_V(\lambda))$;
    \item (independence) the collection $\{\omega^H_{x+i\e_1},\omega^V_{x+j\e_2}:i,j>0\}$ consists of mutually independent variables; the same holds for $\{\omega^H_{x-i\e_1},\omega^V_{x-j\e_2}:0\le i<x_1-u_1,0\le j<x_2-u_2\}$.
\end{enumerate}
\end{theorem}

\begin{proof}
We refer the reader to~\cite[Theorem~A.2]{grojanras25} for details; part~(i) is their (a.v), part~(ii) is their (d) and part~(iii) is their (c). Note that our notation corresponds to their notation in that $q(\lambda)=1-\overline{p}(\lambda,1-\lambda)$ with $r=1-p$.
\end{proof}

Additivity of the boundary travel times~\eqref{eq:additivity}, together with the distributional properties of the above theorem, show that
$$
\E[T(\lambda;0,x)]=x_1\E[\omega^H_0(\lambda)]+x_2\E[\omega^V_0(\lambda)]=x_1\frac{1-q(\lambda)}{q(\lambda)}+x_2\frac{1-p}{p-q(\lambda)}.
$$
From~\eqref{eq:domination} we further obtain that
$$
\E[T(0,x)]\le\inf_{\lambda\in(0,1)}\E[T(\lambda;0,x)]=\frac1p\Big[(1-p)(x_1+x_2)+2\sqrt{(1-p)x_1x_2}\Big].
$$
In the limit, one may in fact show that the above gives an equality, and thus gives an explicit expression for the shape function $\psi(x)$, in that
$$
\psi(x):=\lim_{n\to\infty}\frac1n\E[T(0,nx)]=\frac1p\Big[(1-p)(x_1+x_2)+2\sqrt{(1-p)x_1x_2}\Big].
$$

Just as in~\cite{balbussep20,grojanras25}, for the analysis below, we shall be required to consider boundary travel times for pairs of values of the parameter $\lambda$ as a time. For this purpose we shall extend the above construction into a simultaneous construction of boundary weights for two different values of $\lambda$. More precisely, given $0<\lambda<\lambda'<1$, let $(\omega^V_{u+j\e_2}(\lambda))_{j\in\Z}$ be as above. Next, let $(\tilde\omega_j)_{j\in\Z}$ be i.i.d.\ $\textrm{Geom}(p_V(\lambda'))$, and take $(\omega^V_{u+j\e_2}(\lambda'))_{j\in\Z}$ to be the inter-departure times of a single server queue with inter-arrival times given by $(\tilde\omega_j)_{j\in\Z}$ and service times given by $(\omega^V_{u+j\e_2}(\lambda))_{j\in\Z}$. It follows from standard result of queuing theory that $(\omega^V_{u+j\e_2}(\lambda'))_{j\in\Z}$ is i.i.d.\ $\textrm{Geom}(p_V(\lambda'))$, and hence that in the simultaneous construction the boundary travel times, for $\lambda$ and $\lambda'$, individually satisfy the properties of Theorem~\ref{thm:boundary1}. In addition to the above, the joint process satisfies a few additional properties that we summarise in the next theorem.

\begin{theorem}\label{thm:boundary2}
    For every $p\in(0,1)$, $0<\lambda<\lambda'<1$ and $u\in\Z^2$, we have for every $x\in u+\Z_{\ge0}^2$ that
\begin{enumerate}[label=(\roman*)]
    \item (stationarity) $\{(T(\lambda;x,x+y),T(\lambda';x,x+y)):y\in\Z_{\ge0}^2\}\stackrel{d}{=}\{(T(\lambda;u,u+y),T(\lambda';u,u+y)):y\in\Z_{\ge0}^2\}$;
    \item (monotonicity) $\omega_x\le\omega^H_x(\lambda')\le\omega^H_x(\lambda)$ and $\omega_x\leq\omega^V_x(\lambda)\le\omega^V_x(\lambda')$;
    \item (independence) $\{\omega^V_{x+j\e_2}(\lambda):j>0\}$ and $\{\omega^V_{x+j\e_2}(\lambda'):u_2-x_2<j\le0\}$ are independent; the same holds for $\{\omega^H_{x+i\e_1}(\lambda):i>0\}$ and $\{\omega^V_{x+i\e_1}(\lambda'):u_1-x_1<i\le0\}$.
\end{enumerate}
\end{theorem}

\begin{proof}
We refer the reader to~\cite[Theorem~A.2]{grojanras25} for details; part~(i) is their (a.v), part~(ii) is their (a.iv.2) and part~(iii) is their (b).
\end{proof}

\begin{remark}\label{rem:construction}
    For future reference, we also point out that it is immediate from construction that the boundary weights $\omega^H_x(\lambda)$ and $\omega^V_x(\lambda)$ are measurable with respect to the weights $\{\omega^V_{u+j\e_2}:j\le x_2-u_2\}$ and $\{\omega_y:y\le x\}$. Similarly, the coupled weights $\omega^H_x(\lambda)$, $\omega^V_x(\lambda)$ and $\omega^H_x(\lambda')$, $\omega^V_x(\lambda')$ are measurable with respect to the weights $\{\omega^V_{u+j\e_2}:j\le x_2-u_2\}$, $\{\tilde\omega_j:j\le x_2-u_2\}$ and $\{\omega_y:y\le x\}$.
\end{remark}

We proceed with a brief explanation of how to extend the above construction of boundary passage times on $u+\Z_{\ge0}^2$ to the full plane $\Z^2$. We reiterate that this extension is not essential for what is to come, but will allow us to talk about semi-infinite geodesics, and relate them to the boundary weights.

To extend the construction, we note that the configuration $\{(\omega_x,\omega^H_x,\omega^V_x):x\in u+\Z_{\ge0}^2\}$ induces, for every $u\in\Z^2$, a measure on $(\R\times\R\times\R)^{\Z^2}$, equipped with the product topology and Borel sigma algebra. As noted in~\cite{grojanras25}, taking $u=(-n,-n)$, and sending $n\to\infty$, results in a limiting measure $\P'$, due to the stationarity in Theorem~\ref{thm:boundary1}, which restricted to $u+\Z_{\ge0}^2$ is equal in distribution to the construction above, and in particular enjoys the properties stipulated in Theorem~\ref{thm:boundary1}. In a similar manner we may extend the simultaneous construction for two values $\lambda<\lambda'$.

Although the above extension is not essential, we shall henceforth assume it is in force, and for that reason make no mentioning of a base vertex $u$ with respect to which the boundary travel times are defined. Moreover, whenever we are considering boundary travel times for two different parameter values, we shall assume the coupling construction to be in force. We will continue to use the notation above.

\subsection{Geodesics for boundary travel times}

Recall that $\pi(u,v)$ denotes the set of vertices belonging to some geodesic for $T(u,v)$. Similarly, we let $\pi(\lambda;u,v)$ denote the set of vertices belonging to some geodesic for $T(\lambda;u,v)$, i.e.\
$$
   \pi(\lambda;u, v) := \{x \in \Z^2 : \exists \text{ a $\lambda$-geodesic } \gamma \text{ from } u \text{ to } v \text{ such that } x \in \gamma\},
$$
where we by $\lambda$-geodesic refer to a path $\gamma$ satisfying $T(\lambda; \gamma) = T(\lambda; u, v)$, where $T(\lambda; \gamma)$ is as in (\ref{eq:alternative T_lambda}).

In order to be able to talk about an `upmost' and `downmost' geodesic between two points, we need a notion of order between paths. Given two (finite, directed) paths $\gamma$ and $\gamma'$, we say that $\gamma$ is \textbf{above} $\gamma'$ if and only if for every vertical line $V$ intersecting both $\gamma$ and $\gamma'$, the minimum of $\gamma \cap V$ is larger or equal to the minimum of $\gamma' \cap V$ (these minima exist because $V$ is a fully ordered set and paths are finite); see Figure~\ref{fig:abovedef}.
\begin{figure}[htbp]
\begin{center}
\includegraphics[width = 0.5 \linewidth]{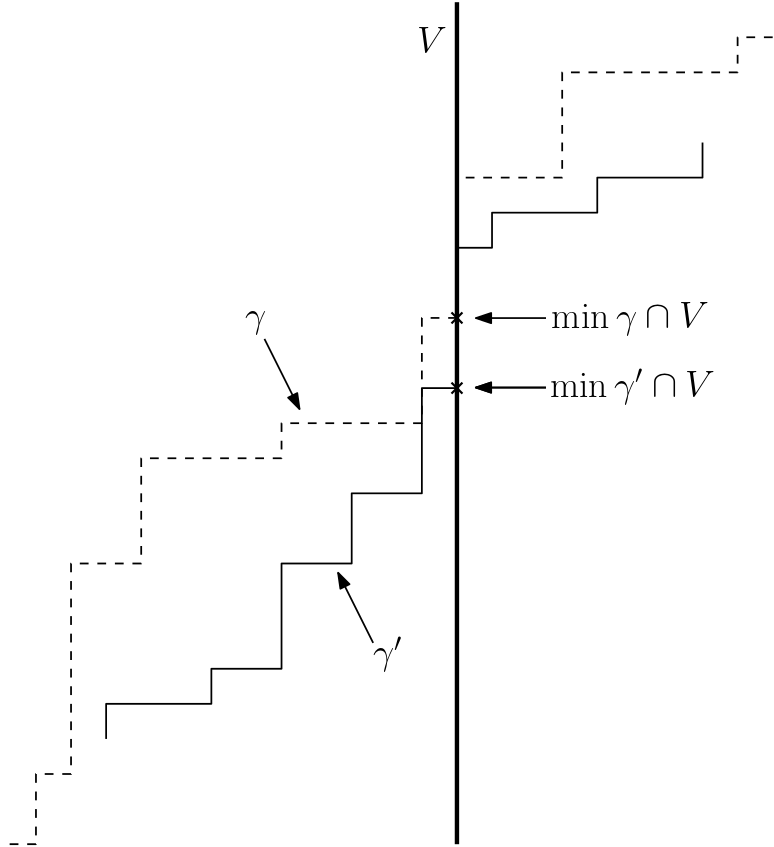}
\caption{The directed path $\gamma$ is above $\gamma'$.} \label{fig:abovedef}
\end{center}
\end{figure}
Equivalently, for every horizontal line $H$ intersecting both paths, the maximum of $\gamma \cap H$ is lesser or equal to the maximum of $\gamma' \cap H$.

For $u \leq v \in \Z^2$, we now define the \textbf{upmost} geodesic from $u$ to $v$ to be the geodesic from $u$ to $v$ that lies above all other, and denote this by $\pi_+(u, v)$ and $\pi_+(\lambda;u,v)$, depending on whether we refer to geodesics with respect to $T(u,v)$ or $T(\lambda;u,v)$. We define the \textbf{downmost} geodesic between $u$ and $v$ analogously, and use the notation $\pi_-(u, v)$ and $\pi_-(\lambda;u,v)$. 

From~\eqref{eq:domination} we note that boundary weights dominate bulk weights, and it may therefore be favourable for the geodesic associated to $T(\lambda;u,v)$ to remain longer at the lower-left boundary of $R_{u,v}$ than the geodesic associated to $T(u,v)$. This will provide us with a comparison between geodesics for travel times with or without boundary. Let $E_{\rightarrow}(\lambda; u, v)$ be the event that the first step of $\pi_-(\lambda; u, v)$ is to the right and $E_{\uparrow}(\lambda; u, v)$ that the first step of $\pi_+(\lambda; u, v)$ is up, so
\begin{align*}
E_{\rightarrow}(\lambda; u, v) &:= \{u+\e_1 \in \pi_-(\lambda;u, v) \} , \\
E_{\uparrow}(\lambda; u, v) &:= \{u+ \e_2 \in \pi_+(\lambda;u, v)\}.
\end{align*}
We then have the following lemma establishing monotonicity in the positions of geodesics with respect to $\lambda$-geodesics provided $E_{\rightarrow}(\lambda; u, v)$ or $E_{\uparrow}(\lambda; u, v)$ is satisfied.

\begin{lemma} \label{monoofgeod}
For all $u \le v \in \Z^2$ and $\lambda \in (0, 1)$, on the event $E_{\rightarrow}(\lambda; u, v)$ the path $\pi_-(u, v)$ is above $\pi_-(\lambda; u, v)$, and on $E_{\uparrow}(\lambda; u, v)$ the path $\pi_+(\lambda; u, v)$ is above $\pi_+(u, v)$.
\end{lemma}

\begin{proof}
The two statements of the proposition are equivalent by symmetry, so we only prove the first one. Let $u \leq v \in \Z^2$ and let $\lambda \in (0, 1)$.
%Let $\pi_-$ and $\pi_-(\lambda)$ respectively denote $\pi_-(u, v)$ and $\pi_-(\lambda; u, v)$.
Let $k$ be the largest integer such that $u+ k \e_1 \in \pi_-(u,v)$ and let $k_{\lambda}$ be the largest integer such that $u + k_{\lambda} \e_1 \in \pi_-(\lambda; u, v)$. (These random integers $k$ and $k_{\lambda}$ are in this context often referred to as `exit times'.) We prove two things: on $E_{\rightarrow}(\lambda; u, v)$ we have $k \leq k_{\lambda}$, and if the geodesics $\pi_-(u,v)$ and $\pi_-(\lambda;u,v)$ intersect in the bulk, then they coincide after their first point of intersection.

We first prove $k \leq k_{\lambda}$ on $E_{\rightarrow}(\lambda; u, v)$. It is clear if $u_2 = v_2$, because the geodesics are then both the horizontal line from $u$ to $v$. It is also clear if $k = 0$, so we can assume that $u_2 < v_2$ and $k > 0$. For $j \geq 0$, let $\pi_j$ be the downmost geodesic from $u+j\e_1+\e_2$ to $v$. We then have that $\pi_-(u,v)$ is the concatenation of $(u, u+\e_1, \dots, u+k \e_1)$ and $\pi_k$, and $k$ is the largest maximizer of $\{\sum_{i = 1}^j \omg_{u+i \e_1} + T(\pi_j): j \geq 1\}$. Thus, for any $j \in \{1, \dots, k\}$,
\[
\sum_{i=1}^j \omg_{u+i\e_1} + T(\pi_j) \leq \sum_{i=1}^k \omg_{u+i\e_1} + T(\pi_k)  .
\]
By~\eqref{eq:domination} we have $\omega_{u+i\e_1}\le\omega^H_{u+i\e_1}$ for all $i$. Applying this equality for $j+1\le i \le k$ in the above, we deduce for all $1 \leq j \leq k$
\begin{equation} \label{jlessk}
\sum_{i=1}^j \omg_{u+i\e_1}^H + T(\pi_j) \leq \sum_{i=1}^k \omg_{u+i\e_1}^H + T(\pi_k) .
\end{equation}
Moreover, on the event $E_{\rightarrow}(\lambda; u, v)$, the $\lambda$-geodesic $\pi_-(\lambda;u,v)$ does not pick up weight on the left boundary, so that $\pi_-(\lambda;u,v)$ is the concatenation of $(u, u+ \e_1, \dots, u+ k_{\lambda} \e_1)$ and $\pi_{k_{\lambda}}$, and $k_{\lambda}$ is the largest maximizer of $\{\sum_{i = 1}^j \omg_{u+i \e_1}^H +T(\pi_j): j > 0 \}$. Therefore, as $k > k_{\lambda}$ would contradict~\eqref{jlessk}, we deduce $k \leq k_{\lambda}$ on $E_{\rightarrow}(\lambda; u, v)$.

It remains to prove that if $\pi_-(u,v)$ and $\pi_-(\lambda;u,v)$ intersect outside of the boundary, they coincide from that point onwards. Assume $\pi_-(u,v)$ visits $x \in R_{u+ \e_+, v}$. By maximality of $\pi_-(u,v)$ and the fact that it is downmost, we have $\pi_-(u,v) \cap R_{x, v} = \pi_-(x, v)$. Similarly, if $\pi_-(\lambda;u,v)$ visits $x \in R_{u+ \e_+, v}$, since the restriction of $\pi_-(\lambda;u,v)$ to $R_{x, v}$ is maximal for $T(x,v)$, as well as downmost, we have $\pi_-(\lambda;u,v) \cap R_{x, v} = \pi_-(x, v)$. Therefore, if $\pi_-(u,v)$ and $\pi_-(\lambda;u,v)$ visit the same $x \in R_{u + \e_+, v}$, then they coincide after $x$.

This concludes the proof as $k \leq k_{\lambda}$ means that $\pi_-(u,v)$ is above $\pi_-(\lambda;u,v)$ until these geodesics leave the boundary, and the coalescence property implies that $\pi_-(u,v)$ cannot visit a point strictly below $\pi_-(\lambda;u,v)$, inside the rectangle $R_{u + \e_+, v}$.
\end{proof}

The next lemma establishes that $\lambda$-geodesics ending in the same point are consistent with each other in the sense that changing the starting point does not change the geodesic in the region where both geodesics are defined. This is a rather straightforward consequence of the additivity of the boundary travel times.

\begin{lemma} \label{ctyofgeod}
Consider $u, u', v$ in $\Z^2$ such that $u, u' < v$, and let $\lambda \in (0, 1)$. Then $\pi(\lambda; u, v)$ and $\pi(\lambda; u', v)$ coincide in $R_{u+\e_+, v} \cap R_{u'+\e_+, v}$.
\end{lemma}

\begin{proof}
Fix $u, u'<v$ and $\lambda\in(0,1)$. Let $y$ denote the bottom-left corner of the rectangle $R_{u, v} \cap R_{u', v}$.
%By item (v) of Theorem \ref{bdarymodels}, we can write $T(\lambda; u, y+ i \e_1) - T(\lambda; u, y + (i - 1) \e_1) = \omg_{y + i \e_1}^H(\lambda)$  for all $i \geq 1$, and similarly with $u'$ instead of $u$. Hence, for all $k \geq 1$,
%\begin{align*}
%T(\lambda; u, y+ k \e_1) - T(\lambda; u, y) &=  \sum_{i = 1}^k \omg_{y + i \e_1}^H(\lambda) \\
%&= T(\lambda; u', y+ k \e_1) - T(\lambda; u', y)  .
%\end{align*}
%By the same argument, with $y + k \e_2$ instead of $y+ k\e_1$, we deduce that for all $x$ in $R_{y, v} \backslash R_{y + \e_+, v}$ (the bottom-left boundary of $R_{y, w}$),
Let $x\in R_{y,v}$. Using additivity of the boundary travel times~\eqref{eq:additivity} we find that
$$
T(\lambda;u,x)=T(\lambda;u,y)+T(\lambda;y,x)
$$
and that
$$
T(\lambda;u',x)=T(\lambda;u',y)+T(\lambda;y,x).
$$
Combining the two equations leads to the identity, for all $x\in R_{y,v}$,
\begin{equation}\label{eq:diff}
T(\lambda; u, x) - T(\lambda; u, y) = T(\lambda; u', x) - T(\lambda; u', y).
\end{equation}

For $x\neq y$ on the bottom-left boundary of $R_{y, v}$, let $\tilde{x}$ be the neighbor of $x$ that is inside $R_{y + \e_+, v}$. Now, let $\gamma$ be a $\lambda$-geodesic from $u$ to $v$. This geodesic visits the bottom-left boundary of $R_{y, v}$, and leaves it at a point $x_0$, so that $T(\lambda; \gamma) = T(\lambda; u, x_0) + T(\tilde x_0, v)$. Moreover, since $\gamma$ is a geodesic,
\[
x_0 \in \argmax \{T(\lambda; u, x) + T(\tilde{x}, v): x \in R_{y, v} \backslash R_{y + \e_+, v},x\neq y\}  .
\]
By~\eqref{eq:diff}, note that $T(\lambda; u, x)$ equals $T(\lambda; u, y) - T( \lambda; u', y) + T(\lambda; u', x)$ for $x$ in $R_{y, v} \backslash R_{y + \e_+, v}$, and since the difference $T(\lambda; u, y) - T( \lambda; u', y)$ does not depend on $x$, we have
\begin{align*}
x_0 \in & \, \argmax \{T(\lambda; u, y) - T( \lambda; u', y) + T(\lambda; u', x) + T(\tilde{x}, v): x \in R_{y, v} \backslash R_{y + \e_+, v},x\neq y\} \\
&= \argmax \{T(\lambda; u', x) + T(\tilde{x}, v): x \in R_{y, v} \backslash R_{y + \e_+, v}, x\neq y\}.
\end{align*}
We deduce that there exists a $\lambda$-geodesic $\gamma'$ from $u'$ to $v$ going through $x_0$ and $\tilde x_0$, and by replacing the segment of $\gamma'$ from $\tilde x_0$ to $v$ by the corresponding segment of $\gamma$  if necessary (which does not change the fact that it is a geodesic, since both paths must be geodesics from $\tilde{x}_0$ to $v$), we can assume that $\gamma$ and $\gamma'$ coincide in $R_{y + \e_+, v}$. This concludes the proof as any element of $\pi(\lambda; u, v) \cap R_{y + \e_+, v}$ thus belongs to $\pi(\lambda; u', v) \cap R_{y + \e_+, v}$, and vice versa.
\end{proof}

The backwards compatibility of geodesics, established in the previous lemma, allows us to define the \textbf{infinite backwards} $\lambda$-geodesic $\pi(\lambda; -\infty, v)$ as the union
\[
\pi(\lambda; - \infty, v) := \bigcup_{\underset{u < v}{u \in \Z^2:}} \big(\pi(\lambda; u, v) \cap R_{u+ \e_+, v}\big).
\]
The above union is increasing due to Lemma \ref{ctyofgeod}. We also let $\pi_+(\lambda; - \infty, v)$ and $\pi_-(\lambda; -\infty, v)$ denote the upmost and downmost infinite backwards $\lambda$-geodesics. This definition of infinite backwards geodesics is convenient to describe how $\lambda$-geodesics behave geometrically. The next theorem connects the $\lambda$-geodesics to the direction $\mathbf{u}_{\lambda} = \lambda \e_1 + (1 - \lambda) \e_2$. Moreover, it gives a precise quantitative bound on how unlikely it is that a $\lambda$-geodesic deviates significantly from that direction; see also Figure~\ref{fig:backgeodlargedev}. Given $x = (x_1, x_2) \in \R^2$ we let $\lfloor x \rfloor$ denote $(\lfloor x_1 \rfloor , \lfloor x_2 \rfloor)$, the couple of floors of the coordinates of $x$.
\begin{figure}[htbp]
\begin{center}
\includegraphics[width = 0.6 \linewidth]{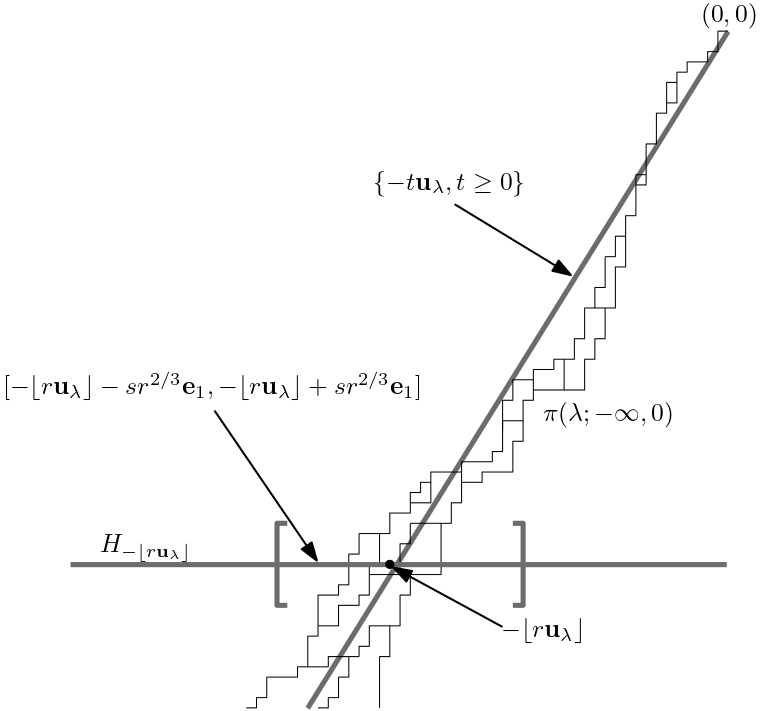}
\caption{Illustration of Theorem \ref{backgeodlargedev}: with probability at least $1 - e^{-c s^3}$, the intersections of $\pi(\lambda; - \infty, 0)$ with $H_{- \lfloor r \u_{\lambda} \rfloor}$ are at distance less than $s r^{2/3}$ from $-r \mathbf{u}_{\lambda}$.}
\label{fig:backgeodlargedev}
\end{center}
\end{figure}

\begin{theorem} \label{backgeodlargedev}
For every $\delta \in (0, \frac{1}{2})$, there exist constants $c, C, r_0 >0$ such that for all $\lambda \in (\delta, 1-\delta)$, $r \geq r_0$ and all $s \in [0, r^{1/3}]$
\[
\P \Big(\pi(\lambda; -\infty, 0)\cap H_{- \lfloor r \u_{\lambda} \rfloor}   \subseteq \big[ - \lfloor r \u_{\lambda} \rfloor - s r^{2/3} \e_1 , - \lfloor r \u_{\lambda} \rfloor + s r^{2/3} \e_1 \big]  \Big)
\geq 1 - C \exp(-c s^3).
\]
\end{theorem}

\begin{proof}
The theorem is a slight reformulation of Theorem~B.1 in~\cite{grojanras25}. Given $\lambda\in(0,1)$ and $u<v$, let
\begin{align*}
Z_H(\lambda;u,v)&:=\max\{j\ge0:u+j\e_1\in\pi_-(\lambda;u,v)\},\\
Z_V(\lambda;u,v)&:=\max\{j\ge0:u+j\e_2\in\pi_+(\lambda;u,v)\}.
\end{align*}
Theorem~B.1 in~\cite{grojanras25} states that for every $\delta\in(0,1-p)$ there exist constants $c_0$, $r_0$ and $s_0$ such that for all $\lambda\in(\delta,1-\delta)$ and $v\in\Z_{\ge0}^2$ satisfying $|v|_1\ge r_0$ and $|\lambda-v_1/|v|_1|\le |v|^{-1/3}$ we have for $s\ge s_0$ that
\begin{equation}\label{eq:exit}
\P \big(\{Z_H(\lambda;0,v) \geq s |v|_1^{2/3} \} \cup \{ Z_V(\lambda;0,v) \geq s |v|_1^{2/3} \} \big) \leq\exp(-c_0 s^3).
\end{equation}
Key to proving this bound comes from being able to compute the moment generating function of the stationary travel times; see~\cite{grojanras25}. We here show how the announced theorem follows from~\eqref{eq:exit}.

It suffices to verify the statement for $\delta\in(0,1/2)$ small enough, so we may assume $\delta<1-p$. Let $c_0=c_0(\delta)$, $r_0=r_0(\delta)$ and $s_0=s_0(\delta)$ be as stated above. Let $y=-\lfloor r\u_\lambda\rfloor-\lfloor sr^{2/3}\rfloor\e_1$ and let $z$ denote the point in $V_y$ closest to the line $\{t\u_\lambda:t\in\R\}$. We first claim that on the event that $Z_H(\lambda;-\lfloor r\u_\lambda\rfloor,0)<sr^{2/3}$ and $Z_V(\lambda;z,0)<(\delta/2) sr^{2/3}$, we have that
\begin{equation}\label{eq:containment}
\pi(\lambda; -\infty, 0) \cap H_{- \lfloor r \u_{\lambda} \rfloor} \subseteq \big[- \lfloor r \u_{\lambda} \rfloor - s r^{2/3} \e_1, - \lfloor r \u_{\lambda} \rfloor+s r^{2/3} \e_1\big].
\end{equation}

To see this, note that the definition of $\pi(\lambda;-\infty,0)$ and Lemma~\ref{ctyofgeod}, if $Z_H(\lambda;-\lfloor r\u_\lambda\rfloor,0)<sr^{2/3}$ then any $\lambda$-geodesic leaves the line $H_{-\lfloor r\u_\lambda\rfloor}$ to the left of $- \lfloor r \u_{\lambda} \rfloor + s r^{2/3} \e_1$. Similarly, since the point $z$ lies at least at distance $(\delta/2) sr^{2/3}$ below the line $H_{-\lfloor r\u_\lambda\rfloor}$, if $Z_V(\lambda;z,0)<(\delta/2) sr^{2/3}$ then it follows that any $\lambda$-geodesic leaves the vertical line $V_y$ before it reaches the horizontal line $H_{-\lfloor r\u_\lambda\rfloor}$, which it therefore must enter to the right of $- \lfloor r \u_{\lambda} \rfloor - s r^{2/3} \e_1$. In conclusion,~\eqref{eq:containment} holds.

We now apply~\eqref{eq:exit} twice, first for $v=\lfloor r\u_\lambda\rfloor$ and then for $v=-z$. We note that in both cases we have $|v|_1\ge r_0$ and $|\lambda-v_1/|v|_1|\le1/r<|v|^{1/3}$, as long as $r\ge r_0$ is not too small. We conclude, from~\eqref{eq:exit}, that for $s,s'\ge s_0$ we have
$$
\P\big(Z_H(\lambda;-\lfloor r\u_\lambda\rfloor,0)\ge sr^{2/3}\big)\le\exp(-c_0s^3)\quad\text{and}\quad\P\big(Z_V(\lambda;z,0)\ge(\delta/2) sr^{2/3}\big)\le\exp(-c_0(s')^3),
$$
where $s'$ is defined such that $s'|z|_1^{2/3}=(\delta/2)sr^{2/3}$. Since $|z|_1\approx r+\frac1\lambda sr^{2/3}$ we obtain for $sr^{2/3}\le r$ that
$$
s'\approx\frac{\delta/2}{(1+\frac1\lambda sr^{-1/3})^{2/3}}s\ge \left(\frac\delta2\right)^{5/3} s \geq \frac{\delta^2}4 s.
$$
Hence, taking $c=c_0\cdot(\delta^2/4)^3$, we find that~\eqref{eq:containment} holds with probability at least $1-2\exp(-cs^3)$, for $s\ge (4/\delta^2) s_0$. Finally, we may remove the restriction on $s$ by setting $C=\max\{2,\exp(c (4 s_0 /\delta^2)^3 )\}$.
\end{proof}

\subsection{Boundary weights and Busemann functions}

As an aside, which will not be used further in this paper, we end this section by interpreting the boundary weights introduced above in terms of so-called Busemann functions. We refer the reader to the survey~\cite{rassoul18} for further references to the literature.

Given a semi-infinite geodesic $\pi=(v_0,v_1,\ldots)$, directed `backwards' in that $v_0\ge v_1\ge...$, we define the Busemann function $B_\pi:\Z^2\times\Z^2\to\R$ of $\pi$ as the limit
$$
B_\pi(x,y):=\lim_{k\to\infty}[T(v_k,x)-T(v_k,y)].
$$
(The limit exists almost surely for all $x$ and $y$ in $\Z^2$.)

It follows from a well-known compactness argument that for every $\lambda\in(0,1)$ and every $x\in\Z^2$ there exists an almost surely unique upmost backwards-infinite geodesic $\gamma_\lambda(x)$ with asymptotic direction $-\u_\lambda$. Moreover, for every $x$ and $y$, the geodesics $\gamma_\lambda(x)$ and $\gamma_\lambda(y)$ coalesce almost surely. Indeed, Theorem~\ref{backgeodlargedev} says that $\pi_+(\lambda;-\infty,x)$ is the upmost geodesic in direction $-\u_\lambda$.

By definitions of boundary weights~\eqref{eq:boundary} and the additivity of boundary travel times~\eqref{eq:additivity}, we have that
$$
\omega^V_x(\lambda)=T(\lambda;x,x-\e_2)=T(\lambda;u,x)-T(\lambda;u,x-\e_2).
$$
Since geodesics coalesce, we may for every $x$ find a point $z\in\pi_+(\lambda;-\infty,x)\cap\pi_+(\lambda;-\infty,x-\e_2)$. Due to Lemma~\ref{ctyofgeod}, the finite geodesics from $u<z$ coincide with the infinite geodesics on $R_{u+\e_+,x}$, so that $z\in\pi_+(\lambda;u,x)\cap\pi_+(\lambda;u,x-\e_2)$. In particular,
\begin{align*}
    T(\lambda;u,x)&=T(\lambda;u,z)+T(z,x),\\
    T(\lambda;u,x-\e_2)&=T(\lambda;u,z)+T(z,x-\e_2),
\end{align*}
and hence
$$
\omega_x^V(\lambda)=T(z,x)-T(z,x-\e_2)=B_{\pi(\lambda;-\infty,0)}(x,x-\e_2).
$$
A similar argument shows that $\omega^H_x(\lambda)=B_{\pi(\lambda;-\infty,0)}(x,x-\e_1)$.

In conclusion, it is possible to construct the boundary weights from Busemann functions, given the existence of a unique geodesic in direction $-\u_\lambda$. Indeed, this gives a simultaneous construction of the boundary weights for all $\lambda\in(0,1)$. However, this construction does not identify the distributional properties of Theorems~\ref{thm:boundary1} and~\ref{thm:boundary2}, for which the queuing theory seems to be essential.

\section{Bounds on the probability of being on the geodesic} \label{se:probageod}

In this section, we prove upper-bounds on the probability that a vertex $v \in R_{0, n \e_+}$ belongs to the set $\pi(0, n \e_+)$ of vertices belonging to some geodesic between the origin and $n\e_+$. We will prove two bounds, that together amount to Theorem~\ref{th:probageod}. The first bound is a large deviation estimate on transversal fluctuations, which we derive from Theorem~\ref{backgeodlargedev}. Since the transversal fluctuation exponent is $2/3$, this bound will be nontrivial only for vertices at least this distance from the main diagonal.

The second bound provides a precise bound on the probability of being on the geodesic for vertices within distance $n^{2/3}$ of the main diagonal. We may expect that points within distance $n^{2/3}$ from the diagonal are roughly equally likely to be visited by the geodesic, resulting in a heuristic upper bound of order $n^{-2/3}$. The second bound verifies this with a logarithmic correction.

The bound for vertices `far' from the main diagonal we derive from Theorem~\ref{backgeodlargedev}, which in turn was deduced from precise bounds on so-called exit times in stationary last-passage percolation obtained in~\cite{grojanras25}. The bound for vertices `close' to the main diagonal will follow an approach from~\cite{balbussep20}, also employed in~\cite{grojanras25}, which reduced the event of being on the geodesic to an estimate of a random walk being non-negative for order $n^{2/3}$ steps.

In both cases a notable difference from the above mentioned works is our bounds hold for all vertices in the square $R_{0,n\e_+}$, and not only for vertices at a macroscopic distance from the boundary.

\subsection{Vertices `far' from the diagonal}

Given a directed path $\gamma$ and a vertex $v$, we say that $\gamma$ is {\bf above} $v$ if every point in $\gamma \cap V_{v}$ is above or equal to $v$. Similarly, $\gamma$ is said to be {\bf below} $v$ if every point in $\gamma\cap V_v$ is below or equal to $v$.

\begin{theorem} \label{boundoutsidebulk}
There exists $0<c, C <\infty$ such that for all $n \in \N$ and all $v\in R_{0, n \e_+}$ such that $|v|_1 \leq n$ and $v_1 \geq v_2$ we have
\[
\P \big( \pi(0, n \e_+) \text{ is above } v\big) \geq 1 - C \exp \Big(-c \frac{(v_1 - v_2)^3}{|v|_1^2} \Big).
\]
\end{theorem}

Note that, by symmetry with respect to the main diagonal, this translates to a similar bound for $v$ such that $v_1 \leq v_2$, but with `below' instead of `above'. Moreover, by symmetry with respect to the transverse diagonal, the theorem translates into a bound for $v\in R_{0,n\e_+}$ such that $|v|_1 \geq n$, with $|v|_1$ replaced with $2n - |v|_1$. 

The idea behind the proof of Theorem~\ref{boundoutsidebulk} is as follows: we compare the geodesics between the corners of the square to the backwards-infinite geodesic issued from $n \e_+$ in direction $- \u_{\lambda}$, for carefully chosen $\lambda$, and rely on Theorem~\ref{backgeodlargedev}. For a vertex $v$ `far' from the diagonal we choose $\lambda$ so that the $\lambda$-geodesic is likely to intersect $V_v$ above $v$ and to intersect the horizontal axis to the right of the origin. On this event we rely on Lemmas~\ref{monoofgeod} and~\ref{ctyofgeod} to control $\pi(0,n\e_+)$ in terms of the $\lambda$-geodesic.

%The idea of the proof of Theorem \ref{boundoutsidebulk} is to rely on the properties of boundary models in order to control the relative positions of the geodesic set $\pi(- n \e_+, 0)$, and a vertex $v$ in the right quadrant of the square $R_{-n \e_+, 0}$. The main tool we use is Theorem \ref{backgeodlargedev}, indicating that backwards $\lambda$-geodesics have very low probability of stepping far from direction $\mathbf{u}_{\lambda}$, for a reasonable choice of $\lambda$. We thus introduce a point $v'$, in-between $v$ and the main diagonal, and claim the $\lambda$-geodesic directed from $0$ to $v'$ has quantitatively high probability of passing above $v$ and to the right of $0$. We then rely on Lemmas \ref{monoofgeod} and \ref{ctyofgeod} to justify that the downmost `true' geodesic $\pi_-(- n\e_+, 0)$ is above the downmost $\lambda$-geodesic $\pi_-(\lambda; - n \e_+, 0)$, therefore above $v$, with high probability. Recall that when comparing a directed path $\gamma$ with a vertex $v$, we say that $\gamma$ is above $v$ if $\gamma \cap H_{v}$ is to the left or on $v$ (equivalently, $\gamma \cap V_v$ is above or on $v$).

\begin{proof}[Proof of Theorem \ref{boundoutsidebulk}]
In order to facilitate the comparison with infinite backwards geodesics, we consider instead the square $R_{- n \e_+, 0}$ for $n \in \N$. By symmetry with respect to the line $\{(x,y):y=-x\}$, and the fact that this symmetry preserves being above or below for directed paths, for all $n \in \N$ and $v=(v_1,v_2) \in R_{0, n \e_+}$ we have
\[
\P \big( \pi(0, n \e_+) \text{ is above } v \big) = \P \big(\pi(-n \e_+, 0) \text{ is above } -(v_2,v_1) \big)  .
\]
As a consequence, we may reformulate the statement we set out to prove as there exists $c, C > 0$ such that for all $n \in \N$ and $v \in R_{- n \e_+, 0}$ with $|v|_1 \leq n$ and $0\ge v_1 \geq v_2$,
\[
\P \big(\pi(- n \e_+, 0) \text{ is above }v \big) \geq 1 - C \exp \left( -c \frac{|v_2 - v_1|^3}{|v|_1^2} \right)  .
\]
So, let $n \in \N$ and $v\in R_{-n \e_+, 0}$ be such that $|v|_1 \leq n$ and $0\ge v_1 \geq v_2$, so that $v$ is in the `right quadrant' of $R_{-n\e_+,0}$ corresponding to the shaded region in Figure~\ref{fig:geodabovev} (left). We may further assume that $v_1-v_2\ge 8$, as the statement is otherwise trivial assuming that $C$ is sufficiently large and $c$ sufficiently small.

We let $v'$ denote the point at half-distance to the main diagonal from $v$, we take $\lambda$ such that $v' = - |v'|_1 \u_{\lambda}$ and $v''$ such that $v''_2 = -n$ and $v'' = - |v''|_1 \u_{\lambda}$; see Figure~\ref{fig:geodabovev}. We can then express $v'$, $\lambda$ and $v''$ explicitly as
\begin{align*}
v' &= \frac{3 v_1 + v_2}{4} \e_1 + \frac{v_1 + 3 v_2}{4}\e_2  ,\\
\lambda &= \frac{3v_1 + v_2}{4 v_1 + 4v_2} = \frac{1}{2} - \frac{v_1 - v_2}{4 |v|_1}  ,\\
v'' &= -\frac{3v_1 + v_2}{v_1 + 3v_2} n \e_1 - n \e_2  .
\end{align*}

\begin{figure}
    \begin{minipage}{.48 \textwidth}
    \centering
	        \includegraphics[width = 0.9\textwidth]{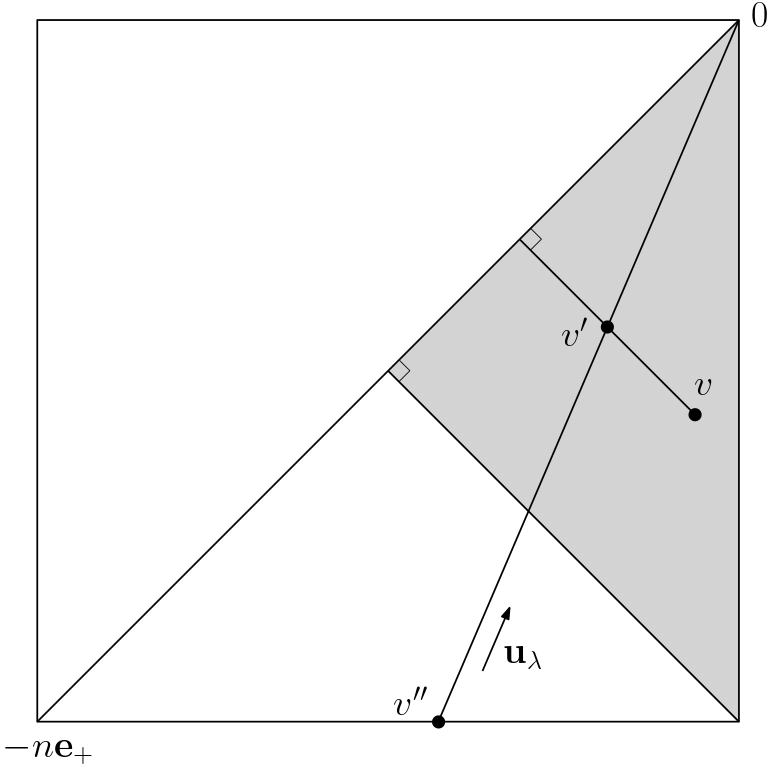}
    \end{minipage}
    \begin{minipage}{.48 \textwidth}
    \centering
	        \includegraphics[width = 0.9\textwidth]{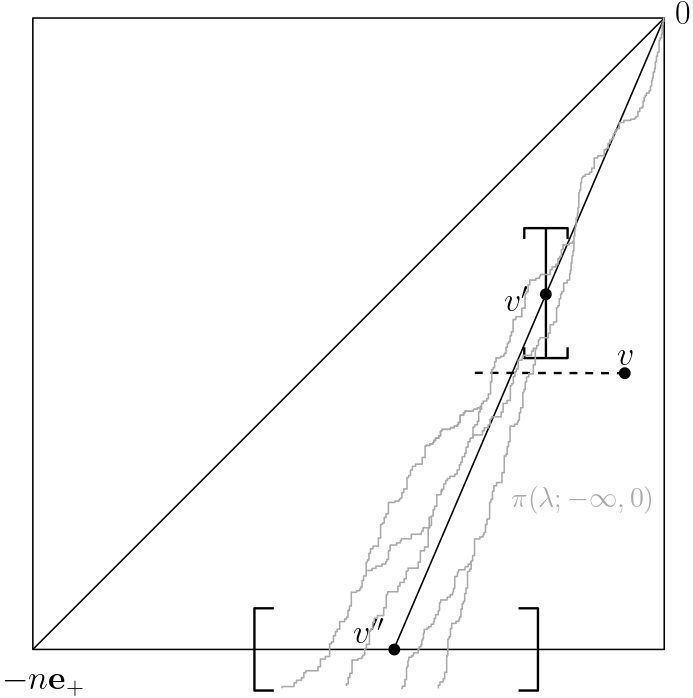}
    \end{minipage}
    \caption{Illustration of $v'$ and $v''$ and the shaded right quadrant of $R_{-n \e_+, 0}$ (left). Illustration of the intervals around $v'$ and $v''$ which $\pi(\lambda; - \infty, 0)$ goes through with high probability (right).}
    \label{fig:geodabovev}
\end{figure}

Since $v$ lies below the main diagonal, and $v'$ half-way between $v$ and the diagonal, we have $\lambda \in [\frac{1}{4}, \frac{1}{2}]$. We now apply Theorem~\ref{backgeodlargedev} twice, with $\delta=1/5$, and obtain $c, C, r_0 > 0$ such that for $|v|_1\ge r_0$ and all $0\le s\le 8|v|_1^{1/3}$ we have, with $c'=c/8^3$, that
%We define $c'$ after $c$ as $c' := c / 8^3$. We have $1 - \delta > \frac{1}{2} \geq \lambda \geq \frac{1}{4} > \delta$, $\frac{s}{8} \leq s \leq |v|_1^{1/3}$, so that if $|v|_1 \geq r_0$, the conclusion of Theorem \ref{backgeodlargedev} yields
\begin{equation} \label{vertdev}
\P \Big( \pi( \lambda; - \infty, 0)\cap V_{\lfloor v' \rfloor} \subseteq \Big[ \lfloor v' \rfloor - \frac{s}{8} |v'|_1^{2/3} \e_2,  \lfloor v' \rfloor + \frac{s}{8} |v'|_1 ^{2/3} \e_2\Big] \Big) \geq 1 - C e^{-c's^3}  ,
\end{equation}
%We apply once again the conclusion of Theorem \ref{backgeodlargedev} at $v''$, obtaining
and
\begin{equation} \label{horizdev}
\P \Big( \pi(\lambda; - \infty, 0)\cap H_{- n \e_+}  \subseteq \Big[ \lfloor v'' \rfloor - \frac{s}{8} |v''|_1^{2/3} \e_1, \lfloor v'' \rfloor + \frac{s}{8} |v''|_1^{2/3} \e_1 \Big] \Big) \geq 1 - C e^{-c' s^3}  .
\end{equation}

We now specify $s$ above to be
\[
s = \frac{v_1 - v_2}{|v|_1^{2/3}},
% \in \big[0, |v|_1^{1/3}\big]  .
\]
and note that $0\le s\le |v|_1^{1/3}$. Since $v_1-v_2\ge10$, this choice of $s$ gives
\begin{equation}\label{eq:vertdev2}
v'_2 - \frac{s}{8} |v'|_1^{2/3} = \frac{v_1 + 3v_2}{4} - \frac{v_1 - v_2}{8 |v|_1^{2/3}} |v|_1^{2/3} = v_2 + \frac{v_1 - v_2}{8} \ge v_2+1 ,
\end{equation}
and hence shows that the lower endpoint of the interval in~\eqref{vertdev} lies strictly above $H_v$, as depicted in Figure~\ref{fig:geodabovev} (right). We further claim, with the above choice of $s$, that
\begin{equation}\label{eq:horizdev2}
v''_1 - \frac{s}{8} |v''|_1^{2/3} \ge -n+4,
\end{equation}
so that the left endpoint of the interval in~\eqref{horizdev} lies strictly to the right of $V_{-n\e_+}$, again as depicted in Figure~\ref{fig:geodabovev} (right). We postpone the verification of~\eqref{eq:horizdev2} to the end of the proof.

We henceforth assume that $s$ is specified as above, so that~\eqref{eq:vertdev2} and~\eqref{eq:horizdev2} hold. On the event in~\eqref{vertdev} we have that $\pi_-(\lambda; - \infty, 0)$ is above $\lfloor v' \rfloor - \frac{s}{8} |v'|_1^{2/3} \e_2$. By~\eqref{eq:vertdev2} and the fact that $\pi_-(\lambda; \infty, 0)$ is a directed path, this implies $\pi_-(\lambda; - \infty, 0)$ is also above $v$. It follows, from~\eqref{vertdev}, that
\begin{equation}\label{eq:vertdev3}
\P \big(\pi_-(\lambda; - \infty, 0) \text{ is above } v \big) \geq 1 - C e^{-c' s^3}  .
\end{equation}
Similarly, on the event in~\eqref{horizdev} we have by~\eqref{eq:horizdev2} that every point in the intersection of $\pi_-(\lambda; - \infty, 0)$ with $H_{-n\e_+}$ lies to the right of $-n\e_++4\e_1$. By Lemma~\ref{ctyofgeod}, the geodesics $\pi_-(\lambda;-\infty,0)$ and $\pi_-(\lambda;-n\e_+,0)$ coincide inside $R_{- (n-1) \e_+, 0}$, so $\pi_-(\lambda; - \infty, 0)$ lying to the right of $-n\e_+4\e_1$ implies that $E_{\rightarrow}(\lambda; - n \e_+, 0)$ occurs.
%Similarly, from $v''_2 - \frac{s}{8} |v''|_1^{2/3} > n$ we deduce \textcolor{red}{Since being above is not in a strict sense, we should ask that $-n \e_+ - \e_2$ is above $\pi_-$.}
%\[
%\Big\{H_{- n \e_+} \cap \pi(\lambda; - \infty, 0) \subseteq \Big[\lfloor v'' \rfloor - \frac{s}{8} |v''|_1^{2/3} \e_1, \lfloor v'' \rfloor + \frac{s}{8} |v''|_1^{2/3} \e_1 \Big] \Big\} \subseteq \Big\{- n \e_+ \text{ is above }\pi_-(\lambda; - \infty, 0) \Big\}  .
%\]
%However, due to Lemma \ref{ctyofgeod}, $\pi_-(\lambda; - \infty, 0)$ and $\pi_-(\lambda; - \infty, 0)$ coincide inside $R_{- (n-1) \e_+, 0}$, so that
%\[
%\big\{\pi_-(\lambda; - \infty, 0)\text{ lies to the right of }-n\e_+2\e_1 \big\}  \subseteq E_{\rightarrow}(\lambda; - n \e_+, 0)  .
%\]
It follows, from~\eqref{horizdev}, that
\begin{equation}\label{eq:horizdev3}
\P \big( E_{\rightarrow}(\lambda; - n \e_+, 0) \big) \geq 1 - C e^{-c' s^3}  .
\end{equation}
Finally, by Lemma \ref{monoofgeod}, on the event $E_{\rightarrow}(\lambda; - n \e_+, 0)$, $\pi_-(- n \e_+, 0)$ is above $\pi_-(\lambda; - n \e_+, 0)$, so that
\[
\{\pi_-(\lambda; - \infty, 0) \text{ is above } v\} \cap E_{\rightarrow}(\lambda; - n \e_+, 0) \subseteq \{\pi(- n \e_+, 0) \text{ is above } v\}  ,
\]
and we conclude from~\eqref{eq:vertdev3} and~\eqref{eq:horizdev3} that
\[
\P \big( \pi(- n \e_+, 0) \text{ is above } v \big) \geq 1 - 2C e^{-c' s^3}  .
\]
The above bound holds only as long as $|v|_1 \geq r_0$, but it suffices to replace the factor $2C$ by $C' := \max(2 C, e^{c' r_0} )$ for the bound to apply to all $v \in R_{- n \e_+, 0}$ such that $|v|_1\le n$ and $0\ge v_1\ge v_2$.
%\[
%\P \big( \pi(- n \e_+, 0) \text{ is above } v \big) \geq 1 - C' e^{-c' s^3} ,
%\]
%which is the announced inequality.

It only remains to verify~\eqref{eq:horizdev2}. For this, first note that
\[
|v''|_1 = \Big(1 + \frac{3v_1 + v_2}{v_1 + 3v_2} \Big)n =  \frac{4 |v|_1}{|v_1 + 3v_2|}n  .
\]
Since $8 \geq 2 \cdot 4^{2/3} \cdot 3^{1/3}$,
\[
v''_1 - \frac{s}{8} |v''|_1^{2/3}  \geq v''_1 - \frac{s}{2 \cdot s 4^{2/3} \cdot 3^{1/3}} |v''|_1^{2/3} = \frac{3v_1 + v_2}{|v_1 + 3v_2|}n- \frac{v_1 - v_2}{2 \cdot 3^{1/3}|v_1 + 3v_2|^{2/3}} n^{2/3}  .
\]
Factorize by $n$, then use $3 n \geq |v|_1 - 2v_2 \geq |v_1 + 3 v_2|$ to obtain
\[
v''_1 - \frac{s}{8} |v''|_1^{2/3} \geq \bigg( \frac{3v_1 + v_2}{|v_1 + 3v_2|} - \frac{v_1 - v_2}{2 \times 3^{1/3} n^{1/3}|v_1 + 3v_2|^{2/3}} \bigg)n \geq \Big(3v_1 + v_2 - \frac{1}{2}(v_1 - v_2) \Big) \frac{n}{|v_1 + 3v_2|}.
\]
Next, $(3 - 1/2)v_1 + (1 + 1/2)v_2 = (v_1 + 3v_2)/2 +2 v_1$. Moreover, the assumption $v_1-v_2\ge8$ implies that $0\ge 4v_1 \ge v_1 +3v_2+24$. This leads to the further lower bound
\[
\bigg(-\frac{1}{2} + \frac{2v_1}{|v_1 + 3v_2|}\bigg)n \geq \bigg(-\frac{1}{2} + \frac{v_1+3v_2+24}{2|v_1+3v_2|}\bigg)n \geq -n+12\frac{n}{|v_1+3v_2|} \ge -n+4,
\]
using the previously established bound $|v_1+3v_2|\le3n$. This completes the proof.
\end{proof}

%We end this subsection by proving Corollary \ref{coroutsidebulk}, which quickly follow by the observation that if the geodesic is above $v - \e_1 - \e_2$, then it does not go through $v$.

\subsection{Vertices `close' to the diagonal}

We proceed in this section with an argument that provides a better bound on the probability of being on the geodesic for vertices close to the main diagonal.

Let $v\in R_{0,n\e_+}$ be fixed. We start by shifting the plane so that $v$ is mapped to the origin. By translation invariance, we have that
\begin{equation}\label{eq:vbound_shift}
\P\big(v\in \pi(0, n\e_+)\big) = \P\big(0\in\pi(-v, n\e_+-v)\big).
\end{equation}
We next note that if the geodesic between $-v$ and $n\e_+-v$ visits the origin, then it also visits either $\e_1$ or $\e_2$. That is, we have
\begin{equation}\label{eq:vbound_split}
\P\big(0\in\pi(-v, n\e_+-v)\big) \leq \P\big(0, \e_1\in\pi(-v, n\e_+-v)\big)+\P\big(0, \e_2\in\pi(-v, n\e_+-v)\big).
\end{equation}
Since the two terms in the above right-hand side can be handled analogously, or by referring to symmetry with respect to the main diagonal, we shall henceforth only consider the former of the two. Indeed, we shall prove the following bound on the probability of visiting the origin and $\e_1$.

\begin{theorem} \label{thm: reformulated theorem}
There exist constants $0<c,C<\infty$ such that the following holds. For all $n\ge1$, $s\ge1$ and $v\in R_{0,n\e_+}$ such that $|v|_1\leq n$, $\min\{v_1,v_2\}\ge C$ and $|v_2-v_1|\le s |v|_1^{2/3}$ we have
%we have for all $s\geq1$ satisfying $|v_2-v_1|\le s |v|_1^{2/3}$ that
$$
\P\big(0, \e_1\in\pi(-v, n\e_+-v)\big) \leq C\Big(s^2 |v|_1^{-2/3} +\exp(-cs^3)\Big).
$$
\end{theorem}

In order to prove this theorem we will follow the approach from~\cite{balbussep20}, also employed in~\cite{grojanras25}, which boils down to an estimate of a random walk staying non-negative for a certain number of steps.

In order to facilitate the notation, we shall for the remainder of this section let $w=n\e_+-v$. We are hence interested in the geodesic $\pi(-v,w)$. Note that
\begin{equation}\label{eq:Tmax}
   T(-v,w) = \max\limits_i \big[T(-v, i\e_2)+T(i\e_2+\e_1, w)\big],
\end{equation}
where the maximum is over $-v_2\le i\le w_2=n-v_2$. Note that some geodesic from $-v$ to $w$ leaving the vertical axis at the origin is equivalent to the maximum in~\eqref{eq:Tmax} being attained for $i=0$. For $-v_2\le i\le w_2$, let
\begin{equation}\label{eq:Di_def}
    D_i := 
    \big[T(-v, 0)+T(\e_1, w)\big] - \big[T(-v, i\e_2)+T(\e_1+i\e_2, w)\big].
\end{equation}
Then $D_i$ denotes the difference in travel time when traveling from $-v$ to $w$ leaving the vertical axis at the origin as opposed to at $i\e_2$. Then, the maximum in~\eqref{eq:Tmax} being attained for $i=0$ is equivalent to the event that $D_i\ge0$ for all $-v_2\le i\le w_2$. This leads to the identity
\begin{equation}\label{eq:piDi}
\P\big(0, \e_1\in\pi(-v, w)\big)=\P\big(D_i\ge0\text{ for }-v_2\le i\le w_2\big)
\end{equation}
By relying on Theorem~\ref{boundoutsidebulk}, we find that it is unlikely that the maximum in~\eqref{eq:Tmax} is attained for $|i|\gg|v|_1^{2/3}$. This means that it is unlikely for $D_i$ to be negative for $i\gg|v|_1^{2/3}$ and that, subject to a small error, it will suffice to consider the sequence $D_i$ for $i$ between $-k$ and $k$, where $k$ is approximately of the order $|v|_1^{2/3}$.

Next, we introduce the notation
\begin{equation}
\begin{aligned}
\Delta_j &:= T(-v, j\e_2)-T(-v, (j-1)\e_2),\\
\Delta_j'&:= T(\e_1+(j-1)\e_2, w) - T(\e_1+j\e_2, w).
\end{aligned}
\end{equation}
Note that since any directed path from $x$ to $y$ extends to a directed path from $x$ to $y+\e_2$, it follows that $\Delta_j\ge0$. By a similar argument we conclude that also $\Delta'_j\ge0$. Note further that, by a telescoping argument, we have
\begin{equation*}
    T(-v,0)-T(-v,i\e_2)=\left\{
\begin{aligned}
    \sum_{j=1}^i-\Delta_j &&\text{for }i\ge0;\\
    \sum_{j=i+1}^0\Delta_j &&\text{for }i<0.
\end{aligned}
    \right.
\end{equation*}
In addition, we have
\begin{equation*}
    T(\e_1+i\e_2,w)-T(\e_1,w)=\left\{
\begin{aligned}
    \sum_{j=1}^i-\Delta'_j &&\text{for }i\ge0;\\
    \sum_{j=i+1}^0\Delta'_j &&\text{for }i<0.
\end{aligned}
    \right.
\end{equation*}
In this notation, we may express the differences $D_i$ in terms of the incremental differences $\Delta_j$ and $\Delta'_j$ as
\begin{equation}
D_i=\left\{
    \begin{aligned}
        \sum_{j=1}^i(-\Delta_j+\Delta'_j) && \text{for }i\ge0;\\
        \sum_{j=i+1}^0(\Delta_j-\Delta'_j) && \text{for }i<0.
    \end{aligned}
    \right.
\end{equation}

We next note that the sequence $(\Delta_j)_j$ is independent of the sequence $(\Delta'_j)_j$, since the former is a function of the weight configuration restricted to the left half-plane (including the vertical axis) and the latter is a function of the weight configuration restricted to the right half-plane (excluding the vertical axis). It does not imply that the variables within each sequence $(\Delta_j)_j$ and $(\Delta'_j)_j$ are mutually independent. However, in the limit as $|v|_1$ and $n$ tends to infinity, for $v$ `close' to the main diagonal of $R_{0,n\e_+}$,
%for $1\ll|v|_1\le n$, when $v$ is `close' to the main diagonal, then
the variables $\Delta_j$ approach the boundary variables $\omega^V_{j\e_2}(1/2)$, which are independent and identically distributed. An analogous statement holds for the variables $\Delta'_j$.
%in each sequence are approximately independent and approximately equal in distribution, for $|j|$ of order at most $k\approx |v|_1^{2/3}$.
That is, for $1\ll|v|_1\le n$, and $v$ close to the main diagonal, $(D_i)_{0\le i\le k}$ is approximately equal in distribution to a zero-mean random walk with i.i.d. increments. Consequently, the probability in~\eqref{eq:piDi} should be approximately equal to the probability of a zero-mean random walk staying above its starting point for order $|v|_1^{2/3}$ steps.
%which is known to be of order $1/\sqrt{k}$.

In order to take advantage of the asymptotic independence for finite $n$, the approach will be to bound the increments $\Delta_j$ and $\Delta'_j$ by the analogues for travel times with boundary, which are precisely the boundary weights $\omega^v_{j\e_2}(\lambda)$, for $\lambda$ slightly above or below $1/2$. This will allow us to bound $D_i$ from above and below by a random walk with a small drift instead of zero drift.

\subsubsection{Approximating the increments $\Delta_j$ and $\Delta_j'$ with boundary weights}

For $v\in R_{0,n\e_+}$ and $0\le s\le \frac{1}{18}|v|_1^{1/3}$, let
%$k:=\lfloor s|v|_1^{2/3}\rfloor+1$ and
$$
\lambda^+ := \dfrac{1}{2}+8s|v|_1^{-1/3}\quad \text{and}\quad    \lambda^- := \dfrac{1}{2}-8s|v|_1^{-1/3}.
$$
Then, as $s$ varies in the given interval, we have $\frac{1}{18}\le\lambda^-\le\frac12\le\lambda^+\le\frac{17}{18}$. We move to bound $\Delta_j$ in terms of boundary weights associated to the directions $\lambda^+$ and $\lambda^-$, for $-k\le j\le k$. Below we shall choose $k$ also depending on $v$ and $s$, but for the moment we impose no restriction on $k$ other than $k\le \min\{v_2,w_2\}$.

%The estimates hold conditioned on the event that the $\lambda^-$-geodesics to all $j\e_2$, $-k\leq j\leq k$, make the first step to the right and the $\lambda^+$-geodesics make the first step up, see Figure \ref{fig:directions for lambda}. By the choice of the directions, these events hold with high probability, which we will show in Lemma \ref{lem: assumption}.

\begin{lemma}\label{lem: bound the difference of times} Let $0\le k\le\min\{v_2,w_2\}$. On the event that $E_{\rightarrow}(\lambda^-; -v, k\e_2)$ and $E_{\uparrow}(\lambda^+; -v, -k\e_2)$ occur, it holds that
\begin{eqnarray*}
\omega_{j\e_2}^V(\lambda^-)\leq \Delta_j \leq \omega_{j\e_2}^V(\lambda^+),\quad \text{for all }-k\leq j\leq k.
\end{eqnarray*}
\end{lemma}

For the proof, we will use the following deterministic lemma.

\begin{lemma}\label{lem: lemma B1 from 8}
    Let $(\omg_x)_{x\in\Z^2}$ and $(\tilde\omg_x)_{x\in\Z^2}$ be two collection of (deterministic) weights satisfying the relation $\omega_{i\e_1}\ge\tilde\omega_{i\e_1}$ and $\omega_{i\e_2}\le\tilde\omega_{i\e_2}$ for $i\ge0$, and $\omega_x=\tilde\omega_x$ for $x\ge \e_+$.
    %Consider the boundary weights $\bar{\omega}^H_{i\e_1}\leq \omega_{i\e_1}$ for all $i\geq 0,$ on the horizontal axis, and $\bar{\omega}^V_{i\e_2}\geq \omega_{i\e_2}$ for all $i\geq 0$ on the vertical axis. 
    Let $T(0,x)$ and $\tilde{T}(0, x)$ denote the travel time between the origin and $x$ with respect to $\omega$ and $\tilde\omega$, respectively. Then, for all $x\in\Z^2_{\geq 0}$ we have
    \begin{eqnarray*}
    \tilde{T}(0, x+\e_2)-\tilde{T}(0, x)\geq T(0, x+\e_2)-T(0, x);\\
    \tilde{T}(0, x+\e_1)-\tilde{T}(0, x)\leq T(0, x+\e_1)-T(0, x).
    \end{eqnarray*}
\end{lemma}

\begin{proof}
See Lemma~B.1 in~\cite{balbussep20}.
\end{proof}

\begin{proof}[Proof of Lemma \ref{lem: bound the difference of times}]
Via a standard path-crossing argument, we note that if some $\lambda^+$-geodesic from $-v$ to $j\e_2$ takes the first step upwards, then some $\lambda^+$-geodesic from $-v$ to $(j+1)\e_2$ will take its first step upwards too. It follows that if $E_\uparrow(\lambda^+,-v,-k\e_2)$ occurs, then $E_\uparrow(\lambda^+,-v,j\e_2)$ occurs for all $j\ge-k$.

    Let $(\tilde\omega_x)_{x\in\Z^2}$ denote the weights configuration where $\tilde{\omega}_{-v+i\e_2}=\omega_{-v+i\e_2}^V(\lambda^+)$ for $i\in\Z$, and $\tilde{\omega}_x=\omega_x$ elsewhere.
    From Lemma~\ref{lem: lemma B1 from 8}, and the observation that on $E_{\uparrow}(\lambda^+;-v, j\e_2)$ we have $\tilde{T}(-v, j\e_2) = T(\lambda^+;-v, j\e_2)$, it follows that 
    \begin{align*}
        \Delta_j&=T(-v, j\e_2) - T(-v, (j-1)\e_2) \\
        &\leq\tilde{T}(-v, j\e_2) - \tilde{T}(-v, (j-1)\e_2)\\
        &= T({\lambda^+};-v, j\e_2)-T({\lambda^+};-v, (j-1)\e_2) = \omega_{j\e_2}^V(\lambda^+),
    \end{align*}
    where  the last equality follows from additivity of the boundary travel times~\eqref{eq:additivity}.

A analogous argument shows that if $E_{\rightarrow}(\lambda^-; -v, k\e_2)$ occurs, then $E_{\rightarrow}(\lambda^-; -v, k\e_2)$ occur for all $j\le k$, and hence that $\Delta_j \ge \omega_{j\e_2}^V(\lambda^-)$ for all $j\le k$ on the event $E_{\rightarrow}(\lambda^-; -v, k\e_2)$. 
\end{proof}

We proceed with an estimate of the probability of the events in Lemma~\ref{lem: bound the difference of times}. The events will occur with probability close to one for $k$ of order $s|v|_1^{2/3}$. We henceforth set $k:=\lfloor 2s|v|_1^{2/3}\rfloor+1$.

\begin{figure}
    \centering
    \includegraphics[width=0.7\linewidth]{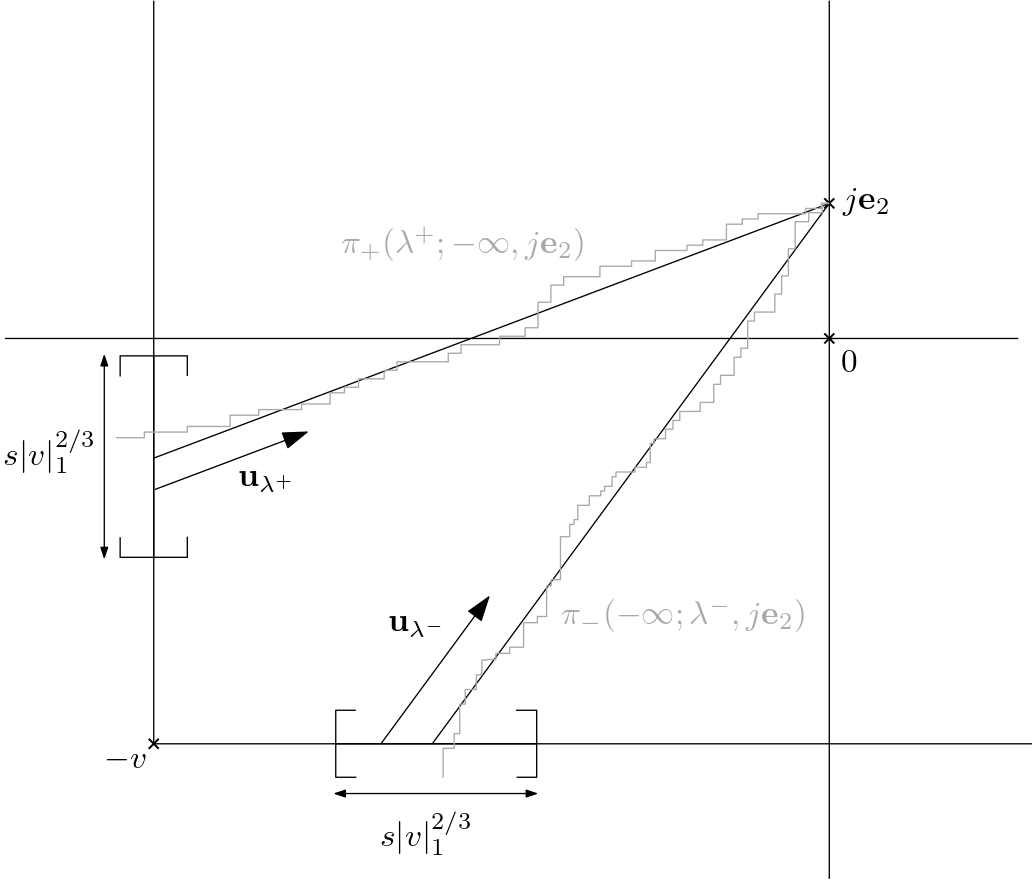}
    \caption{Schematic representation of the directed geodesics given by $\lambda^+$ and $\lambda^-$}
    \label{fig:directions for lambda}
\end{figure}

\begin{lemma}\label{lem: assumption}
    There exist constants $0<c,C<\infty$ such that the following holds. For every $v\in\Z^2_{\ge0}$ and $1\le s\le\frac{1}{18}|v|_1^{1/3}$ such that $\min\{v_1,v_2\} \geq C$ and $|v_2-v_1|\le s|v|_1^{2/3}$,
%    , let $s$ be a parameter with
%    $$
%    |v_1-v_2|\cdot|v|_1^{-2/3}\vee1\leq s\leq \frac{1}{18} |v|_1^{1/3}
%    $$
    we have with $k = \lfloor 2s|v|_1^{2/3}\rfloor+1$ that
%\begin{align*}
%\P\big( E_{\rightarrow}(\lambda^-; -v, k\e_2)\big) &\geq 1-C\exp(-cs^3), \\
%\P\big( E_{\uparrow}(\lambda^+; -v, -k\e_2)\big) &\geq 1-C\exp(-cs^3).
%\end{align*}
$$
\P\Big(\omega_{j\e_2}^V(\lambda^-)\leq \Delta_j \leq \omega_{j\e_2}^V(\lambda^+)\text{ for all }-k\leq j\leq k\Big)\ge1-C\exp(-cs^3).
$$
\end{lemma}

\begin{proof}
%To bound the probabilities of the events $E_{\rightarrow}(\lambda^-; -v, k\e_2)$ and $E_{\uparrow}(\lambda^+; -v, -k\e_2)$, we apply Theorem~\ref{backgeodlargedev}. Since the two events can be treated similarly, we consider only the latter.
According to Lemma~\ref{lem: bound the difference of times}, it will suffice to show that each of the events $E_{\rightarrow}(\lambda^-; -v, k\e_2)$ and $E_{\uparrow}(\lambda^+; -v, -k\e_2)$ occur with probability $1-C\exp(-cs^3)$ for some $0<c,C<\infty$. Since the two events can be treated similarly, we consider only the latter.

Recall that $V_{-v}$ denotes the vertical line passing through $-v$. Let $u$ denote the intersection between $V_{-v}$ and the ray $\big\{-k\e_2- t\u_{\lambda^+}\big\}_{t\ge0}$, and let $x$ be such that $u = -k\e_2- x\u_{\lambda^+}$. That is, $x = v_1/\lambda^+$. Let $r_0$ be as in Theorem~\ref{backgeodlargedev}. By assumption we have $\frac12\le\lambda^+\le\frac{17}{18}$, and for $v_1\ge r_0$ we have $x\ge v_1/\lambda^+\ge r_0$. By Theorem~\ref{backgeodlargedev} there exist constants $0<c,C<\infty$ such that for $0\le s\le x^{1/3}$ we have
\begin{equation}\label{eq:geoboundagain}
\P\left(\pi(\lambda^+; -\infty, -k\e_2)\cap V_{-v}\subseteq \big[u-sx^{2/3}\e_2, u+sx^{2/3}\e_2\big]\right)\geq 1-C\exp(-cs^3).
\end{equation}
Note that the assumptions that $|v_2-v_1|\le s|v|_1^{2/3}$ and  $s\le\frac{1}{18}|v|_1^{1/3}$ imply, in particular, that $|v|_1\le3v_1$. Moreover, $s\le\frac{1}{18}|v|_1^{1/3}$ and $|v|_1\le3v_1$ guarantee that $s\le v_1^{1/3}\le x^{1/3}$.
All that remains is therefore to verify that $u-sx^{2/3}\e_2$ is above $-v$, and hence that the event in~\eqref{eq:geoboundagain} implies $E_{\uparrow}(\lambda^+; -v, -k\e_2)$.

Since $u_2=-k-\frac{1-\lambda^+}{\lambda^+}v_1$ and $x=v_1/\lambda^+$, we need to verify that
\begin{equation}\label{eq:paraboundagain}
v_2-\frac{1-\lambda^+}{\lambda^+}v_1-s\Big(\frac{v_1}{\lambda^+}\Big)^{2/3}>k.
\end{equation}
Using the assumption that $v_2\ge v_1-s|v|_1^{2/3}$, and the bounds $\lambda^+\ge\frac12$ and $(1-\lambda^+)/\lambda^+\le1+(1-2\lambda^+)$, we obtain that
$$
v_2-\frac{1-\lambda^+}{\lambda^+}v_1-s\Big(\frac{v_1}{\lambda^+}\Big)^{2/3}\ge v_1-s|v|_1^{2/3}-v_1\big(1-16s|v|_1^{-1/3}\big)-2sv_1^{2/3}.
$$
Using again the observation that $\frac13|v|_1\le v_1\le|v|_1$ gives the further lower bound
$$
16sv_1|v|_1^{-1/3}-s|v|_1^{2/3}-2sv_1^{2/3}\ge \frac{16}{3}s|v|_1^{2/3}-3s|v|_1^{2/3}= \frac{7}{3}s|v|_1^{2/3}>k.
$$
So~\eqref{eq:paraboundagain} holds, and the proof is complete.
\end{proof}

Lemma~\ref{lem: assumption} shows that we may approximate the increments $\Delta_j$ for $-k\le j\le k$ by boundary weights. We need analogous bounds on $\Delta_j'$ for $-k\le j\le k$. Such bounds can be obtained through a construction of stationary travel times with boundary weights on the top-right boundary as opposed to the bottom-left boundary as in the case of $T(\lambda;x,y)$. Instead of repeating the construction in Section~\ref{se:stationarylpp} for the top-right boundary, we may instead appeal to the existing construction via a simple rotation of the plane.

Let $x\mapsto-x+\e_1$ denote the transformation of $\Z^2$ obtained by first translating $\e_1$ to the origin, and then performing a rotation by 180 degrees. This transformation translates the top-right boundary of $R_{-v,w}$ to the bottom-left boundary of $R_{-w+\e_1,v+\e_1}$. Denote by $\tilde\omega^V_x(\lambda)$ and $\tilde\omega^H_x(\lambda)$ the boundary weights at $x$, obtained from the construction in Section~\ref{se:stationarylpp}, after applying the transformation $x\mapsto -x+\e_1$. Finally, set $\hat\omega^V_x(\lambda):=\tilde\omega^V_{-x+\e_1}(\lambda)$ and $\hat\omega^H_x(\lambda):=\tilde\omega^H_{-x+\e_1}(\lambda)$.

Note that the collections $(\hat\omega^V_x(\lambda))_{x\in\Z^2}$ and $(\hat\omega^H_x(\lambda))_{x\in\Z^2}$ are in effect boundary weights generated from stationary travel times with boundary weights on the top-right boundary, and that in distribution $(\hat\omega^V_{-x}(\lambda))_{x\in\Z^2}$ and $(\hat\omega^H_{-x}(\lambda))_{x\in\Z^2}$ are equal to $(\omega^V_x(\lambda))_{x\in\Z^2}$ and $(\omega^H_x(\lambda))_{x\in\Z^2}$.
Let
$$
    \hat{\lambda}^+ := \dfrac{1}{2}+8s|w-\e_1|_1^{-1/3}
    \quad\text{and}\quad
    \hat{\lambda}^- := \dfrac{1}{2}-8s|w-\e_1|_1^{-1/3}.
$$
The following is the required analogue of Lemma~\ref{lem: assumption}.

%\begin{eqnarray*}
%    \P(\omega_{j\e_2}^V(\lambda^-)\leq \Delta_j \leq \omega_{j\e_2}^V(\lambda^+), \text{ for all}-k\leq j\leq k) \geq 1-2\exp(-cs^3)
%\end{eqnarray*}
%To show estimates on $\Delta_j'$, we use the result above and apply symmetry with respect to zero. Moreover, to make sure that the boundary weights estimating $\Delta_j'$ are independent of the boundary weights estimating $\Delta_j$, we translate the plane by the vector $\e_1$. The whole transformation is then $R(x):=-x+\e_1$. After applying symmetry to $\Z^2$, the boundary weights are assigned to the top and right boundaries instead of the bottom and left ones. We denote these boundary weights as $\hat{\omega}^V$ at the right (vertical) boundary and $\hat{\omega}^H$ at the top (horizontal) boundary. We also note that
%\begin{eqnarray*}
%    R(j\e_2) &=& -j\e_2+\e_1,\\
%    R((j-1)\e_2) &=& (-j+1)\e_2+\e_1,\\
%    R(\Z_{\leq0}) &=& \Z_{\geq\e_1}.
%\end{eqnarray*}
%
%For a point $w\in \Z_{\geq0}$, we define
%$$\Tilde{\Delta}_j(w)=T((j-1)\e_2+\e_1, w)-T(j\e_2+\e_1, w),$$
%so that $\Delta_j' = \Tilde{\Delta}_j(n\e_+-v)$.

\begin{lemma}\label{lem: another bound on the difference of times}  
There exist constants $0<c, C<\infty$ such that the following holds. For every $n\ge1$, $v\in \Z^2_{\geq0}$ and $1\le s\le\frac{1}{18}|v|_1^{1/3}$ such that $|v|_1\le n$ and $|v_2-v_1|\le s|v|_1^{2/3}$, we have with $w=n\e_+-v$ and $k=\lfloor 2s|v|_1^{2/3}\rfloor+1$ that
%or any $s$ such that $|w_1+1-w_2|\cdot|w+\e_1|_1^{-2/3}+1\leq s\leq \delta |w+\e_1|_1^{\frac{2}{3}}$ and the number of steps $k_w = [2s|w+\e_1|_1^{2/3}]$ with probability at least $1-2\exp(-cs^3)$ it holds that for all $-k_w\leq j\leq k_w$
$$
\P\Big(\hat{\omega}_{(j-1)\e_2+\e_1}^V(\hat{\lambda}^-)\leq \Delta'_j \leq \hat{\omega}_{(j-1)\e_2+\e_1}^V(\hat{\lambda}^+)\text{ for all }-k\le j\le k\Big)\ge1-C\exp(-cs^3).
$$
Moreover, the weights $(\omega^V_{j\e_2}(\lambda))_{j\in\Z}$ and $(\hat\omega^V_{\e_1+j\e_2}(\lambda'))_{j\in\Z}$ are independent for every choice of $\lambda$ and $\lambda'$.
\end{lemma}

\begin{proof}
Via the transformation $x\mapsto-x+\e_1$ and Lemma~\ref{lem: bound the difference of times}, we have on the events $E_\rightarrow(\hat\lambda^-,-w+\e_1,k\e_2)$ and $E_\uparrow(\hat\lambda^+,-w+\e_1,-k\e_2)$ that
$$
\hat{\omega}_{(j-1)\e_2+\e_1}^V(\hat{\lambda}^-)\leq \Delta'_j \leq \hat{\omega}_{(j-1)\e_2+\e_1}^V(\hat{\lambda}^+),\quad\text{for all }-k\le j\le k.
$$
It thus suffices to verify that the events $E_\rightarrow(\hat\lambda^-,-w+\e_1,k\e_2)$ and $E_\uparrow(\hat\lambda^+,-w+\e_1,-k\e_2)$ occur with probability $1-C\exp(-cs^3)$ for some $0<c,C<\infty$ under the stated conditions.

By assumption we have $|v|_1\le n$, and hence $n-1\le|w-\e_1|_1\le2n$. By increasing $C$ if necessary, we may assume $n\ge17$. In this case, we have $\frac12\le\hat\lambda^+\le\frac{35}{36}$, and by increasing $n$ further if necessary, we may again apply Theorem~\ref{backgeodlargedev} to obtain that $E_\uparrow(\hat\lambda^+;-w+\e_1,-k\e_2)$ holds with probability $1-C\exp(-cs^3)$, provided that
$$
w_2-\frac{1-\hat\lambda^+}{\hat\lambda^+}(w_1-1)-s\bigg(\frac{w_1-1}{\hat\lambda^+}\bigg)^{2/3}>k.
$$
Noting, in particular, that $|w_2-w_1-1|\le|v_2-v_1|+1\le s|v|_1^{2/3}+1$, this follows as before.

    %The lemma is a verbatim translation of Lemma~\ref{lem: assumption} via the transformation $x\mapsto-x+\e_1$.
    %By Theorem \ref{thm:boundary1} \textit{iii} the weight $\omega_{j\e_2}^V$ is independent of the weights $\omega_i$ to the right of it, i.e. $i\e_1>0$. The weight $\hat{\omega}_{(j-1)\e_2+\e_1}$ is independent of the weights $\omega_i$ to the left of it, i.e. $i\e_1\leq 0$; hence, we have independence.

%$\min\{w_1,w_2\}\geq C$
    
    Finally, for every choice of $\lambda$ and $\lambda'$, the weights $(\omega^V_{j\e_2}(\lambda))_{j\in\Z}$ and $(\hat\omega^V_{\e_1+j\e_2}(\lambda'))_{j\in\Z}$ are independent since the former is measurable with respect to the weight configuration restricted to $\Z_{\le0}\times\Z$, whereas the latter is measurable with respect to the weight configuration restricted to $\Z_{\ge1}\times\Z$; see Remark~\ref{rem:construction}.
\end{proof}

We remark that when $|v|_1\ll n$ then bound in Lemma~\ref{lem: another bound on the difference of times} is suboptimal, in that it would hold for $j$ up to order $sn^{2/3}$ (assuming $sn^{2/3}\le\min\{v_1,w_2\}$). However, there is no gain in this for our intended purposes. We have instead stated the lemma so that its assumptions match those of Lemma~\ref{lem: assumption}, in order to facilitate its later application.

%\begin{remark}
% Lemma \ref{lem: another bound on the difference of times} holds for $w = n\e_+-v$, the value $s$ in the interval  $\max(1, |v_1-v_2|\cdot |v|_1^{-2/3})\leq s\leq \delta\cdot |v|_1^{2/3}$ for some small $\delta$ and for $k=[2s|v|_1^{2/3}]$. Indeed, given that $|v|_1\leq n$ we get that the interval of $s$ above is contained in the interval given by Lemma \ref{lem: bound the difference of times}. As for the number of steps, it holds that $|n\e_+-v+\e_1|_1^{2/3}\geq|v|_1^{2/3}$. 
%
%\end{remark}

\subsubsection{A random walk with a small positive drift}

For ease of notation, we let for $j\in\Z$
\begin{equation}\label{eq:YjZj}
\begin{aligned}
    Y_j &:= -\omega_{j\e_2}^V(\lambda^-)+\hat{\omega}^V_{(j-1)\e_2+\e_1}(\hat{\lambda}^+);\\
    Z_j &:= \omega_{j\e_2}^V(\lambda^+)-\hat{\omega}^V_{(j-1)\e_2+\e_1}(\hat{\lambda}^-).
\end{aligned}
\end{equation}
It follows from Lemmas~\ref{lem: assumption} and~\ref{lem: another bound on the difference of times} that with probability $1-2C\exp(-cs^3)$ we have
\begin{equation*}
D_i=\left\{
    \begin{aligned}
        \sum_{j=1}^i(-\Delta_j+\Delta'_j)\le \sum_{j=1}^iY_j &&& \text{for }0\le i\le k;\\
        \sum_{j=i+1}^0(\Delta_j-\Delta'_j)\le \sum_{j=i+1}^0Z_j &&& \text{for }-k\le i<0.
    \end{aligned}
    \right.
\end{equation*}

As we shall see, in Lemma~\ref{lem: estimating mu, delta, sigma} below, the variables in each of the sequences $(Y_j)_{j \in \Z}$ and $(Z_j)_{j \in \Z}$ are i.i.d.\ with positive mean of order $s|v|_1^{-1/3}$. In addition, the two sequences are in fact mutually independent. The sequence $(D_i)_{-k\le i\le k}$ can thus be bounded by two independent random walks with a slight positive drift. In light of~\eqref{eq:piDi}, it will suffice to bound the probability that these random walks remain non-negative for distance $k=\lfloor2s|v|_1^{2/3}\rfloor+1$. That is, we will need a bound on the non-recurrence of a random walk with a positive drift. Recurrence bounds of this type for random walks without drift are classical; see e.g.\ \cite[Lemma~5.1.8]{lawlim10}. The following proposition is an extension to random walks with non-negative drift, and is proved similarly.

%In Proposition \ref{prop: random walk} we estimate the probability of a random walk with a small drift being above $0$ for $N$ steps. In Lemma \ref{lem: estimating mu, delta, sigma} we bound the parameters of the distribution of $X_j$ and $Y_j$. In the end of the section, we present the full proof of Theorem~\ref{thm: reformulated theorem}.

%The following is an extension to random walks with drift of a classical recurrence bound on random walks without drift. The proof follows closely the proof of Lemma~5.1.8 in~\cite{Lawler-Limic book}.

\begin{proposition}\label{prop: random walk}
Let $X_1,X_2,\ldots$ be i.i.d.\ random variables with mean $\mu\ge0$, variance $\sigma^2>0$ satisfying $\delta:=\P(X_1\ge1)>0$. Set $S_0=0$ and let $S_k=X_1+\ldots+X_k$ for $k\ge1$. Then
$$
\P\Big(S_k\ge0\text{ for }k=1,2,\ldots,N\Big)\le\frac{4\sigma}{\delta\sqrt{N}}+\frac{\mu}{\delta}.
$$
\end{proposition}

For $\sigma$ and $\delta$ bounded away from zero and infinity, and with the mean varying with $N$, the bound can be expressed as $C\max\{1/\sqrt{N},\mu\}$ for some constant $C<\infty$.

\begin{proof}
Let $q_N:=\P(S_k\ge0\text{ for }k=1,2,\ldots,N)$, and note that $q_N$ is non-increasing in $N$. We have, moreover, that
\begin{equation}\label{eq:+1-bound}
\P(S_1\ge1,\ldots,S_N\ge1)\ge\P(X_1\ge1)\P(S_k-X_1\ge0\text{ for }k=2,3,\ldots,N)=\delta q_{N-1}\ge\delta q_N.
\end{equation}
For $N\ge k\ge1$ let $J_{k,N}:=\{S_{k+1},\ldots,S_N\ge S_k+1\}$, and note (from~\eqref{eq:+1-bound}) that
$$
\P(J_{k,N})=\P(S_1\ge1,\ldots,S_{N-k}\ge1)\ge\delta q_{N-k}\ge \delta q_N.
$$

Let $m_N$ and $M_N$ denote the minimal and maximal values, respectively, of $S_k$ over the range $k=0,1,\ldots,N$. Then, since the number of times the process can jump up by one, and remain above that level until time $N$, is bounded by the range of $S_k$, that is that
$$
\sum_{k=0}^{N-1}{\bf 1}_{J_{k,N}}\le M_N-m_N,
$$
we obtain that
\begin{equation}\label{eq:qn_bound}
N\delta q_N\le\sum_{k=0}^{N-1}\P(J_{k,N})\le\E[M_N-m_N].
\end{equation}

Next, we define $S_k':=S_k-\mu k$ by removing the drift from $S_k$. Let $m_N'$ and $M_N'$ denote the minimal and maximal values, respectively, of $S_k'$ over the range $k=0,1,\ldots,N$. Since $\mu\ge0$ by assumption, we have $m_N\ge m_N'$ and $M_N\le M_N'+\mu N$. Consequently, we have
\begin{equation}\label{eq:range_bound}
\E[M_N-m_N]\le\E[M_N'-m_N']+\mu N\le2\E\big[\max\{|S_k'|:k=0,1,\ldots,N\}\big]+\mu N.
\end{equation}
Kolmogorov's maximal inequality gives that
$$
\P\big(\max\{|S_k'|:k=0,1,\ldots,N\}\ge t\big)\le\frac{\Var(S_N')}{t^2}=\frac{\sigma^2 N}{t^2},
$$
and hence that
\begin{equation}\label{eq:mean_bound}
\begin{aligned}
\E\big[\max\{|S_k'|:k=0,1,\ldots,N\}\big]&=\int_0^\infty \P\big(\max\{|S_k'|:k=0,1,\ldots,N\}\ge t\big)\,dt\\
&\le\sigma\sqrt{N}+\int_{\sigma\sqrt{N}}^\infty \frac{\sigma^2 N}{t^2}\,dt=\sigma\sqrt{N}+\frac{\sigma^2 N}{\sigma\sqrt{N}}=2\sigma\sqrt{N}.
\end{aligned}
\end{equation}

Finally, putting together~\eqref{eq:qn_bound},~\eqref{eq:range_bound} and~\eqref{eq:mean_bound}, we obtain that
$$
q_N\le\frac{1}{\delta N}\E[M_N-m_N]\le\frac{1}{\delta N}\Big(2\E\big[\max\{|S_k'|:k=0,1,\ldots,N\}\big]+\mu N\Big)\le\frac{1}{\delta N}\Big(4\sigma\sqrt{N}+\mu N\Big),
$$
as required.
\end{proof}

In order to apply the proposition, we need information regarding the variables $Y_j$ and $Z_j$ in~\eqref{eq:YjZj}.

\begin{lemma}\label{lem: estimating mu, delta, sigma} 
There exist constants $0<c,C<\infty$ such that the following holds. For every $n\ge1$, $v\in R_{0,n\e_+}$ and $1\le s\le \frac{1}{18}|v|_1^{1/3}$ such that $|v|_1\le n$ we have that
\begin{enumerate}[label=(\roman*)]
\item $(Y_j)_{j=1,\ldots,k}$ and $(Z_j)_{j=-k+1\ldots,0}$ are i.i.d.\ sequences;
\item $(Y_j)_{j=1,\ldots,k}$ and $(Z_j)_{j=-k+1\ldots,0}$ are mutually independent;
\item $\P(Y_j\ge1)\ge c$ and $\P(Z_j\ge1)\ge c$, and $\Var(Y_j)\ge c$ and $\Var(Z_j)\ge c$;
\item $0\le\E[Y_j]\le Cs|v|_1^{-1/3}$ and $0\le\E[Z_j]\le Cs|v|_1^{-1/3}$.
\end{enumerate}
%Let $v\in Bulk(\eta)$, $v\geq0$, $C\leq|v|_1\leq n$ for some constant $C$, and let $s$ be in the interval $1\vee |v_1-v_2|\cdot|v|_1^{-2/3} \leq s \leq \delta|v|_1^{2/3}$ for some small $\delta$. Let $k$ be defined as $k = \left[s|v|_1^{\frac{2}{3}}\right]$, and $X_j, Y_j$ be defined as in \eqref{eqn: defn X_j}.
%    The variance and probability of being at most -1 of $X_j$ and $Y_j$ are bounded from zero and infinity uniformly in $j$, $-k\leq j\leq k$. Moreover, the expectation can be estimated as
%    \begin{eqnarray*}
%         0\leq\E[X_j]\leq Cs|v|_1^{-1/3};\\
%         0\leq \E[Y_j]\leq Cs|v|_1^{-1/3};
%    \end{eqnarray*}
%    for some constant $C<\infty$.
\end{lemma}

\begin{proof}
By construction, for every choice of $\lambda$ and $\lambda'$, the sequences $(\omega^V_{j\e_2}(\lambda))_{j\in\Z}$ and $(\hat\omega^V_{\e_1+(j-1)\e_2}(\lambda'))_{j\in\Z}$ are independent, and in the case $\lambda=\lambda'$ they are equal in distribution; see Lemma~\ref{lem: another bound on the difference of times}. Moreover, by part~\emph{(iii)} of Theorem~\ref{thm:boundary1}, the variables $(\omega^V_{j\e_2}(\lambda))_{j\ge1}$, and hence also the variables $(\hat\omega^V_{\e_1+(j-1)\e_2}(\lambda'))_{j\ge1}$, are mutually independent. It follows that $(Y_j)_{j\ge1}$ forms an i.i.d.\ sequence, and a similar argument shows the same for $(Z_j)_{j\le0}$. This proves~\emph{(i)}.

Note that $(Y_j)_{j\ge1}$ is defined in terms of $\omega^V_{j\e_2}(\lambda^-)$ and $\hat\omega^V_{\e_1+(j-1)\e_2}(\hat\lambda^+)$ for $j\ge1$, whereas $(Z_j)_{j\le0}$ is defined in terms of $\omega^V_{j\e_2}(\lambda^+)$ and $\hat\omega^V_{\e_1+(j-1)\e_2}(\hat\lambda^-)$ for $j\le0$. By part~\emph{(iii)} of Theorem~\ref{thm:boundary2} the variables $(\omega^V_{j\e_2}(\lambda^-))_{j\ge1}$ and $(\omega^V_{j\e_2}(\lambda^+))_{j\le0}$ are independent, and the same thus applies to $(\hat\omega^V_{\e_1+(j-1)\e_2}(\hat\lambda^+))_{j\ge1}$ and $(\hat\omega^V_{\e_1+(j-1)\e_2}(\hat\lambda^-))_{j\le0}$. This implies that $(Y_j)_{j\ge1}$ and $(Z_j)_{j\le0}$ are independent, establishing~\emph{(ii)}.

We proceed with the properties in parts~\emph{(iii)} and~\emph{(iv)}. Under the assumption that $s\le\frac{1}{18}|v|_1^{2/3}$ and $n\ge17$ we obtain, as in Lemmas~\ref{lem: assumption} and~\ref{lem: another bound on the difference of times}, that
$$
\frac{1}{18}\le\lambda^-\le\lambda^+\le\frac{17}{18}\quad\text{and}\quad\frac{1}{36}\le\hat\lambda^-\le\hat\lambda^+\le\frac{35}{36}.
$$
Since $p_V(\lambda)$ is a (decreasing) bijection from $(0,1)$ to $(0,p)$, it follows that under the given assumptions $p_V(\lambda)$ is bounded away from zero and one. Since $\omega^V_{j\e_2}(\lambda)$ is geometrically distributed with parameter $p_V(\lambda)$, and since $\omega^V_{j\e_2}(\lambda^-)$ and $\hat\omega^V_{\e_1+(j-1)\e_2}(\hat\lambda^+)$ are independent, it follows trivially that $\P(Y_j\ge1)\ge c$ and $\Var(Y_j)\ge c$ for some $c>0$. An identical argument gives the same for $Z_j$, which verifies~\emph{(iii)}.

Again, since $\omega^V_{j\e_2}(\lambda)$ is geometrically distributed with parameter $p_V(\lambda)$, we have
$$
\E[Y_j]=-\frac{1-p_V(\lambda^-)}{p_V(\lambda^-)}+\frac{1-p_V(\hat\lambda^+)}{p_V(\hat\lambda^+)}=\frac{p_V(\lambda^-)-p_V(\hat\lambda^+)}{p_V(\lambda^-)p_V(\hat\lambda^+)}.
$$
In addition, by definition of $p_V(\lambda)$, and since by~\eqref{eqn: derivative of q(lambda)} the derivative of $q(\lambda)$ is bounded on $[\frac{1}{36},\frac{35}{36}]$, it follows from the mean-value theorem that there exists a constant $C<\infty$ such that
$$
p_V(\lambda^-)-p_V(\hat\lambda^+)=(1-p)\frac{q(\hat\lambda^+)-q(\lambda^-)}{(1-q(\hat\lambda^+))(1-q(\lambda^-))}\le C(\hat\lambda^+-\lambda^-).
$$
In conclusion, for some constant $C'$ we have
$$
\E[Y_j]\le C'(\hat\lambda^+-\lambda^-)=C'\big(8s|w-\e_1|_1^{-1/3}+8s|v|_1^{-1/3}\big)\le 20C's|v|_1^{-1/3},
$$
as required. An identical argument gives the same for $Z_j$, which proves~\emph{(iv)}.
\end{proof}

%    From \eqref{eqn: derivative of q(lambda)} it follows that there exists a constant $C = C(\delta)>0$ such that
%    \begin{equation}
%    -C \leq  \frac{d}{d\lambda}p_V\left(\lambda\right) = -\dfrac{q'(\lambda)}{(1-q(\lambda))^2}\leq  0\; \text{ for } \lambda\in (\delta, 1-\delta) \label{eqn: derivative of q}.    
%    \end{equation}

%It follows that $q(\lambda)$ is bounded away from 0 and 1 for $\lambda\in(\delta, 1-\delta)$, hence, variance and $\P(X_j\leq-1)$ are bounded away from zero and infinity.
%    The expectations could be estimated as follows 
%    $$X_j = -\omega_{j\e_2}^V(\lambda^-)+\hat{\omega}^V_{(j-1)\e_2+\e_1}(\hat{\lambda}^+) = -\frac{1}{p_V(\lambda^-)}+\frac{1}{p_V(\hat{\lambda}^+)} = \dfrac{p_V(\lambda^-)-p_V(\hat{\lambda}^+)}{p_V(\hat{\lambda}^+)p_V(\lambda^+)}.$$
%    Since $\hat{\lambda}^+\geq\frac{1}{2}
%    \geq\lambda^-$, the expectation is at least zero. By \eqref{eqn: derivative of q} there exists a constant $C=C(\delta)>0$ such that 
%    $$q(\lambda^-)-q(\hat{\lambda}^+)\leq C(\lambda^--\hat{\lambda}^+)\leq Cs(|v|_1^{-\frac{1}{3}}+|n\e_+-v+\e_1|^{-\frac{1}{3}})\leq 2Cs|v|_1^{-\frac{1}{3}},$$
%    where we use that $|v|_1\leq n$.

\subsubsection{Proof of Theorem~\ref{thm: reformulated theorem}}

%\begin{proof}[Proof of Theorem~\ref{thm: reformulated theorem}]
    First, note that by increasing the constant $C$ if necessary, we may assume that $s\leq \frac{1}{18}|v|_1^{1/3}$. We henceforth assume that $v\in R_{0,n\e_+}$ and $1\le s\le\frac{1}{18}|v|_1^{1/3}$ satisfy $\min\{v_1,v_2\}\ge C$ and $|v_2-v_1|\le s|v|_1^{2/3}$. Finally, set $k=\lfloor 2s|v|_1^{2/3}\rfloor+1$, and recall that $w=n\e_+-v$.
%    and consider the event
%    $$
%    A = \{\pi(-v, n\e_+-v) \text{ is above }-k\e_2\text{ and below }k\e_2\}.
%    $$
%     Since $s\geq |v_1-v_2|\cdot |v|_1^{2/3}$, the point $k\e_2$ is above the main diagonal of the square $R_{-v, n\e_+-v}$ and $-k\e_2$ is below the main diagonal. By Theorem \ref{} we have that there exist constant $C, c>0$ such that
%    \begin{eqnarray*}
%        \P(\pi(-v, n\e_+-v)\text{ is above }-k\e_2) &=& \P(\pi(0, n\e_+)\text{ is above }v-k\e_2)\\
%        &\geq& 1-C\exp\left(-c\dfrac{(v_1-v_2+k)^3}{|v-k\e_2|_1^{2}}\right).
%    \end{eqnarray*}
%    Since $|v_1-v_2|\cdot |v|_1^{2/3}\leq s\leq \delta |v|_1$, we have $v_1-v_2+k\geq s|v|_1^{2/3}$ and $|v-k\e_2|_1 \leq (1-\delta)|v|_1$. Similar bound holds for $\pi(-v, n\e_+-v)$ staying below $k\e_2$. This way, there exist $C, c>0$ such that
%    $$\P(A)\geq 1-C\exp(-cs^3).$$

    By~\eqref{eq:piDi} we have
    \begin{equation}\label{eq:piDi_bound}
    \begin{aligned}
    \P\big(0, \e_1\in\pi(-v, w)\big)&=\P\big(D_i\ge0\text{ for }-v_2\le i\le w_2\big)\\
    &\le\P\big(D_i\ge0\text{ for }-k\le i\le k\big).
    \end{aligned}
    \end{equation}
    By definition, we have $D_0=0$, and by Lemmas~\ref{lem: assumption} and~\ref{lem: another bound on the difference of times}, with probability at least $1-2C\exp(-cs^3)$, we have  
%    We bound these increments by boundary weights in Lemmas \ref{lem: bound the difference of times} and \ref{lem: another bound on the difference of times}, and the bounds hold with high probability. The condition $\min(v_1, v_2)\geq\delta|v|_1$ is implied from $v\in Bulk(\eta)$ and $|v|_1\geq C$ and $\min(n-v_1+1, n-v_2)\geq \delta|n\e_+-v+\e_1|_1$ follows from $|v|_1\leq n$. The upper bound on $s$ follows from the condition $s\leq \delta|v|_1^{1/3}$ that we assumed at the beginning of this proof. Applying these lemmas, we get that with probability at least $1-c\exp(-cs^3)$ the following bounds hold
    \begin{eqnarray*}
        D_i \leq \sum_{j=1}^i Y_j, \text{ for } 0< i\leq k\quad \text{and}\quad
        D_i \leq \sum_{j=i+1}^{0}Z_j, \text{ for }-k\leq i<0,
        \end{eqnarray*}
    where $(Y_j)_{j\ge1}$ and $(Z_j)_{j\le0}$ are as defined in~\eqref{eq:YjZj}. Consequently, we obtain the following upper bound on the probability in the right-hand side of~\eqref{eq:piDi_bound}
    $$
    \P\bigg(\bigg\{\sum_{j=1}^iY_j\geq 0\text{ for }i = 1, ..., k\bigg\}\bigcap\bigg\{\sum_{j=i+1}^{0}Z_j\geq0\;\text{ for }i= -k, ..., -1 \bigg\}\bigg)+2C\exp(-cs^3).
    $$
    By part~\emph{(ii)} of Lemma~\ref{lem: estimating mu, delta, sigma}, the two sequences $(Y_j)_{j\ge1}$ and $(Z_j)_{j\le0}$ are independent, so that the above expression is equal to
    $$
    \P\bigg(\sum_{j=1}^iY_j\geq 0\text{ for }i = 1, ..., k\bigg)\cdot\P\bigg(\sum_{j=i+1}^{0}Z_j\geq0\;\text{ for }i= -k, ..., -1\bigg)+2C\exp(-cs^3).
    $$
    By parts~\emph{(i)} and~\emph{(iv)} of Lemma~\ref{lem: estimating mu, delta, sigma}, the two sequences $(Y_j)_{j\ge1}$ and $(Z_j)_{j\le0}$ consist of i.i.d.\ random variables with non-negative mean bounded by $Cs|v|_1^{-1/3}$, and by part~\emph{(iii)} the quantities $\P(Y_j\ge1)$ and $\Var(Y_j)$ are bounded away from zero. We obtain, hence, from Proposition~\ref{prop: random walk} that there exists a constant $C<\infty$ such that
%    Finally, in Proposition \ref{prop: random walk} we estimated the probability of a random walk staying above zero for $k$ steps. To apply the proposition we estimate parameters of the distribution of $X_1$ and $Y_1$ in Lemma \ref{lem: estimating mu, delta, sigma}. In particular, we check that variance and probability of being at least 1 are bounded away from 0 and infinity. The expectations of both random variables are not larger than $cs|v|_1^{-1/3}$ for some constant $c>0$. This way, we get that there exists a constant $c>0$ such that
     \begin{align*}
         \P\big(0, \e_1\in\pi(-v, w)\big)&\le
         %\P(D_i\geq 0\text{ for }i=-k, ..., k)&\leq
         \left(\frac{C}{\sqrt{k}}+C\,\E[Y_j]\right)\cdot\left(\frac{C}{\sqrt{k}}+C\,\E[Z_j]\right)+2C\exp(-cs^3)\\
         &\leq\, C^2\bigg(\frac{1}{\sqrt{s}}|v|_1^{-1/3}+s|v|_1^{-1/3}\bigg)^2+2C\exp(-cs^3) \le\, 2C^2\Big(s^2|v|_1^{-2/3}+\exp(-cs^3)\Big),
     \end{align*}
   where we in the last inequality used that $s\geq 1$. This completes the proof of the theorem.
 %   \end{proof}

\subsection{Proof of Theorem~\ref{th:probageod}}

We end this section deducing the theorem from the introduction from Theorems~\ref{boundoutsidebulk} and~\ref{thm: reformulated theorem}.

We begin with part~\emph{(i)}, which we derive as a consequence of Theorem~\ref{boundoutsidebulk}.
By symmetry, it will suffice to consider vertices below the main diagonal, i.e.\ satisfying $v_1\ge v_2$. Subject to increasing the constant $C$ at the end, we may further assume that $v_1-v_2\ge 4$. In this case also $v-\e_1+\e_2$ lies below the main diagonal, and if $\pi(0,n\e_+)$ is above $v-\e_1+\e_2$, then $\pi(0,n\e_+)$ cannot visit $v$, so that 
$$
\P\big(v \notin \pi(0, n \e_+)\big) \geq \P\big(\pi(0, n \e_+) \text{ is above } v - \e_1 + \e_2\big).
$$
We hence obtain, from Theorem~\ref{boundoutsidebulk}, that for some constants $C,c>0$ we have
%Up to taking a larger constant $C$ in the end, it is enough to prove the result for large enough $n$ and $v \in R_{0, n \e_+}$ such that $|v|_1$ is large enough and $\frac{|v_2 - v_1|^3}{|v|_1^2}$ is large enough. Let $n$ and $v$ be as such. By symmetry with respect to the main diagonal, we can also assume $v_1 \geq v_2$. We then apply Theorem \ref{boundoutsidebulk} to the vertex $v - \e_1 + \e_2$ to obtain
\[
\P\big(v \notin \pi(0, n \e_+)\big) \geq 1 - C\exp\bigg(-c \frac{(|v_2 - v_1| - 2)^3}{|v|_1^2}\bigg) \ge 1 - C\exp\bigg(-\frac{c}{8} \frac{|v_2 - v_1|^3}{|v|_1^2}\bigg)  ,
\]
where we in the last step have used that $\frac12|v_1-v_2|\ge 2$. This completes the proof of part~\emph{(i)}.

We proceed with the proof of part~\emph{(ii)}, which we deduce from part~\emph{(i)} and Theorem~\ref{thm: reformulated theorem}. We introduce an auxiliary variable $s\ge1$. From part~\emph{(i)} of the theorem there exist $0<c,C<\infty$ such that for all $n\ge1$ and $v\in R_{0,n\e_+}$ such that $|v|_1\le n$ and $|v_2-v_1|\ge s|v|_1^{2/3}$ we have
\begin{equation}\label{eq:vbound1}
\P\big(v\in\pi(0,n\e_+)\big)\le C\exp(-cs^3).
\end{equation}
From Theorem~\ref{thm: reformulated theorem}, together with~\eqref{eq:vbound_shift} and~\eqref{eq:vbound_split}, there exist $0<c,C<\infty$ such that for all $n\ge1$ and $v\in R_{0,n\e_+}$ such that $C\le|v|_1\le n$ and $\min\{v_1,v_2\}\ge\frac13|v|_1$ and $|v_2-v_1|\le s|v|_1^{2/3}$ we have
\begin{equation}\label{eq:vbound2}
\P\big(v\in\pi(0,n\e_+)\big)\le C\Big(s^2|v|_1^{-2/3}+\exp(-cs^3)\Big).
\end{equation}

We now set $s=(\frac{2}{3c}\log|v|_1)^{1/3}$. Since the bounds in~\eqref{eq:vbound1} and~\eqref{eq:vbound2} apply to all $v\in R_{0,n\e_+}$ such that $|v|_1\le n$ apart from a finite number not depending on $n$, it follows from~\eqref{eq:vbound1} and~\eqref{eq:vbound2} that for all such $v$ we have
$$
\P\big(v\in\pi(0,n\e_+)\big)\le C\bigg(\frac{\log|v|_1}{|v|_1}\bigg)^{2/3}+C|v|_1^{-2/3}.
$$
By increasing $C$ if necessary, the above bound applies for all $n\ge1$ and $v\in R_{0,n\e_+}$ such that $2\le|v|_1\le n$. This completes the proof of part~\emph{(ii)}, and hence of the theorem.

\subsection{The transversal fluctuation exponent via the coupling method}

We here derive as a corollary to Theorem~\ref{th:probageod} that the transversal fluctuation exponent $\xi=2/3$. This corollary will not be used in the remainder of the paper. In order to emphasise the fact that Theorem~\ref{th:probageod} applies to vertices all the way to the corners, we consider here a sharper definition of the fluctuation exponent, which considers the whole segment of the geodesic. 

For $\alpha>0$ and integers $\ell,k,n\ge1$ we let
\begin{align*}
I(\alpha,k,n)&:=k\e_++\big[-\min\{k,n-k\}^\alpha\e_2,\min\{k,n-k\}^\alpha\e_2\big],\\
B(\alpha,\ell,n)&:=\big\{v\in R_{0,n\e_+}:|v_2-v_1|\le\min\{|v|_1,|n\e_+-v|_1\}^{\alpha}+\ell\big\}.
\end{align*}
That is, for $k\le n/2$, $I(\alpha,k,n)$ is an interval of width $2k^{\alpha}$ around the diagonal, and $B(\alpha,\ell,n)$ is the set containing every point $v$ at distance at most $|v|_1^\alpha+\ell$ from the diagonal.
We now set
\begin{align*}
\xi_-&:=\sup\big\{\alpha>0:\lim_{\ell\to\infty}\sup_{n\ge2\ell}\max_{\ell\le k\le n-\ell}\P\big(\pi_n\cap I(\alpha,k,n)\neq\emptyset\big)=0\big\},\\
\xi_+&:=\inf\big\{\alpha>0:\lim_{\ell\to\infty}\inf_{n\ge\ell}\P\big(\pi_n\subseteq B(\alpha,\ell,n)\big)=1\big\}.
\end{align*}

\begin{corollary}\label{cor:xi}
    For planar last-passage percolation with geometric weights we have $\xi_-=\xi_+=2/3$.
\end{corollary}

\begin{proof}
By definition, we have $\xi_-\le\xi_+$, so it will suffice to show that $\xi_+\le2/3$ and $\xi_-\ge2/3$. We begin with the former of the two.

Take $\alpha>2/3$. By part~\emph{(i)} of Theorem~\ref{th:probageod}, and by symmetry with respect to the transverse diagonal, we have
$$
\P\big(\pi_n\not\subseteq B(\alpha,\ell,n)\big)\le\sum_{v\in B(\alpha,\ell,n)^c}\P(v\in\pi_n)\le\sum_{v\in B(\alpha,\ell,n)^c:|v|_1\le n}2\exp\bigg(-\frac{(|v|_1^\alpha+\ell)^3}{|v|_1^{2}}\bigg).
$$
Discriminating with respect to the $\ell_1$-norm gives the further upper bound
$$
\sum_{k=1}^n\sum_{v\in B(\alpha,\ell,n)^c:|v|_1=k}2\exp\big(-k^{3\alpha-2}-\ell^3\cdot k^{-2}\big)\le\sum_{k=1}^\ell2\ell\exp\big(-\ell\big)+\sum_{k=\ell}^n2k\exp\big(-k^{3\alpha-2}\big)
$$
We then set $\beta=\min\{1,3\alpha-2\}$, and note that there exists $C<\infty$ such that $\exp(-x^\beta)\le Cx^{-3}$, for all large $x$. Hence, we obtain that
$$
\P\big(\pi_n\subseteq B(\alpha,\ell,n)\big)=1-\P\big(\pi_n\not\subseteq B(\alpha,\ell,n)\big)\ge1- \frac{2C}{\ell}-2C\sum_{k=\ell}^\infty k^{-2},
$$
and so that $\xi_+\le\alpha$.

Now, take $\alpha<2/3$. By part~\emph{(ii)} of Theorem~\ref{th:probageod} we have for $k\le n/2$ that
$$
\P\big(\pi_n\cap I(\alpha,k,n)\neq\emptyset\big)\le\sum_{v\in I(\alpha,k,n)}\P(v\in\pi_n)\le 2k^{\alpha}C\bigg(\frac{\log k}{k}\bigg)^{2/3}.
$$
By symmetry with respect to the transverse diagonal, it follows that
$$
\max_{\ell\le k\le n-\ell}\P\big(\pi_n\cap I(\alpha,k,n)\le2C(\log \ell)^{2/3}\ell^{\alpha-2/3},
$$
for all sufficiently large $\ell$. Hence $\xi_-\ge\alpha$, as required.
\end{proof}

\section{Proof of Theorem \ref{th:nsenoftraveltimes}} \label{se:proof of nsen lpp}

In this section we prove noise sensitivity of travel times via Theorem~\ref{th:genbks}. In order to do this, we recall the Bernoulli encoding of the geometric random variables introduced in~\eqref{eq:encoding} and its subsequent perturbation introduced in~\eqref{eq:bitresampling}. Note that $T_n=T(0,n\e_+)$ is transformed into a function from $\{0,1\}^{\Z^2\times\N}$ to $\R$ via the Bernoulli encoding, and that the bit-resampling noise defined in~\eqref{eq:bitresampling} is equivalent to the perturbation introduced in~\eqref{def:noise}. Consequently, in order to prove Theorem~\ref{th:nsenoftraveltimes}, it follows from Theorem~\ref{th:genbks} that it suffices to show that
$$
\sum_{v\in\Z^2}\sum_{i\in\N}\I_{v,i}(T_n)^2=o\big(\Var(T_n)\big),
$$
where $\I_{v,i}(T_n)$ is the influence of the variable $X_{v,i}$ in the encoding $\omega_v=\min\{i\ge0:X_{v,i}=1\}$ on $T_n$, i.e.\
$$
\I_{v,i}(T_n)=\E\big[\big|\E[T_n(\omega)|(X_{u,j})_{(u, j) \neq (v, i)}]-T_n(\omega)\big|\big].
$$

The first step in the proof of Theorem~\ref{th:nsenoftraveltimes} will be to relate the influences of bits associated to a vertex $v$ to the probability of $v$ being on some geodesic for $T_n$.

\begin{lemma} \label{lem: bounding sum of influences}
    For every $\delta > 0$ and $p\in(0,1)$ there exists a constant $C = C(\delta, p)<\infty$ such that for all $n \geq 1$ and $v\in R_{0, n \e_+}$ we have
    \[
    \sum_{i \in \N} \I_{v, i}(T_n)^2 \leq C\, \P(v \in \pi_n)^{2 - \delta} .
    \]
\end{lemma}

\begin{proof}
Recall the encoding, from~\eqref{eq:encoding}, of $\omega=(\omega_v)_{v\in\Z^2}$ where $\omg_v = \min\{i\ge0 : X_{v, i} = 1\}$ for i.i.d\ Bernoulli random variables $(X_{v, i})_{v\in\Z^2,i \in \N}$ with parameter $p$. For $v\in\Z^2$, $i \in \N$ and $k\in\{0,1\}$ we let $\sigma_{v,i}^k(\omega)$ denote the weight configuration obtained from $\omega$ by replacing $X_{v,i}$ in the encoding by $k$, and let $\sigma_i^k(\omega_v)$ denote the weight obtained from $\omega_v$ by replacing $X_{v,i}$ in the encoding by $k$. Then, we may think of $\sigma_{v,i}^k$ as an operator on $\{0,1\}^{\Z^2\times\N}$, and coordinate $v$ of the configuration $\sigma_{v,i}^k(\omega)$ equals $\sigma_i^k(\omega_v)$.

%we let $\sigma_{v,i}^k:\{0,1\}^{\Z^2\times\N}\to\{0,1\}^{\Z^2\times\N}$ be the operator which replaces the value at coordinate $(v,i)$ by $k$. 
Let $\xi \sim \mathrm{Ber}(p)$ be a random variable independent of $(X_{v, i})_{v \in \Z^2, i \in \N}$.
%We let $\omg_{v}^{i, \xi} \in \N$ denote the  obtained by replacing $X_{v, i}$ by $\xi$ in the sequence $(X_{v, j})_{j \in \N}$. We then let $\sigma_{v, i}^{\xi}\omg \in \N^{\Z^2}$ be defined by $\sigma_{v, i}^{\xi} \omg_u = \omg_u$ for $u \neq v$ and $\sigma_{v, i}^{\xi} \omg_v = \omg_{v}^{i \xi}$. Recall that $\E^{\xi}$ is the conditional expectation with respect to every random variable but $\xi$.
By definition of the difference operators in~\eqref{def:nabla}, we have
    $$
    \nabla_{v, i} T_n(\omega) = \E^{\xi}\big[T_n\circ\sigma_{v, i}^{\xi}(\omega)\big]-T_n(\omega) = \E^{\xi}\big[T_n\circ\sigma_{v, i}^{\xi}(\omg) - T_n(\omega)\big],
    $$
    where $\E^\xi$ denotes expectation with respect to $\xi$ only. Since all coordinates $u\neq v$ of $\omega$ and $\sigma_{v, i}^{\xi}(\omega)$ coincide, the total passage time cannot change when replacing $X_{v,i}$ by $\xi$ unless $v$ belongs to $\pi_n(\omg)$ or to $\pi_n\big(\sigma_{v, i}^{\xi}(\omg)\big)$.
    Moreover, since the function $T_n(\omega)$ is Lipschitz in every coordinate $\omega_v$, the change in travel time when replacing $\omg$ by $\sigma_{v,i}^\xi(\omega)$ is bounded by $|\sigma_{i}^{\xi}(\omega_v)-\omega_v|$. We deduce
    $$
    \big|T_n\circ\sigma_{v,i}^{\xi}(\omega)-T_n(\omega)\big|\leq \big|\sigma_{i}^\xi(\omega_v)-\omega_v\big|\cdot\1_{\big\{v\in \pi_n(\sigma_{v,i}^{\xi}(\omega))\text{ or }v\in\pi_n(\omega)\big\}}.
    $$
     %Suppose $v$ belongs to $\pi_n(\omega)$ before resampling. Then the travel time along a geodesic through $v$ changes exactly by $\omega_v^{i, \xi}-\omega_v$. If $v$ belongs to $\pi_n(\sigma_{v, i}(\omega))$ as well, then the inequality follows immediately. If instead $v$ is not on the new geodesic, that means that the resampling has decreased the weight at $v$, so $\omega_v-\omega_v^{i, \xi}\geq 0$. Consequently, the new passage time $T_n(\sigma_{v, i}^{\xi}(\omega))$ cannot exceed the old one, but it must be at least the passage time of the old maximal-weight path through $v$ with the resampled weight at $v$. This way, we obtain
     %$$T_n(\sigma_{v, i}^{\xi}\omega)\geq T_n(\omega)-\omega_v+\omega_v^{i, \xi}.$$
     %If $v$ belongs to $\pi_n(\sigma_{v,i}^{\xi}\omega)$, but not to $\pi_n(\omega)$, the argument is similar.
   Hence, for the influence of the bit $X_{v, i}$ on $T_n$ we obtain that 
    \begin{align*}
      \I_{v, i}(T_n) &= \E \big[\big|\nabla_{v, i}\,T_n(\omega)\big| \big]\\
      &\leq \E\big[\big|T_n\circ\sigma_{v, i}^{\xi}(\omega)-T_n(\omega)\big|\big]\\
      &\leq \E\Big[\big|\sigma_i^{\xi}(\omega_v)-\omega_v\big|\cdot \big(\1_{\{v\in \pi_n(\sigma_{v,i}^{\xi}(\omega))\}}+\1_{\{v\in\pi_n(\omega)\}}\big)\Big].
    \end{align*} 
    Applying Hölder's inequality with conjugate exponents $\frac{2}{\delta}$ and $\frac{1}{1 - \delta/2}$, we obtain that
    \begin{equation}\label{eq:inf_bound}
      \I_{v, i}(T_n) \leq 2\,\E \Big[ \big|\sigma_i^{\xi}(\omega_v) - \omega_v\big|^{2 / \delta} \Big]^{\delta/2} \cdot\,\P\big(v\in \pi_n\big)^{1-\delta/2}.
    %\P \left(\{v\in \pi_n( \sigma_{v,i}^{\xi}\omega)\}\cup\{v\in\pi_n(\omega)\}\right)^{1 - \frac{\delta}{2}}\\
    %&\leq \left(\E\Big[|\omega_v^{i, \xi}-\omega_v|^{2 / \delta} \Big]\right)^{\frac{\delta}{2}} \big(2\P(v\in \pi_n)\big)^{1-\frac{\delta}{2}}\, . 
    \end{equation}
    
    To bound the expectation in~\eqref{eq:inf_bound}, we note that if $X_{v,j}=1$ for some $j<i$ or if $X_{v, i} = \xi$, then $\sigma_i^{\xi}(\omega_v)=\omega_v$. 
    Hence, if $\sigma_i^{\xi}(\omega_v)\neq\omega_v$, then $X_{v, j} = 0$ for all $j \leq i - 1$ and $\xi = 1-X_{v, i}$, implying that $\min\{\sigma_i^{\xi}(\omega_v), \omega_v\}=i$ and $\max\{\sigma_i^{\xi}(\omega_v), \omega_v\}$ is equal to the position of the first 1 after position $i$. The weight difference can then be written as 
    $$
    \big|\sigma_i^{\xi}(\omega_v)-\omega_v\big|=\big(\min \{j > i : X_{v, j} = 1\} - i\big)\cdot \1_{\{X_{v, j}=0 \text{ for all }j< i\text{ and } X_{v, i}\neq\xi\}}.
    $$
    Since $(X_{v, j})_{j \leq i}$ and $\xi$ are independent of $(X_{v, j})_{j > i}$, and since the difference $\min \{j > i : X_{v, j} = 1\} - i$ has the same distribution as $\omega_v+1$, we obtain that
    \begin{equation}\label{eq:inf_bound2}\begin{aligned}
    \E \Big[ \big|\sigma_i^{\xi}(\omega_v) - \omega_v\big|^{2 / \delta} \Big] &=  \E\big[ (\omega_v+1)^{2 / \delta}\big] \cdot\P\big(X_{v, j} = 0 \text{ for all } j < i \text{ and }X_{v, i} \neq \xi\big)  \\
    &= \E\big[ (\omega_v+1)^{2 / \delta}\big]\cdot 2 p(1 - p)^{i}.
    \end{aligned}\end{equation}
    
We now set $C= 2p\,\E\big[(\omg_v + 1)^{2 / \delta}\big]$, which is a finite constant depending only on $p$ and $\delta$. By~\eqref{eq:inf_bound} and~\eqref{eq:inf_bound2} it follows that
    \begin{align*}
    \sum_{i \in \N} \I_{v, i}(T_n)^2 &\leq 4\sum_{i \in \N} \Big( C (1 - p)^{i}  \Big)^{\delta}\cdot\,\P\big(v \in \pi_n\big)^{2 - \delta} \\
    &= 4C^\delta \frac{1}{1 - (1 - p)^{\delta}}\,\cdot\, \P\big(v\in \pi_n\big)^{2-\delta}.
    \end{align*}
    That concludes the proof.
\end{proof}

\begin{proof}[Proof of Theorem~\ref{th:nsenoftraveltimes}]
Recall from~\eqref{eqn:variance of Tn} that there exists $c>0$ such that $\Var(T_n) \geq c n^{2/3}$ for all large $n$.
    Hence, by Theorem~\ref{th:genbks} it will suffice to show that
    \begin{equation}\label{eq:proof_cond}
    \sum_{v \in \Z^2} \sum_{i \in \N} \I_{v, i}(T_n)^2 = o(n^{2/3}).
    \end{equation}
     Fix $0<\delta < 0.001$. By Lemma~\ref{lem: bounding sum of influences}, there exists $C=C(\delta,p)$ such that
    \begin{equation}\label{eq:proof_cond2}
    \sum_{v \in \Z^2} \sum_{i \in \N} \I_{v, i}(T_n)^2 \leq C \sum_{v \in \Z^2} \P(v \in \pi_n)^{2 - \delta} \, .
    \end{equation}
   Since only vertices inside $R_{0,n\e_+}$ may affect the outcome of $T_n$, and by the symmetry with respect to the transverse diagonal of $R_{0, n \e_+}$, from $n \e_2$ to $n \e_1$, it will suffice to consider only $v\in R_{0,n\e_+}$ such that $|v|_1\le n$. That is,
$$
        \sum_{v \in \Z^2} \P(v \in \pi_n)^{2 - \delta} \le 2\sum_{|v|_1\leq n} \P(v \in \pi_n)^{2 - \delta}\le 2\sum_{k=1}^n\sum_{|v|_1=k}\P(v \in \pi_n)^{2 - \delta},
$$
where the summation is restricted to vertices in $R_{0,n\e_+}$.

     %&= \sum_{k = 0}^n \sum_{v \in R_{0, n \n} : \n \cdot v = k} \P(v \in \pi_n)^{2 - \delta} \\
       % & \overset{\text{symmetry}}{\leq} 2 \sum_{k = 0}^{n/2} \sum_{v \in R_{0, n \n}} \P(v \in \pi_n)^{2 - \delta} \, .

    %\textcolor{red}{Use $\delta$ as $\eta$.} Consider two regions: vertices far from the main diagonal $\{|v_1-v_2|>|v|_1^{\frac{2}{3}+\delta}\}$ and vertices close to the main diagonal $\{|v_1 - v_2| \leq |v|_1^{\frac{2}{3} + \delta}\}$. %Note that both results assume \textcolor{red}{That might no longer be true, but that would simplify this computation} that  $|v|_1\geq C$ for some constant $C>0$. For points $|v|_1\leq C$ we estimate the probability by 1:
    %$$\sum\limits_{|v|_1\leq C}\P(v\in\pi_n)\leq C^2=o(n^{\frac{2}{3}}). $$

For $k$ fixed, we split the sum over vertices in $R_{0,n\e_+}$ with $|v|_1=k$ into two sums, depending on whether $v$ is `close' or `far' from the main diagonal. Let $C_k=\{v\in \Z_{\ge0}^2:|v|_1=k,|v_2-v_1|\le k^{2/3+\delta}\}$ and $F_k=\{v\in \Z_{\ge0}^2:|v|_1=k,|v_2-v_1|> k^{2/3+\delta}\}$. This gives the further bound
$$
\sum_{v \in \Z^2} \P(v \in \pi_n)^{2 - \delta} \le 2\sum_{k=1}^n\bigg[\sum_{v\in C_k}\P(v \in \pi_n)^{2 - \delta}+ \sum_{v\in F_k}\P(v \in \pi_n)^{2 - \delta}\bigg].
$$
Part~\emph{(i)} of Theorem~\ref{th:probageod} gives for $v\in F_k$ that $\P(v \in \pi_n)\le C\exp(-ck^{3\delta})$, and part~\emph{(ii)} of the same theorem gives for $v\in C_k$ that $\P(v \in \pi_n)\le Ck^{-2/3+\delta}$, for some constants $0<c,C<\infty$. Consequently, we obtain that
$$
\sum_{v \in \Z^2} \P(v \in \pi_n)^{2 - \delta} \le 2C\sum_{k=1}^n\bigg[\sum_{v\in C_k}k^{-(2 - \delta)(2/3-\delta)}+ \sum_{v\in F_k}\exp\big(-2ck^{3\delta}\big)\bigg],
$$
which is upper bounded by
$$
4C\sum_{k=1}^n\Big[k^{2/3+\delta}k^{-(2 - \delta)(2/3-\delta)}+k\exp\big(-2ck^{3\delta}\big)\Big]\le4C\sum_{k=1}^nk^{-2/3+4\delta}+C\le Cn^{1/3+4\delta}+C.
$$
This shows, via~\eqref{eq:proof_cond2}, that~\eqref{eq:proof_cond} holds, as required.
%    We sum up probabilities along the line $|v|_1=k$ for all $k$. First case that we consider is `inside the bulk', that is $v$ such that $|v|_1=k, |v_1-v_2|\leq|v|_1^{\frac{2}{3}+\delta}$. For every $k$, the number of such points is at most $2 k^{\frac{2}{3}+\delta}$ and from Theorem \ref{theorem: inside the bulk} we have $\P(v\in\pi_n)\leq |v|_1^{-(\frac{2}{3}-2\delta)}=k^{-(\frac{2}{3}-\delta)}$, hence the contribution of these vertices in the total sum is, with $\eps : = 2 \delta + \frac{2}{3} \delta - \delta^2 < \frac{1}{3}$,
%\begin{align*}
%    \sum_{k=1}^{(n+1)/2} k^{\frac{2}{3}+\delta}\left(\frac{1}{k^{\frac{2}{3}-2\delta}}\right)^{2-\delta}=\sum_{k=1}^{(n+1)/2}k^{-\frac{2}{3}+\eps}=O(n^{\frac{1}{3}+\eps})=o(n^{\frac{2}{3}}).
%\end{align*}
%    The points `outside the bulk' could be estimated in a similar manner. For $1 \leq k \leq \frac{n+1}{2}$, the number of points in $R_{0, n \e_+}$ with $|v|_1=k$ is at most $2k$. By Corollary~\ref{coroutsidebulk}, the probability $\P(v\in\pi_n)$ is bounded by $C\exp(-c|v_2 - v_1|^{3} / |v|_1^{2})=C\exp(-ck^{3\delta})$ for points such that $|v_1-v_2| > |v|_1^{\frac{2}{3}+\delta}$, where $c, C $ are positive and depend only on $\delta$ and $p$. Hence the contribution of vertices outside the bulk into the total sum is
%\begin{align*}
%    \sum_{k=1}^{(n+1)/2} k\left(Ce^{-ck^3}\right)^{2-\delta} \leq \sum_{k=1}^{(n+1)/2} o(k^{-2})=o(n^{\frac{2}{3}}).
%\end{align*}
%This concludes the proof.
\end{proof}

\begin{remark}
The sum of influences squared is smaller than the variance by a polynomial margin. Examining the exponent $\theta(p,t)=\tanh(\rho t/2)$, where $\rho$ is a constant depending on $p$, we note that for any sequence $(t_n)_{n\ge1}$ such that $t_n\log n\to\infty$, the above proof gives that
\[
\lim_{n \rightarrow + \infty} \Corr \big( T_n(\omg), T_n(\omg^{t_n}) \big) = 0.
\]
\end{remark}

%\begin{remark}
%    The same proof certainly holds for noise sensitivity of $T(0, nv)$ where $v$ is a fixed vertex in $\Z^2_{\geq 0}$. But we would also need to adapt the proofs of the intermediate results, such as Theorem~\ref{boundoutsidebulk}.
%\end{remark}

%\input{section fpp}

\section{Comparison between bit-resampling and site-resampling noise} \label{se:coupling argument}

In Section~\ref{se:proof of nsen lpp} we proved noise sensitivity of the travel times $T_n$ for the bit-resampling noise $(\omg_t)_{t \geq 0}$. We shall here compare the bit-resampling noise to the site-resampling noise .
%by establishing a covariance inequality.
On the heuristic side, for a typical vertex site-resampling results in a larger perturbation than bit-resampling. However, vertices with very large weight are more likely to be affected by bit-resampling than site-resampling at a fixed noise level.
In order to establish a covariance inequality, comparing site-resampling to bit-resampling, we note that there exists a constant $C$ such that in a square of side length $n$, none of the geometrically distributed weights will attain a value greater than $C\log n$ with probability tending to one. This is why we in Proposition~\ref{prop:couplingargument} obtain a comparison between the two noises at noise strengths that differ by a log-factor.

%As it seems `more' bits are resampled in the geometric-resampling noise, it could be expected the observables decorrelate faster with that noise than with the bit-resampling noise. This is not clear however. Indeed, for a single geometric random variable, for every $t > 0$, there is a positive probability that no bit at all has been resampled in the geometric-resampling noise, whereas almost surely, infinitely many bits have been resampled in the bit-resampling noise. Using that in a square of side $n$, the probability that at least one geometric random variable exceeds $C \log(n)$ is very small if $C$ is large enough, we can still derive an implication.

\subsection{Monotonicity with respect to resampling}

Before we attend to the proof of Proposition~\ref{prop:couplingargument}, we first formalise the idea that resampling more variables leads to a lower covariance in a general setting.
Let $I$ be a finite set and $Y=(Y_i)_{i \in I}$ and $Y'=(Y'_i)_{i \in I}$ two independent vectors of independent random variables, all drawn from some common distribution. For every subset $S\subseteq I$ we define another vector $Y^S=(X_i^S)_{i\in I}$ where
\begin{equation*}
Y^S_i := \left\{
\begin{aligned}
Y'_i &&& \text{if } i \in S, \\
Y_i &&& \text{otherwise}.
\end{aligned}
\right. 
\end{equation*}
%Additionally, for fixed $S \subseteq I$ we denote by $X_S$ the family of random variables $(X_i)_{i \in S}$.
The idea that resampling more variables leads to a smaller correlation is captured by the following lemma.
%The proof relies on standard use of conditional expectation and Jensen's inequality, as well as an identity similar in nature to $\E[ f(\omg) f(\omg_{2t})] = \E[(P_t f)^2]$ (See (\ref{eq:cov_id}) in Section \ref{se:nsen}).

\begin{lemma} \label{le:covineq}
Let $f:\R^I\to\R$ be any function such that $\E[f(Y)^2]<\infty$. Let $\mathscr{S}$ and $\tilde{\mathscr{S}}$ be random subsets of $I$, independent of $Y$ and $Y'$. Suppose that $\mathscr{S} \subseteq \tilde{\mathscr{S}}$ almost surely, then
\[
\Cov \big( f( Y ), f ( Y^{\mathscr{S}}) \big) \geq \Cov \big( f ( Y ), f ( Y^{\tilde{\mathscr{S}}}) \big).
\]
\end{lemma}

\begin{proof}
By equality in law of $Y$, $Y^{\mathscr{S}}$ and $Y^{\tilde{\mathscr{S}} }$, we have $f(Y)$, $f(Y^{\mathscr{S}})$ and $f(Y^{\tilde{\mathscr{S}}})$ are equal in mean, so
\begin{equation}\label{eq:covar1}
\Cov \big( f ( Y ), f ( Y^{\mathscr{S}}) \big) - \Cov \big( f( Y ), f ( Y^{\tilde{\mathscr{S}}}) \big) = \E \big[ f ( Y ) f ( Y^{\mathscr{S}}) \big] - \E \big[ f ( Y) f ( Y^{\tilde{\mathscr{S}}}) \big]  .
\end{equation}
Summing over all $S \subseteq \tilde{S}\subseteq I$, using that $\mathscr{S} \subseteq \tilde{\mathscr{S}}$ a.s.\ and that $(Y,Y')$ and $(\mathscr{S}, \tilde{\mathscr{S}})$ are independent,
\begin{equation}\label{eq:covar2}
\E \big[ f ( Y) f ( Y^{\mathscr{S}}) \big] - \E \big[ f( Y ) f ( Y^{\tilde{\mathscr{S}}}) \big] = \sum_{S \subseteq \tilde{S}}\P\big((\mathscr{S},\tilde{\mathscr{S}}) = (S,\tilde{S})\big) \Big( \E \big[f(Y) f(Y^{S}) \big] - \E \big[ f(Y) f(Y^{\tilde{S}}) \big]\Big)  . 
\end{equation}
We now claim that for fixed sets $S \subseteq \tilde{S}\subseteq I$ we have
\begin{equation}\label{eq:covar3}
 \E \big[f(Y) f(Y^{S}) \big]  \geq \E\big[f(Y) f(Y^{\tilde{S}})\big],
\end{equation}
which via~\eqref{eq:covar1} and~\eqref{eq:covar2} proves the lemma.

Below we shall write $Y_S$ for the collection of variables $(Y_i)_{i\in S}$. Since $\tilde{S}^c \subseteq S^c$, we have
\[
\E \big[ f(Y) | Y_{\tilde{S}^c} \big] = \E \big[ \E \big[ f(Y) | Y_{S ^c}  \big] \big| Y_{\tilde{S}^c}\big]  ,
\]
which via Jensen's inequality implies that
\begin{equation}\label{eq:covar4}
\E \Big[ \E \big[ f(Y) | Y_{\tilde{S}^c} \big]^2 \Big] \leq \E \Big[ \E \big[ f(Y) | Y_{S ^c}  \big]^2 \Big]  .
\end{equation}
Finally, since $Y_S$ and $Y'_S$ are independent and identically distributed, and independent of $Y_{S^c}$, we have
\[
\E\big[f(Y) f(Y^{S})\big] = \E\Big[ \E\big[f(Y) f(Y^{S}) |Y_{S^c}\big]\Big]= \E\Big[ \E\big[f(Y)|Y_{S^c}\big]\E\big[ f(Y^{S}) |Y_{S^c}\big]\Big]  = \E\Big[ \E\big[f(Y) | Y_{S^c}\big]^2\Big]  .
\]
Together with~\eqref{eq:covar4}, this proves~\eqref{eq:covar3}, as required.
\end{proof}

\subsection{Proof of Proposition~\ref{prop:couplingargument}}

We shall construct a coupling of the bit-resampling and site-resampling noises, defined in~\eqref{eq:bitresampling} and~\eqref{eq:siteresamping} respectively. In the case of bit-resampling, the construction below will be identical to that in~\eqref{eq:bitresampling}, whereas for site-resampling, the construction below will differ from the construction in~\eqref{eq:siteresamping} in order to create a suitable coupling between the two.
Before we begin we fix a constant $C<\infty$ such that for all $n\ge1$
%Let $\kappa >0$ be such that $\sqrt{\E[T_n^4]} / \Var(T_n) = o (n^{\kappa})$ when $n \rightarrow + \infty$ (the existence of such $\kappa$ stems from (\ref{eqn:variance of Tn}) \textcolor{red}{Reference to introduction} and $\E[T_n^4] \leq 4 \E[\sum_{v} \omg_v^4]$ where the sum is over vertices in $R_{0, n \e_+}$). The geometric random variables having exponential tail, let $C$ be such that for large enough $n$,
\begin{equation}\label{eq:Cdef}
\P\big( \exists v \in R_{0, n\e_+} : \omg_v \geq C\, {\log n}\big) \leq Cn^{-3}.
\end{equation}
This can be done since the geometric distribution has an exponential tail.

Let $(X_{v, i}, X'_{v, i}, U_{v, i})_{v \in \Z^2, i \in \N}$ be a family of mutually independent random variables such that $X_{v, i}$ and $X_{v, i}'$ are Bernoulli distributed with parameter $p \in (0, 1)$, and $U_{v, i}$ is exponential distributed with parameter $1$, for all $v \in \Z^2$ and $i \in \N$. 
Given $n \ge1$, and with $C$ as in~\eqref{eq:Cdef}, we let $M := \lfloor C\log(n) \rfloor$ and set, for $v\in\Z^2$,
\[
\tilde{U}_v := M \Big( \min \limits_{0 \leq i \leq M-1} U_{v, i} \Big),
\]
which again is exponentially distributed with parameter $1$. For $t\ge0$ we let
\begin{equation*}
X^t_{v, i} := \left\{
\begin{aligned}
X_{v, i} &&& \text{if }t < U_{v, i}; \\
X'_{v, i} &&& \text{if } t \geq U_{v, i};
\end{aligned}
\right.
\quad\text{and}\quad
\tilde{X}_{v, i} := \left\{
\begin{aligned}
X_{v, i} &&& \text{if }t < \tilde{U}_v; \\
X'_{v, i} &&& \text{if } t \geq \tilde{U}_v.
\end{aligned}
\right.
\end{equation*}
Finally, we specify $(\omg^t)_{t \geq 0}$ and $(\tilde{\omg}^t)_{t \geq 0}$ by
\begin{equation}\label{eq:coupling}
\begin{aligned}
\omg^t_v &:= \min \big\{i \in \N : X_{v, i}^t = 1 \big\}  , \\
\tilde{\omg}_v^t &:= \min \big\{i \in \N : \tilde{X}_{v, i}^t = 1 \big\} .
\end{aligned}
\end{equation}
We emphasise that the construction of $(\omega^t)_{t\ge0}$ is identical to that in~\eqref{eq:bitresampling} but that the construction of $(\tilde\omega^t)_{t\ge0}$ differs from that given in~\eqref{eq:siteresamping} in order to create a simultaneous construction of the two processes on a common probability space. The process $(\tilde\omega^t)_{t\ge0}$, as defined in~\eqref{eq:coupling}, is nevertheless equal in distribution to that given in~\eqref{eq:siteresamping}.

Given a configuration $\omega=(\omega_v)_{v\in\Z^2}$ we let $\omega\wedge M$ be the configuration defined by $(\omega\wedge M)_v=\min\{\omega_v,M\}$ for all $v\in\Z^2$.
Note that, for each $t > 0$, every bit that has been resampled in $(X_{v, i}^t)_{v \in \Z^2, i < M}$ has also been resampled in $(\tilde{X}^{M t}_{v, i})_{v \in \Z^2, i < M}$. Lemma~\ref{le:covineq}, applied to $I = R_{0, n \e_+} \times \{0, \dots, M-1\}$, $\mathscr{S} = \{(v, i) \in I : t \geq U_{v, i}\}$ and $\tilde{\mathscr{S}} = \{(v, i) \in I : Mt \geq \tilde{U}_{v, i} \}$, thus gives that
\begin{equation} \label{eq:covineq}
\Cov \big(T_n(\omg \wedge M), T_n(\omg^t \wedge M) \big) \geq \Cov \big(T_n(\omg \wedge M), T_n(\tilde{\omg}^{M t} \wedge M) \big)  .
\end{equation}

We next claim that
\begin{equation}\label{eq:covcomp}
\big|  \Cov \big(T_n(\omg), T_n(\omg^t) \big) - \Cov \big(T_n(\omg \wedge M), T_n(\omg^t \wedge M) \big)\big| = o(1),
\end{equation}
and that the analogous expression for $\tilde{\omg}^{Mt}$, instead of $\omg^t$, also holds. From~\eqref{eq:covcomp}, for $\tilde\omega^{Mt}$, it follows that
$$
\Cov \big( T_n(\omg), T_n( \tilde{\omg}^{Mt}) \big) \leq \Cov \big(T_n(\omg \wedge M), T_n(\tilde{\omg}^{Mt} \wedge M) \big) + o(1),
$$
which via~\eqref{eq:covineq} gives
$$
\Cov \big( T_n(\omg), T_n( \tilde{\omg}^{Mt}) \big) \leq\Cov \big(T_n(\omg \wedge M), T_n(\omg^t \wedge M) \big) + o (1).
$$
Another application of~\eqref{eq:covcomp} now gives
$$
\Cov \big( T_n(\omg), T_n( \tilde{\omg}^{Mt}) \big) \leq \Cov \big(T_n(\omg), T_n(\omg^t) \big) + o(1),
$$
which would complete the proof of the proposition.

It remains to show that~\eqref{eq:covcomp} holds as stated, as well as for $\tilde\omega^{Mt}$ instead of $\omega^t$. Since the computations are similar for $\omg^t$ and $\tilde{\omg}^{Mt}$, we present them for $\omg^t$ only. Let $A_t$ denote the event that there exists $v\in R_{0,n\e_+}$ for which $\omega_v^t\ge M$. We write $A$ instead of $A_0$ for brevity. 
%\[
%A := \{ \sigma \in \N^{\Z^2} : \exists v \in R_{0, n\e_+}, \sigma_v \geq M\}  .
%\]
Since $T_n$ only depends on the weights associated to $v \in R_{0, n \e_+}$, we have on the event $A$ that $T_n(\omg) = T_n(\omg \wedge M)$. Using that $T_n(\omg) - T_n(\omg \wedge M) \leq T_n(\omg)$ and the Cauchy-Schwarz inequality, we deduce that
\begin{align*}
\E \big[ T_n(\omg)\big] - \E\big[T_n(\omg \wedge M) \big] &= \E \big[\big(T_n(\omg) - T_n(\omg \wedge M) \big)\1_{A} \big] \leq \E\big[  T_n(\omega)\1_{A}\big] \leq \sqrt{\P(A)\,\E[T_n^2]}  .
\end{align*}
Via the identity $a^2 - b^2 = (a-b)(a+b)$ and the inequality $T_n(\omega\wedge M)\le T_n(\omega)$, we deduce further that
%\E[T_n] \sqrt{\E[T_n^2]} \leq \sqrt{\E[T_n^4]}$, we deduce similarly that
\[
\E \big[T_n(\omg)^2\big] - \E\big[T_n (\omg \wedge M)^2\big]\le 2\,\E[T_n]\sqrt{\P(A)\,\E[T_n^2]} \leq 2 \sqrt{\P(A)\,\E[T_n^4]}  .
\]
An analogous argument gives also that
\begin{align*}
\E \big[T_n(\omg)T_n(\omg^t)\big]& - \E\big[T_n(\omg \wedge M) T_n(\omg^t \wedge M) \big] \leq \E\big[T_n(\omega)T_n(\omega^t)\1_{A\cup A^t}\big]\\
&\le\sqrt{\P(A \cup A_t)\,\E\big[T_n(\omg)^2 T_n(\omg^t)^2\big]} \leq 2 \sqrt{\P(A)\,\E[T_n^4]}  ,
\end{align*}
and hence that
$$
\big|  \Cov \big(T_n(\omg), T_n(\omg^t) \big) - \Cov \big(T_n(\omg \wedge M), T_n(\omg^t \wedge M) \big)\big| \leq 4 \sqrt{\P(A)\,\E[T_n^4]}.
$$
By~\eqref{eq:Cdef} we have $\P(A)\le Cn^{-3}$, and since
$$
\E[T_n^4]\le\E\bigg[\Big(\sum_{v\in R_{0,n\e_+}}\omega_v\Big)^4\bigg]\le15n^2\,\E[\omega^4],
$$
we conclude that~\eqref{eq:covcomp} holds, as required.

\begin{remark}
    The comparison between bit-resampling and site-resampling exhibited in Proposition~\ref{prop:couplingargument} is not exclusive to last-passage percolation. A similar statement would hold for any sequence $(f_n)_{n\ge1}$ of functions $f_n:\N^n\to\R$ which takes as an input independent geometrically distributed random variables, as long as the ratio $\sqrt{\E[f_n^4]}/\Var(f_n)$ and the number of geometric variables $f_n$ depends on grow at most polynomially in $n$.
\end{remark}

\section{Open problems}\label{se:open}

We end this paper with a few suggestions for further research. We begin with the obvious problem, discussed already in the introduction, of pinning down the correct threshold for noise sensitivity to occur. 

\begin{problem}
    Prove part~(ii) of Conjecture~\ref{conj:threshold}, i.e.\ that geometric last-passage percolation is noise sensitive, as $n\to\infty,$ for $t\gg n^{-1/3}$.
\end{problem}

Our main theorem shows that geometric last-passage percolation is noise sensitive with respect to the bit-resampling noise. Although we come close, as discussed in Section~\ref{se:coupling argument}, our approach is not enough to establish noise sensitivity with respect to the site-resampling noise.

\begin{problem}\label{prob:siteresampling}
    Show that for every $t>0$ we have $\Corr\big(T_n(\omega)),T_n(\tilde\omega^t)\big)\to0$ as $n\to\infty$.
\end{problem}

One way to approach the above problem would be to prove a version of Theorem~\ref{th:genbks} with respect to the site-resampling noise, as defined in~\eqref{eq:siteresamping}. It is possible that such a result would only hold for a subclass of functions, since hypercontractivity fails in general; see Section~\ref{se:geometric resampling not hyperc}.

\begin{problem}
    Prove a version of the BKS theorem for a suitable class of function under site-resampling.
\end{problem}

Closely related to noise sensitivity is the question about exceptional times in dynamical percolation. In dynamical Bernoulli percolation, edges (or sites) are resampled according to independent Poisson processes. The central question is whether, at critical, there exist exceptional times at which an infinite connected component appears. (In dimensions $d=2$ and $d\ge11$, at a fixed time, an infinite connected component is known not to exist at criticality, almost surely.) Dynamical percolation was introduced by Häggström, Peres and Steif~\cite{hagperste97}, who showed that no exceptional times exist in dynamical percolation in sufficiently high dimensions. Interestingly, exceptional times were show to exist in two dimensions by Schramm and Steif~\cite{schste10} and Garban, Pete and Schramm~\cite{garpetsch10}; see also~\cite{tasvan23}.

A natural candidate for an exceptional event in the context of spatial growth, substituting the occurrence of an infinite connected component, is the existence of bi-infinite geodesics, or bigeodesics. A {\bf bigeodesic} is a bi-infinite directed nearest-neighbour path $(x_k)_{k\in\Z}$ with the property that every finite segment is a geodesic between the endpoints. Bigeodesics always exist, as every horizontal or vertical line is trivially a bigeodesic. However, non-trivial bigeodesics are known not to exist. This was first proved by Basu, Hoffman and Sly~\cite{bashofsly22} for exponential last-passage percolation, and an alternative proof was later given in~\cite{balbussep20}. The argument of the latter was later extended to rule out non-trivial bigeodesics in geometric last-passage percolation in~\cite{grojanras25}.

The question of exceptional times at which non-trivial bigeodesics exist was recently considered in work of Bhatia~\cite{bhatia1,bhatia2}. The existence of exceptional times at which a bigeodesics appears was not confirmed in this work, but bounds were given that shows that if exceptional times do not exist, then they just barely do not exist. At the same time, the author conjectured that if there are exceptional times at which bigeodesics exist, then the Hausdorff dimension of this set of times is zero.

\begin{problem}
    Determine whether there exist exceptional times at which the origin is contained in a non-trivial bigeodesic.
\end{problem}

In loose terms, a sequence of Boolean functions is noise sensitive if we cannot tell the outcome of the function when only partial information, selected randomly, is available. Benjamini, Kalai and Schramm~\cite{benkalsch99} also posed some questions regarding partial information selected deterministically. A question of this type in the context of last-passage percolation would be the following. Consider again the encoding in~\eqref{eq:encoding}, and suppose that for all $v\in\Z^2$ we are told the value of all bits in the encoding of the weight at $v$ except for the first one. That is, for all vertices $v\in\Z^2$, we are told the value of the weight at $v$ in case it is non-zero, but we do not know if it is non-zero or not. Based on this information, what can we say about $T_n$?

Let $\mathcal{F}$ be the $\sigma$-algebra generated by all bits $(X_{v,i})_{v\in\Z^2,i\ge1}$. The information obtained by observing $\mathcal{F}$ can then be measured by $\Var\big(\E[T_n|\mathcal{F}]\big)$. If $\Var\big(\E[T_n|\mathcal{F}]\big)$ is small in comparison to $\Var(T_n)$, then observing $\mathcal{F}$ gives little information about $T_n$.

\begin{problem}
    Show that $\Var\big(\E[T_n|\mathcal{F}]\big)=o\big(\Var(T_n)\big)$ as $n\to\infty$.
\end{problem}

Last-passage percolation with geometric weights is one out of a few exactly solvable models known to belong to the KPZ class of universality. As mentioned in the introduction, it is possible to encode an exponential random variable either via a Bernoulli encoding or a Gaussian encoding, and hence prove noise sensitivity via Theorem~\ref{th:genbks} or Theorem~\ref{th:gaussbks}, for the respective notion of noise. One could argue that the exponential distribution is more naturally perturbed via site-resampling. We expect that proving noise sensitivity of exponential last-passage percolation via site-resampling is roughly equivalent to solve Problem~\ref{prob:siteresampling}.
A more interesting extension could be to establish noise sensitivity for a continuum model of last-passage percolation, such as Poisson last-passage percolation (also known as the Hammersley process) or the related longest increasing subsequence problem of a random perturbation.

Let $\eta$ be a unit Poisson point process on $\R^2$ and let $L_n(\eta)$ denote the maximal number of points than one may collect when travelling from $(0,0)$ to $(n,n)$ along an up-right directed path. For $t>0$ we let $\eta^t$ be obtained from $\eta$ by first thinning $\eta$ by parameter $1-t$ and then super-positioning an independent Poisson point process $\eta'$ of intensity $t$, sometimes written as $\eta^t=(1-t)\eta+t\eta'$.

\begin{problem}
    Show that Poisson last-passage percolation is noise sensitive in the sense that for every $t>0$ we have $\Corr\big(L_n(\eta),L_n(\eta^t)\big)\to0$ as $n\to\infty$.
\end{problem}

Let us finally propose a problem asking for a limiting distribution for the transversal fluctuations of a geodesic. It was proved in~\cite{joh00} that there exist constants $a$ and $b$ such that $an^{-1/3}(T_n-bn)$ converges in distribution to the Tracy-Widom law. While the correct scaling of transversal fluctuations is also known, a similar limiting distribution remains unknown. One way to address this is in terms of the last position on the vertical axis visited by some geodesic from $-n\e_+$ to $n\e_+$. That is, since
$$
T(-n\e_+,n\e_+)=\max_i\big[T(-n\e_+,i\e_2)+T(i\e_2+\e_1,n\e_+)\big],
$$
we may consider
$$
Z_n:=\max\big\{i\in\Z:T(-n\e_+,n\e_+)=T(-n\e_+,i\e_2)+T(i\e_2+\e_1,n\e_+)\big\},
$$
where we take the maximum since the equality may hold for various values of $i$.

\begin{problem}
    Determine the limiting distribution, if it exists, of $n^{-2/3}Z_n$ as $n\to\infty$.
\end{problem}

It is well-known, since the work of~\cite{johansson00b,baideimclmilzho01}, that $Z_n$ is typically of order $n^{2/3}$.

\printbibliography

@article{benkalsch99,
 author = {Benjamini, Itai and Kalai, Gil and Schramm, Oded},
 title = {Noise sensitivity of Boolean functions and applications to percolation},
 journal = {Publications Mathématiques de l'Institut des Hautes \'{E}tudes Scientifiques},
 volume = {90},
 year = {1999}
}

@Article{kelkin13,
 Author = {Keller, Nathan and Kindler, Guy},
 Title = {Quantitative relation between noise sensitivity and influences},
 Journal = {Combinatorica},
 Volume = {33},
 Number = {1},
 Year = {2013}
}

@Book{cha14,
 Author = {Chatterjee, Sourav},
 Title = {Superconcentration and related topics},
 Series = {Springer Monographs in Mathematics},
 Year = {2014},
 Publisher = {Cham: Springer}
}

@Article{ahldeisfr24,
 Author = {Ahlberg, Daniel and Deijfen, Maria and Sfragara, Matteo},
 Title = {From stability to chaos in last-passage percolation},
 Journal = {Bulletin of the London Mathematical Society},
 Volume = {56},
 Number = {1},
 Year = {2024}
}

@InCollection{corled12,
 Author = {Cordero-Erausquin, Dario and Ledoux, Michel},
 Title = {Hypercontractive measures, {Talagrand}'s inequality, and influences},
 BookTitle = {Geometric aspects of functional analysis. Proceedings of the Israel seminar (GAFA) 2006--2010},
 Year = {2012},
 Publisher = {Berlin: Springer}
}

@unpublished{roshan20,
  title={Ramon van Handel’s Remarks on the Discrete Cube},
  author={Rosenthal, Gregory},
  note={Lecture notes},
  year={2020},
  url={https://www.cs.toronto.edu/~rosenthal/RvH_discrete_cube.pdf}
}

@unpublished{ahldeisfr1,
author="Ahlberg, Daniel and Deijfen, Mia and Sfragara, Matteo",
title="Chaos, concentration and multiple valleys in first-passage percolation",
note="\emph{Ann. Inst. Henri Poincar\'e Probab. Stat.}, to appear"
}

@incollection {rassoul18,
    AUTHOR = {Rassoul-Agha, Firas},
     TITLE = {Busemann functions, geodesics, and the competition interface
              for directed last-passage percolation},
 BOOKTITLE = {Random growth models},
    SERIES = {Proc. Sympos. Appl. Math.},
    VOLUME = {75},
     PAGES = {95--132},
 PUBLISHER = {Amer. Math. Soc., Providence, RI},
      YEAR = {2018},
}

@article {brogarste13,
    AUTHOR = {Broman, Erik I. and Garban, Christophe and Steif, Jeffrey E.},
     TITLE = {Exclusion sensitivity of {B}oolean functions},
   JOURNAL = {Probab. Theory Related Fields},
  FJOURNAL = {Probability Theory and Related Fields},
    VOLUME = {155},
      YEAR = {2013},
    NUMBER = {3-4},
     PAGES = {621--663},
}

@article {johansson00b,
    AUTHOR = {Johansson, Kurt},
     TITLE = {Transversal fluctuations for increasing subsequences on the
              plane},
   JOURNAL = {Probab. Theory Related Fields},
  FJOURNAL = {Probability Theory and Related Fields},
    VOLUME = {116},
      YEAR = {2000},
    NUMBER = {4},
     PAGES = {445--456},
}

@article {ahlbrogrimor14,
    AUTHOR = {Ahlberg, Daniel and Broman, Erik and Griffiths, Simon and
              Morris, Robert},
     TITLE = {Noise sensitivity in continuum percolation},
   JOURNAL = {Israel J. Math.},
  FJOURNAL = {Israel Journal of Mathematics},
    VOLUME = {201},
      YEAR = {2014},
    NUMBER = {2},
     PAGES = {847--899},
}

@article {borlugzhi20,
    AUTHOR = {Bordenave, Charles and Lugosi, G\'{a}bor and Zhivotovskiy, Nikita},
     TITLE = {Noise sensitivity of the top eigenvector of a {W}igner matrix},
   JOURNAL = {Probab. Theory Related Fields},
  FJOURNAL = {Probability Theory and Related Fields},
    VOLUME = {177},
      YEAR = {2020},
    NUMBER = {3-4},
     PAGES = {1103--1135},
}

@article {hagperste97,
    AUTHOR = {H{\"a}ggstr{\"o}m, Olle and Peres, Yuval and Steif, Jeffrey
              E.},
     TITLE = {Dynamical percolation},
   JOURNAL = {Ann. Inst. H. Poincar\'e Probab. Statist.},
  FJOURNAL = {Annales de l'Institut Henri Poincar\'e. Probabilit\'es et
              Statistiques},
    VOLUME = {33},
      YEAR = {1997},
    NUMBER = {4},
     PAGES = {497--528},
}

@article {schramm00,
    AUTHOR = {Schramm, Oded},
     TITLE = {Scaling limits of loop-erased random walks and uniform
              spanning trees},
   JOURNAL = {Israel J. Math.},
  FJOURNAL = {Israel Journal of Mathematics},
    VOLUME = {118},
      YEAR = {2000},
     PAGES = {221--288},
}

@incollection {eden61,
    AUTHOR = {Eden, Murray},
     TITLE = {A two-dimensional growth process},
 BOOKTITLE = {Proc. 4th {B}erkeley {S}ympos. {M}ath. {S}tatist. and {P}rob.,
              {V}ol. {IV}},
     PAGES = {223--239},
 PUBLISHER = {Univ. California Press, Berkeley, Calif.},
      YEAR = {1961},
}

@Article{tasvan23,
 Author = {Tassion, Vincent and Vanneuville, Hugo},
 Title = {Noise sensitivity of percolation via differential inequalities},
 Journal = {Proceedings of the London Mathematical Society},
 Volume = {126},
 Number = {4},
 Year = {2023}
}

@Article{diasal96,
 Author = {Diaconis, P. and Saloff-Coste, L.},
 Title = {Logarithmic {Sobolev} inequalities for finite {Markov} chains},
 Journal = {Annals of Applied Probability},
 Volume = {6},
 Number = {3},
 Year = {1996}
}

@article{eldgro22,
 author = {Eldan, Ronen and Gross, Renan},
 title = {Concentration on the {Boolean} hypercube via pathwise stochastic analysis},
 journal = {Inventiones mathematicae},
 volume = {230},
 number = {3},
 year = {2022}
}

@article{cardon25,
 author = {Caravenna, Francesco and Donadini, Anna},
 title = {Enhanced noise sensitivity, 2D directed polymers and Stochastic Heat Flow},
 year = {2025},
 journal = {Preprint, {arxiv}:2507.10379}
}

@article{demelb25,
      title={Noise sensitivity and variance lower bound for minimal left-right crossing of a square in first-passage percolation}, 
      author={Barbara Dembin and Dor Elboim},
      year={2025},
      journal={Preprint, {arxiv}:2505.03211},

}

@article {ahlriv26,
    AUTHOR = {Ahlberg, Daniel and de la Riva, Daniel},
     TITLE = {Is `being above the median' a noise sensitive property?},
   JOURNAL = {Trans. Amer. Math. Soc.},
  FJOURNAL = {Transactions of the American Mathematical Society},
    VOLUME = {379},
      YEAR = {2026},
    NUMBER = {1},
     PAGES = {681--711},
}

@article{karparzha86,
  title = {Dynamic Scaling of Growing Interfaces},
  author = {Kardar, Mehran and Parisi, Giorgio and Zhang, Yi-Cheng},
  journal = {Physical Review Letters},
  volume = {56},
  issue = {9},
  pages = {889--892},
  year = {1986},
  publisher = {American Physical Society}
}

@Misc{hamwel65,
 Author = {Hammersley, J. M. and Welsh, D. J. A.},
 Title = {First-passage percolation, subadditive processes, stochastic networks, and generalized renewal theory},
 Year = {1965},
 Language = {English},
 HowPublished = {Bernoulli-{Bayes}-{Laplace}, {Anniversary} {Vol}., {Proc}. {Int}. {Res}. {Semin}., {Berkeley} 1963, 61-110 (1965).}
}

@inproceedings{balbussep20,
  title={Non-existence of bi-infinite geodesics in the exponential corner growth model},
  author={Bal{\'a}zs, M{\'a}rton and Busani, Ofer and Sepp{\"a}l{\"a}inen, Timo},
  booktitle={Forum of Mathematics, Sigma},
  volume={8},
  pages={e46},
  year={2020},
  organization={Cambridge University Press}
}

@Article{schste10,
 Author = {Schramm, Oded and Steif, Jeffrey E.},
 Title = {Quantitative noise sensitivity and exceptional times for percolation},
 Journal = {Annals of Mathematics. Second Series},
 Volume = {171},
 Number = {2},
 Year = {2010}
}

@article{bhatia1,
  title={Exceptional times when bigeodesics exist in dynamical last passage percolation},
  author={Bhatia, Manan},
  journal={arXiv preprint arXiv:2504.12293},
  year={2025}
}

@article{bhatia2,
  title={Geodesic switches and exceptional times in dynamical Brownian last passage percolation},
  author={Bhatia, Manan},
  journal={arXiv preprint arXiv:2510.27589},
  year={2025}
}

@article {basganzha21,
    AUTHOR = {Basu, Riddhipratim and Ganguly, Shirshendu and Zhang, Lingfu},
     TITLE = {Temporal correlation in last passage percolation with flat
              initial condition via {B}rownian comparison},
   JOURNAL = {Comm. Math. Phys.},
  FJOURNAL = {Communications in Mathematical Physics},
    VOLUME = {383},
      YEAR = {2021},
    NUMBER = {3},
     PAGES = {1805--1888},
}

@article {elnjansep23,
    AUTHOR = {Emrah, Elnur and Janjigian, Christopher and
              Sepp\"{a}l\"{a}inen, Timo},
     TITLE = {Optimal-order exit point bounds in exponential last-passage
              percolation via the coupling technique},
   JOURNAL = {Probab. Math. Phys.},
  FJOURNAL = {Probability and Mathematical Physics},
    VOLUME = {4},
      YEAR = {2023},
    NUMBER = {3},
     PAGES = {609--666},
}

@article {catgro06,
    AUTHOR = {Cator, Eric and Groeneboom, Piet},
     TITLE = {Second class particles and cube root asymptotics for
              {H}ammersley's process},
   JOURNAL = {Ann. Probab.},
  FJOURNAL = {The Annals of Probability},
    VOLUME = {34},
      YEAR = {2006},
    NUMBER = {4},
     PAGES = {1273--1295},
}

@Article{garpetsch10,
 Author = {Garban, Christophe and Pete, G{\'a}bor and Schramm, Oded},
 Title = {The {Fourier} spectrum of critical percolation},
 Journal = {Acta Mathematica},
 Volume = {205},
 Number = {1},
 Year = {2010}
}

@Book{odonnell,
 Author = {O'Donnell, Ryan},
 Title = {Analysis of {Boolean} functions},
 Year = {2014},
 Publisher = {Cambridge: Cambridge University Press}
}

@Article{baideimclmilzho01,
 Author = {Baik, Jinho and Deift, Percy and McLaughlin, Kenneth T.-R. and Miller, Peter and Zhou, Xin},
 Title = {Optimal tail estimates for directed last passage site percolation with geometric random variables},
 Journal = {Advances in Theoretical and Mathematical Physics},
 Volume = {5},
 Number = {6},
 Pages = {1207--1250},
 Year = {2001}
}

@Book{garste15,
 Author = {Garban, Christophe and Steif, Jeffrey E.},
 Title = {Noise sensitivity of {Boolean} functions and percolation},
 Series = {Institute of Mathematical Statistics Textbooks},
 Volume = {5},
 Year = {2015},
 Publisher = {Cambridge: Cambridge University Press}
}

@article{schsmi11,
author = {Schramm, Oded and Smirnov, Stanislav and Garban, Christophe},
title = {{On the scaling limits of planar percolation}},
volume = {39},
journal = {The Annals of Probability},
number = {5},
publisher = {Institute of Mathematical Statistics},
pages = {1768 -- 1814},
year = {2011}
}

@article{dauortvir22,
author = {Dauvergne, Duncan and Ortmann, Janosch and Vir{\'a}g, B{\'a}lint},
title = {{The directed landscape}},
volume = {229},
journal = {Acta Mathematica},
number = {2},
publisher = {Institut Mittag-Leffler},
pages = {201 -- 285},
year = {2022}
}

@misc{himpar25,
      title={The directed landscape is a black noise},
      author={Zoe Himwich and Shalin Parekh},
      year={2025},
      howpublished={Preprint, {arXiv}:2404.16801 [math.{PR}] (2025)}
}

@misc{sep09,
 author = {Sepp{\"a}l{\"a}inen, Timo},
 title = {Lecture Notes on the Corner Growth Model},
 year = {2009},
 url = {https://people.math.wisc.edu/~tseppalainen/cornergrowth-book/etusivu.html}
}

@article{joh00,
 author={Johansson, Kurt},
 title={Shape Fluctuations and Random Matrices},
 journal={Communications in Mathematical Physics},
 volume={209},
 issue={2},
 pages={437 -- 476},
 year={2000}
}

@article{grojanras25,
 author={Groathouse, Sean and Janjigian, Christopher and Rassoul-Agha, Firas},
 title={Non-Existence of Non-Trivial Bi-Infinite Geodesics in Geometric Last Passage Percolation},
 journal={Journal of Statistical Physics},
 volume={192},
 issue={6},
 pages={81},
 year={2025}
}

@article{bashofsly22,
 author={Basu, Riddhipratim and Hoffman, Christopher and Sly, Allan},
 title={Nonexistence of Bigeodesics in Planar Exponential Last Passage Percolation},
 journal={Communications in Mathematical Physics},
 volume={389},
 issue={1},
 pages={1 -- 30},
 year={2022}
}

@article{ganham24,
  title={Stability and chaos in dynamical last passage percolation},
  author={Shirshendu Ganguly and Alan Hammond},
  journal={Communications of the American Mathematical Society},
  volume={4},
  pages={387 -- 479},
  year={2024}
}

@misc{ivazha24,
      title={On the Eldan-Gross inequality},
      author={Paata Ivanisvili and Haonan Zhang},
      year={2024},
      howpublished={Preprint, {arXiv}:2407.17864 [math.{PR}] (2024)}
}

@article{baideijoh99,
 author={Baik, Jinho and Deift, Percy and Johansson, Kurt},
 title={On the distribution of the length of the longest increasing subsequence of random permutations},
 journal={Journal of the American Mathematical Society},
 volume={12},
 number={4},
 pages={1119 -- 1178},
 year={1999}
}

@book{led01,
 author={Ledoux, Michel},
 title={The concentration of measure pheonomenon},
 series={Mathematical Surveys and Monographs},
 volume={89},
publisher = {Providence, RI: American Mathematical Society (AMS)},
 year={2001}
}

@book{lawlim10,
 author = {Lawler, Gregory F. and Limic, Vlada},
 title = {Random walk: {A} modern introduction},
 series = {Cambridge Studies in Advanced Mathematics},
 volume = {123},
 year = {2010},
 publisher = {Cambridge: Cambridge University Press}
}

@article{balcatsep06,
 author={Bal{\'a}zs, M{\'a}rton and Cator, Eric and Sepp{\"a}l{\"a}inen, Timo},
 title={Cube Root Fluctuations for the Corner Growth Model Associated to the Exclusion Process},
 journal={Electronic Journal of Probability},
 volume={11},
 pages={1094 -- 1132},
 year={2006}
}

@book{abcfgmrs00,
 author = {An{\'e}, C{\'e}cile and Blach{\`e}re, S{\'e}bastien and Chafaï, Djalil and Foug{\`e}res, Pierre and Gentil, Ivan and Malrieu, Florent and Roberto, Cyril and Scheffer, Gr{\'e}gory},
 title = {Sur les in{\'e}galit{\'e}s de {Sobolev} logarithmiques},
 series = {Panoramas et Synth{\`e}ses},
 volume = {10},
 year = {2000},
 publisher = {Paris: Soci{\'e}t{\'e} Math{\'e}matique de France},
 language = {French}
}

@article{smi01,
author = {Stanislav Smirnov},
title = {Critical percolation in the plane: conformal invariance, Cardy's formula, scaling limits},
journal = {Comptes Rendus de l'Académie des Sciences - Series I - Mathematics},
volume = {333},
number = {3},
pages = {239-244},
year = {2001}
}

@incollection {tsi04,
    AUTHOR = {Tsirelson, Boris},
     TITLE = {Scaling limit, noise, stability},
 BOOKTITLE = {Lectures on probability theory and statistics},
    SERIES = {Lecture Notes in Math.},
    VOLUME = {1840},
     PAGES = {1--106},
 PUBLISHER = {Springer, Berlin},
      YEAR = {2004},
}

@incollection {sep18,
    AUTHOR = {Sepp\"{a}l\"{a}inen, Timo},
     TITLE = {The corner growth model with exponential weights},
 BOOKTITLE = {Random growth models},
    SERIES = {Proc. Sympos. Appl. Math.},
    VOLUME = {75},
     PAGES = {133--201},
 PUBLISHER = {Amer. Math. Soc., Providence, RI},
      YEAR = {2018},
}

\end{document}